\documentclass[10pt]{amsart}
\usepackage{amsmath}
\usepackage{amssymb, amsfonts, amsmath, amsthm, amscd}
\usepackage{mathrsfs}
\usepackage{color, graphics}
\usepackage[all]{xy}
\usepackage{hyperref}
\usepackage{amsxtra} 

\newtheorem{theorem}{Theorem}[section]
\newtheorem{proposition}[theorem]{Proposition}
\newtheorem{lemma}[theorem]{Lemma}
\newtheorem{claim}[theorem]{Claim}
\newtheorem{corollary}[theorem]{Corollary}

\newtheorem{remark}[theorem]{Remark}

\newtheorem{definition}[theorem]{Definition}
\theoremstyle{remark}

\numberwithin{equation}{section}

\def\deg{\operatorname{deg}}%
\def\dim{\operatorname{dim}}%
\def\max{\operatorname{max}}%
\def\codim{\operatorname{codim}}%
\def\corank{\operatorname{corank}}%
\def\ker{\operatorname{ker}}%

\def\pic{\hbox{\rm Pic}}

\def\codim{\operatorname{codim}}%
\def\deg{\operatorname{deg}}%
\def\dim{\operatorname{dim}}%
\def\max{\operatorname{max}}%
\def\codim{\operatorname{codim}}%
\def\ker{\operatorname{ker}}%

\DeclareMathOperator{\cork}{{\rm cork}}
\newcommand{\Jj}{\mathcal{J}}
\newcommand{\Ii}{\mathcal{I}}

\newcommand \im   {\ensuremath{\mathrm{im}}}

\newcommand \ext {\ensuremath{\mathrm{Ext}}}
\newcommand \Sec {\ensuremath{\mathrm{Sec}}}

\newcommand \coker {\ensuremath{\mathrm{coker}}}

\newcommand \lra {\rightarrow}

\def\deg{\mbox{deg}}

\def\Oc{\mathcal O}

\def\N{\mathcal N}

\def\C{\mathbf C}

\def\Ff{\mathcal F}

\def\Pp{\mathbb P}

\begin{document}

\title[Complete classification of  irreducible components of the Brill-Noether locus]{Complete classification of  irreducible components of the Brill-Noether locus of rank-$2$ vector bundles of degree $d$ and speciality 2 on a general $\nu$-gonal curve}
\author{Youngook Choi}
\address{Department of Mathematics Education, Yeungnam University, 280 Daehak-Ro, Gyeongsan, Gyeongbuk 38541,
 Republic of Korea }
\email{ychoi824@yu.ac.kr}
\author{Flaminio Flamini}
\address{Universita' degli Studi di Roma Tor Vergata, 
Dipartimento di Matematica, Via della Ricerca Scientifica-00133 Roma, Italy}
\email{flamini@mat.uniroma2.it}
\author{Seonja Kim}
\address{Department of  Electronic Engineering,
Chungwoon University, Sukgol-ro, Nam-gu, Incheon, 22100, Republic of Korea}
\email{sjkim@chungwoon.ac.kr}
\thanks{The first author was supported by the National Research Foundation of Korea(NRF) grant funded by the Korea government(MSIT) (RS-2024-00352592). The third  author  was supported by the National Research Foundation of Korea(NRF) grant funded by the Korea government(MSIT) (2022R1A2C100597713). First and third authors also benefitted of support by research funds ``AAAGH-CUP E83C22001710005" (Responsible Prof. G. Pareschi) of the University of Rome Tor Vergata, during their visit at the Department on Mathematics in November 2023, when collaboration began. The second author acknowledges the support of the MIUR Excellence Department Project MatMod@TOV, MIUR CUP-E83C23000330006, 2023-2027,  awarded to the Department of Mathematics, University of Rome Tor Vergata, and the support from the Organizers of the conference ``Workshop in Classical Algebraic Geometry", Daejeon-South Korea, 12-15 December 2023 (Prof. D. S. Hwang, J.-M. Hwang, Y. Lee) as an invited speaker, occasion in which this collaboration has been further developed; moreover he is member of  INdAM-GNSAGA}

\subjclass[2020]{14H60, 14H10, 14D20, 14E08, 14J26, 14E05}

\keywords{general $\nu$-gonal curves, stable rank-two vector bundles, Brill-Noether loci, birational geometry}
\begin{abstract} In this paper, we address the {\em Brill–Noether theory} for rank-two, degree $d$, stable vector bundles of speciality $2$ on a general $\nu$-gonal curve $C$ of genus $g$, leveraging {\em universal extension spaces}, {\em modular maps}, and recent developments in rank-one {\em Brill-Noether theory over Hurwitz spaces} on $C$.

We  completely classify the irreducible components of such Brill-Noether loci in the whole range of interest for $d$, namely $2g-2 \leq d \leq 4g-4$. Using specialization techniques, we further uncover a {\em stratification} into locally closed subsets within some of these components, and we also provide additional insight into the {\em birational geometry} and the {\em local structure} of every such a component.

Our methods yield thorough descriptions of the irreducible components of any such Brill-Noether locus and, as a by-product of our more general results, also derive interesting consequences for Brill-Noether loci of stable, rank-two bundles with a fixed {\em general determinant}, rather than fixed degree.

\end{abstract}
\maketitle
\section*{Introduction} 

Brill-Noether theory of line bundles on smooth, complex, projective curves $C$ of genus $g$ is a fundamental topic in Algebraic Geometry, whose study goes back (at least) to Riemann \cite{R}, and which addresses the problem of determining which {\em invariants} are needed to understand the behavior of maps from a smooth, projective curve $C$ of genus $g$ to a suitable projective space. In other words, it focuses on the geometry of the {\em space of maps} $C \to \mathbb P^r$, with $r>0$ and of given positive degree $d$. 

The {\em discrete invariants} $(g,r,d)$ involved are equivalent to the datum of line bundles $L$ on $C$ of genus $g$, which are of degree $d$ and which are equipped with $(r + 1)$-dimensional base-point-free vector spaces of global sections. Thus investigation on the geometry of such maps is related to the study of (locally determinantal) subschemes of ${\rm Pic}^d(C)$, called {\em Brill-Noether loci} $W_d^r(C)$, set-theoretically defined as:
$$W_d^r(C) := \left\{ L \in {\rm Pic}^d(C) \; | \; h^0(C, L) \geq r+1\right\} \subseteq {\rm Pic}^d(C).$$

When $C$ is a {\em general} curve of genus-$g$, equivalently when $C$ is with {\em general moduli} (this means that $C$ corresponds to a general point of the {\em moduli space} $\mathcal M_g$ parametrizing isomorphism classes of such curves), Brill-Noether loci are well-understood by main theorems from the 1970-1980's (due to several authors as D. Eisenbud-J. Harris,\; W. Fulton-R. Lazarsfeld, D. Gieseker, P. Griffiths-J. Harris, G. Kempf, S. Kleiman-D. Laksov, K. Petri, just to mention a few) which give good descriptions of the geometry of such $W^r_d (C)$'s, cf. e.g. \cite{ACGH}. Some properties, e.g. non-emptyness and smoothness, are of local nature  whereas some others, e.g. connectedness or irreducibility, require a global understanding of such loci; the core of the aforementioned results establishes that properties of Brill-Noether loci are quite uniform when $C$ is  with {\em general moduli}, being basically determined by the aforementioned discrete invariants $(g, r,d)$ (cf. e.g. \cite{ACGH} for full details).

Nevertheless some curves $C$ can be already equipped with a map $C \stackrel{f}{\longrightarrow} \mathbb P^k$ of some given degree (which could impose to $C$ not being with general moduli, in which case $C$ will be called  a curve with {\em special moduli}). It is therefore natural to ask how the presence of such a map $f$ affects the properties of Brill-Noether loci $W^r_d(C)$ to be studied. The basic and most natural case occurs when $k=1$ and $\deg(f) = \nu$ with $\nu < \lfloor \frac{g+3}{2}\rfloor$ (the integer $\lfloor \frac{g+3}{2}\rfloor$ being the so called {\em general gonality} for genus-$g$ curves, i.e. the gonality of a curve of genus $g$ with general moduli), and when $C$ is a {\em general curve} among those genus-$g$ curves possessing such a map $f$ onto $\mathbb P^1$, of degree $\nu < \lfloor \frac{g+3}{2}\rfloor$ strictly less than the general gonality. 

The term ``general curve" in this context can be made more precise by considering the so called {\em Hurwitz scheme} $\mathcal H_{g,\nu}$, which is an irreducible scheme parametrizing branched covers $ C \stackrel{f}{\longrightarrow} \mathbb P^1$ of degree $\nu$ and genus $g$, which admits a natural {\em modular map} $$\mathcal H_{g,\nu} \stackrel{\Phi}{\longrightarrow}\mathcal M_g,$$given by forgetting the maps $f$. From the above overview, when $\nu \geq \lfloor \frac{g+3}{2}\rfloor$, i.e. the general gonality, then the map $\Phi$ is dominant onto $\mathcal M_g$; this is the reason why one restricts to $\nu < \lfloor \frac{g+3}{2}\rfloor$ and, in such a case, one refers to $C$ to be a {\em general $\nu$-gonal} curve of genus $g$  when $C$ corresponds to a general point of ${\rm Im}(\Phi)$.

In this set-up, one can similarly ask about properties of Brill-Noether loci  $W^r_d(C)$'s; in such a case, the uniform picture for Brill-Noether theory for general curves completely breaks down as it is clear even from very first examples. For $\nu = 2$ and $3$ there are classical results, the first case dealing with {\em hyperelliptic curves}, due to Clifford \cite{Cli} in 1878, whereas the second dealing with {\em trigonal curves}, which have been studied by Maroni \cite{Maro} in 1946.

For arbitrary special gonality $\nu$, some preliminary results have been given by e.g. \cite{BK,CKM,CoM,KK,K,Mart}, just to mention a few and, in more recent years, there has been a great deal of interest in this topic which produced many breakthrough results (cf.\;e.g.\,\cite{CJ,JP,JR,LLV,Lar,P}) showing that, differently from the case with general moduli, discrete invariants $(g, r,d)$ are not enough to study  Brill-Noether theory on a general $\nu$-gonal curve. In \S\,\ref{ss:BNlbgon} we will remind the most general results in this direction, following in particular terminology and results by H. Larson as in \cite{Lar}, which will be sometimes used later on. 

The present paper is focused on Brill-Noether theory of rank-$2$, stable vector bundles of degree $d$ on a general $\nu$-gonal curve of genus $g$.

Recall that, for any smooth, irreducible projective curve $C$ of genus $g$, one denotes by $U_C(d)$ the {\em moduli space} of (semi)stable, degree-$d$, rank-$2$ vector bundles on $C$ and by $U^s_C(d)$ the open, dense subset filled-up by stable bundles. Similarly as in the line-bundle case, one can define a (locally determinantal) subscheme $B^k_{2,d} \subseteq U_C(d)$ called as above {\em Brill-Noether locus}, which parametrizes {\em $S$-equivalence} classes of rank-$2$, degree $d$, semistable vector bundles $\mathcal F$  having $h^0(C, \mathcal F)\ge k$, for a given positive integer $k$. In rank one, traditionally the Brill-Noether locus  $B^{k+1}_{1,d}$  is denoted as in the above lines by $W^k_d$.

The schemes $B^k_{2,d}$, from now on simply denoted by $B^k_d := B^k_{2,d}$, being natural generalizations of the Brill--Noether loci for line bundles, have received considerable attention from various authors. Beyond their intrinsic interest, these loci play a central role in several applications to other areas: birational geometry of several moduli spaces (cf.\,e.g.\,\cite{Beau,Be,BeFe,HR,Muk2,Muk}), Hilbert schemes of curves and surface scrolls in projective spaces (cf.\,e.g.,\,\cite{CCFM,CFK2,CIK,CIK1,CIK2,CF2,FS}, for the use of stable and unstable special bundles, in particular for explicit constructions of {\em non-reduced} components of such Hilbert schemes), just to mention a few.

In contrast with the extensive results available for the rank-$1$ case, several fundamental questions about rank-$2$ Brill--Noether loci—such as non-emptiness, dimension, irreducibility, local structure, and more—remain open in general. Moreover, unlike the case for line bundles, Brill--Noether loci in rank $2$ do not always behave as expected, even when the curve $C$ is general in moduli (cf. e.g.\;\cite{BeFe} and Remark \ref{rem:BNloci}-(b) below). 

From standard facts concerning Serre duality and (semi)stability, it turns out that in the rank-$2$ case it makes sense to consider non-trivial Brill-Noether loci $B_d^k$ when $$2g-2 \leq d \leq 4g-4\;\;\; {\rm and}\;\;\; k = k_i := d-2g+2+i,$$for $i \geq 1$ an integer (cf. Proposition \ref{prop:sstabh1}, formula \eqref{eq:congd} and Definition \ref{def:BNloci} below). In this case, it is well-known that the {\em expected dimension} of $B_d^{k_i} \cap U^s_C(2,d)$ is $$\rho_d^{k_i}:=4g-3-ik_i,$$which is the so-called rank-$2$ {\em Brill-Noether number} (cf.\;\cite{Sun} and \eqref{eq:bn}, Remark \ref{rem:BNloci} below). Contrary to the line bundle case, the integer ${\rm max}\{-1, \; \rho_d^{k_i}\}$ is no longer the {\em expected dimension} for possible components of $B^{k_i}_d$ contained in the strictly-semistable locus $U_C(d) \setminus U^s_C(d)$ (i.e. components containing no stable points), which case can occur only if $d$ is even (cf. Remark \ref{rem:BNloci} below for more details).

With this reminded set-up, an irreducible component $\mathcal B$ of $B_d^{k_i} \cap U_C^s(d)$ is said to be {\em regular}, if $\mathcal B$ is generically smooth and of the expected dimension $\rho_d^{k_i}$, {\em superabundant}, otherwise (cf. Definition \ref{def:regsup} below). 

\medskip

When the speciality is $i=1$, the Brill-Noether locus  $B^{k_1}_d\cap U_C^s(d)$  has been deeply studied on any smooth, irreducible, projective curve $C$:  

\begin{theorem}\label{SunLau} (cf.\,e.g. \cite{L,Sun,Teixidor1}) Let $C$ be a non-singular, irreducible, projective curve of genus $g>0$. Then, for any integer $d$, one has that $B^{k_1}_d \cap U^s_C(d)$ is either empty or it is irreducible. 

When not empty, a general point $[\mathcal F] \in B^{k_1}_d \cap U^s_C(d)$ corresponds to a rank-$2$ stable vector bundle $\Ff$ of degree $d$ on $C$, whose space of sections $H^0(C, \mathcal F)$ has dimension $k_1= d - 2g +3$. When moreover $C$ is with general moduli, then $B^{k_1}_d \cap U^s_C(d)$ is also {\em regular}. 
\end{theorem}

\medskip

For speciality $i=2$,  using degeneration arguments, N. Sundaram \cite{Sun} proved  that $B^{k_2}_d\cap U_C^s(d)$ is not empty for any  smooth, irreducible projective curve $C$ and for odd degree $2g-1 \le d\le 3g-4$. Later on M. Teixidor I Bigas generalized Sundaram's result  in \cite{Teixidor1, Teixidor2} as follows:

\begin{theorem}\label{TeixidorRes} ( {\em ``Serre-dual" version} of \cite[Theorem]{Teixidor1}) Let $C$ be a non-singular, irreducible, projective curve of genus $g \geq 3$. 
If $d$ is an integer such that $2g-3\le d\le 4g-7$,  then for any $C$ one has that $B^{k_2}_d \cap U^s_C(d)$ contains an irreducible component $\mathcal B$ of (expected)  dimension $\rho^{k_2}_d=8g-2d-11$ and a general point $[\mathcal F]$ on it corresponds to a rank-$2$ stable vector bundle $\Ff$ of degree $d$ on $C$, whose space of sections $H^0(C, \mathcal F)$ has dimension $k_2 = d - 2g +4$. 

If moreover $C$ is with {\em general moduli}, this is the only component of $B^{k_2}_d \cap U^s_C(d)$, which is {\em regular}. Moreover, $B^{k_2}_d \cap U^s_C(d)$ has extra components if and only if $W^1_n(C)$ is non-empty with $\dim (W^1_n(C)) \ge 2g + 2n - d - 5$,  for some integer $n$ such that $2n<4g-4-d$.
\end{theorem} 

\noindent
To be more precise, the bundle $\mathcal F$ in the statement of Theorem \ref{TeixidorRes} above is $\mathcal F := \omega_C \otimes \mathcal E^{\vee}$, where $\mathcal E$ stands for the bundle as in the proof of Theorem \ref{TeixidorRes} in \cite{Teixidor1}, thus $h^j(C, \mathcal F) = h^{1-j}(C, \mathcal E)$, $0 \leq j \leq 1$. 

\bigskip

Inspired by the previous results, which provide more comprehensive descriptions 
basically when $C$ is a curve with {\em general moduli}, the aim of this paper is to deal with the recent aforementioned breakthrough results in Brill-Noether theory of line bundles on a general $\nu$-gonal curve $C$ as in \cite{CJ,JP,JR,LLV,Lar,P} (in particular we follow terminology and results by H. Larson as in \cite{Lar}), in order to study and classify all the irreducible components of $B_d^{k_2} \cap U^s_C(d)$ when $C$ is a general $\nu$-gonal curve of genus $g$, with $3 \leq \nu < \lfloor \frac{g+3}{2}\rfloor$.

 We extend some preliminary results in \cite{CFK,CFK2} by giving a thorough description of 
all the irreducible components of  $B_d^{k_2} \cap U^s_C(d)$, including information on the {\em birational structure} of these irreducible components, on their local behavior (namely regarding {\em smoothness} of their points), on their dimension (in some cases distinguishing between a {\em superabundant} and a {\em regular} component), and finally providing an explicit {\em presentation} of the rank-$2$ vector bundle $\Ff$ on $C$ corresponding to the general point $[\Ff]$ of each constructed component.

As the approach to prove our classification result in this paper is rather extensive and stratified, we have undertaken the study of the speciality $i=3$ case in a separate forthcoming paper (cf. \cite{CFK3}).

Our investigation hinges on determining suitable  \emph{``geometric descriptions''} of the (isomorphism classes of) bundles $\Ff$ corresponding to general points of possible irreducible components of $B^{k_2}_d\cap U^s_C(d)$ on $C$. By ``geometric descriptions'', we refer to a characterization of any such a $\Ff$ through properties of suitable families of curves on the associated ruled surface (or {\em surface scroll}) $F:= \Pp(\Ff)$; this, in turn, translates into expressing $\Ff$ as an extension of line bundles: $$(*)\;\;\; 0 \to N \to \Ff \to L \to 0,$$which we term a {\em presentation} of $\Ff$, where $L$ and $N$ are suitable line bundles on $C$. In order to construct irreducible components of $B_d^{k_2}\cap U^s_C(d)$,  a presentation $(*)$ of $\Ff$ featuring suitable {\em minimality properties} for the quotient line bundle $L$ is particularly relevant. This explains the reason why rank-$1$ Brill-Noether theory on a general $\nu$-gonal curve, as in \cite{CJ,JP,JR,LLV,Lar,P}, deeply enter into the game.

Representing rank-$2$ bundles $\Ff$ as extensions $(*)$ is a classical approach, tracing back at least to C. Segre \cite{Seg}. More recently, this method has been effectively used by A. Bertram and B. Feinberg in \cite[\S\,2,3]{BeFe}, by S. Mukai in \cite[\S\,8]{Muk}, and by \cite{CF} to study Brill-Noether loci in rank two, primarily when $C$ has {\em general moduli}. As explicitly noted in \cite{BeFe}, while this approach performs well in low genera, its implementation becomes more challenging in general.

Nevertheless, in this paper, we pursue this route without imposing upper bounds on the genus $g$ for a general $\nu$-gonal curve $C$. We moreover let the (kernel and cokernel) line bundles $N$ and $L$ as in $(*)$ vary in their own rank-$1$ Brill-Noether loci in all possible degrees and, correspondingly, we construct bundles $\Ff$ as extensions $(*)$ by letting them vary in suitable degeneracy loci $$\mathcal W \subseteq {\rm Ext}^1(L,N),$$  defined in such a way that $\Ff \in \mathcal W$ general has the desired speciality  $i=h^1(C, \Ff)=2$,  according to those of $N$ and $L$ and to the corank of the associated {\em coboundary map} $H^0(C, L) \stackrel{\partial}{\longrightarrow} H^1(C,N)$. 

For any chosen degree pair $(\deg(L), \deg(N)) = (\delta, d-\delta)$, we obtain in this way an irreducible {\em universal extension space} parametrizing triples $(N,L, \Ff)$ and, once the general bundle $\Ff$ is shown to be {\em stable}, any such universal extension space is therefore endowed with a (rational) {\em modular map}$$\left\{(N,L,\Ff)\right\} \stackrel{\pi}{\dasharrow} U^s_C(d),$$whose image will be contained in an irreducible component $\mathcal B$ of a suitable Brill-Noether locus, where the Brill--Noether locus we hit depends therefore on the cohomology of the chosen $N$ and $L$ and on the corank of the coboundary map $H^0(C, L) \stackrel{\partial}{\longrightarrow} H^1(C,N)$. To find out a {\em minimal presentation} $(*)$ for a bundle $\Ff$ corresponding to a general point $[\Ff]$ in ${\rm im}(\pi) \subseteq \mathcal B$ will turn out to be equivalent to finding conditions on $L$, $N$ and $\partial$ ensuring the aforementioned modular map $\pi$ to be dominant onto the component $\mathcal B$, with an appropriate dimension for its general fiber. 

Our analysis proceeds in two stages: first, we construct suitable irreducible loci within the moduli space. If such loci prove to constitute an irreducible component of $B^{k_3}_d \cap U^s_C(d)$, we then classify and study the geometry of such constructed components, hinging on the {\em expected} presentation $(*)$ of the bundle $\Ff$ representing a general point of such a component. This in turn guides the explicit construction of the proposed components, which is then followed by a detailed investigation of their global and local geometric properties. 

\bigskip

Our main result in this paper is as follows:

\bigskip

\begin{theorem}\label{thm:i=2} Let $g \geq 4$, $3 \leq \nu < \lfloor \frac{g+3}{2}\rfloor$ and $2g-2 \leq d \leq 4g-4$ be integers and let $C$ be a general $\nu$-gonal curve of genus $g$.  Then, for $B^{k_2}_d \cap U^s_C(d)$, one has the following cases:

\begin{enumerate}
\item[$(i)$] for $4g-6\leq d \leq 4g-4$, $B^{k_2}_d \cap U^s_C(d)=\emptyset$; 

\item[$(ii)$] for $4g-4-2\nu \leq d \leq 4g-7$, $B^{k_2}_d \cap U^s_C(d)$ is {\em irreducible}, consisting of only one irreducible component, denoted by  $\mathcal B_{\rm reg,2}$, which is {\em regular} and {\em uniruled}. Moreover, if $[\Ff] \in \mathcal B_{\rm reg,2}$ denotes a general point then the corresponding bundle $\Ff$ is stable, with speciality $h^1(\Ff) = 2$ and it arises as an extension of the following type: 

\smallskip

\noindent
$(ii-1)$ for any $4g-4-2\nu \leq d \leq 4g-7$, when   $3 \leq \nu \leq \frac{g}{2}$, or for any 
$3g-4 \leq d \leq 4g-7$,  when $\frac{g+1}{2}\leq \nu < \lfloor\frac{g+3}{2} \rfloor$, $\Ff$ fits in an exact sequence of the form: 
$$0 \to  \omega_{C}(-D) \to \Ff \to \omega_C (-p) \to 0,$$where 
$D \in C^{(4g-5-d)}$  and $p \in C$ are general; 

\smallskip

\noindent
$(ii-2)$ for any $4g-4-2\nu \leq d \leq 3g-5$, when $\frac{g+1}{2}\leq \nu < \lfloor\frac{g+3}{2} \rfloor$, $\Ff$ fits in an exact sequence of the form: 
$$0 \to N \to \Ff \to \omega_C (-p) \to 0,$$where $p \in C$ general, $N \in {\rm Pic}^{d-2g+3}(C)$ general, i.e. special and non-effective. 
\smallskip

\noindent
$(ii-3)$ In any of the above cases, $\omega_C(-p)$ is a line bundle of minimal degree among all effective and special quotient line bundles of $\Ff$;

\item[$(iii)$] for $3g-5 \leq d \leq 4g-5-2\nu$, which occur for $3 \leq \nu \leq \frac{g-2}{2}$ (otherwise we are in case $(ii-2)$), $B^{k_2}_d \cap U^s_C(d)$ is {\em reducible}, consisting of two irreducible components, $\mathcal B_{\rm reg,2}$ and $\mathcal B_{\rm sup,2}$.

\smallskip

\noindent
$(iii-1)$ The component $\mathcal B_{\rm reg,2}$ is {\em regular}  and {\em uniruled}.  A general element $[\mathcal F] \in \mathcal B_{\rm reg,2}$ corresponds to a bundle  $\Ff$ which is stable, of speciality $h^1(\Ff) = 2$ and fitting in an exact sequence as in case $(ii-1)$ above, where $p\in C$ and $D\in C^{(4g-5-d)}$ are general. Moreover, $ \omega_C(-p)$  is of minimal degree among all effective and special quotient line bundles of $\mathcal F$.

\smallskip

\noindent
$(iii-2)$ The component $\mathcal B_{\rm sup,2}$ is {\em generically smooth}, of dimension $6g-6- d- 2\nu > \rho_d^{k_2}$, i.e. it is 
{\em superabundant}. Moreover, it is  {\em ruled} and its general point $[\Ff]$ corresponds to a bundle $\Ff$ which is stable, of speciality $h^1(\Ff) = 2$, fitting   in an exact sequence of the form:
$$ 0\to N \to \mathcal F\to   \omega_C \otimes A^{\vee} \to 0,$$where $A \in {\rm Pic}^{\nu}(C)$ is the unique line bundle on $C$ such that $|A| = g^1_{\nu}$ whereas $N \in {\rm Pic}^{d-2g+2+\nu}(C)$ is general of its degree. Moreover, $\omega_C \otimes A^{\vee}$ is of {\em minimal degree} among all quotient line bundles of $\Ff$.

\item[$(iv)$] For $2g-4+2\nu \leq d \leq 3g-6$, which occurs for $3 \leq \nu \leq \frac{g-2}{2}$ (otherwise  $2g-4+2\nu > 3g-6$  and we are in previous cases),  $B^{k_2}_d \cap U^s_C(d)$ is {\em reducible}, consisting of two components, $\mathcal B_{\rm reg,2}$ and $\mathcal B_{\rm sup,2}$. The description of $\mathcal B_{\rm sup,2}$ is as in $(iii)$ whereas description of its general point is identical  to that described in case $(iii-2)$ above. The description of $\mathcal B_{\rm reg,2}$ is as in $(iii)$ whereas the description  of its general point is identical  to that described in case $(ii-2)$.

\item[$(v)$] For $d = 2g-5+2\nu$ and  for any $3 \leq \nu \leq \frac{g-2}{2}$ (otherwise $2g-5+2\nu \geq 3g-4$ and we are in previous cases), $B^{k_2}_d \cap U^s_C(d)$ is {\em reducible}, consisting of two components, $\mathcal B_{\rm reg,2}$ and $\mathcal B_{\rm sup,2}$. In this case the two components are both {\em regular}. The rest of the description is identical to that described in case $(iv)$ (except for the {\em superabundance} of $\mathcal B_{\rm sup,2}$, which is no-longer true).

\item[$(vi)$] For $2g-2 \leq d \leq 2g-6 + 2\nu$, for any $3 \leq \nu < \lfloor\frac{g+3}{2}\rfloor$, $B_{d}^{k_2}\cap U_C^s(d)$ is {\em irreducible}, consisting of only one irreducible component $\mathcal B_{\rm reg,2}$, which is {\em regular} and {\em uniruled}. The descriptions of such a component and of its general point are identical  to those described in case $(iv)$ above. 

\end{enumerate}
\end{theorem}

\begin{proof} The proof is given in full details in Sect.\;\ref{S:BN2}; more precisely it will be given by the collection of Propositions \ref{prop:caso1spec2}, \ref{prop:caso2spec2}, \ref{prop:caso3spec2},  \ref{thm:i=1.3}, \ref{prop:caso1sup2}, \ref{lem:i=2.1}, \ref{thm:noothercompsspec2}, Corollaries \ref{cor:aiuto12feb1634}, \ref{cor:finale2} and Remark \ref{rem:finale2}. 
\end{proof}

\bigskip

\begin{remark}\label{rem:i=2}  {\normalfont (i) Notice that Theorem \ref{thm:i=2} above strongly improves \cite[Theorem\,3.1]{CFK} and \cite[Theorem\,1.5]{CFK2}. Indeed 
for the whole range of interest $2g-2 \leq d \leq 4g-4$, Theorem \ref{thm:i=2} classifies all the irreducible components of $B_d^{k_2} \cap U_C^s(d)$ and also precisely describes the geometric behavior of all such components, giving in particular important information on their birational structure as well as on the explicit presentation of the general point of any such a component.

\smallskip

\noindent
(ii) Theorem \ref{thm:i=2} also corrects \cite[Theorem\,3.12-(ii)]{CFK2} where, for $d = 3g-4,\,3g-5$, it was wrongly stated that the component $\mathcal B_{\rm sup,2}$ is non-reduced (cf. also Remark \ref{rem:erroreColl}). This was a consequence of tricky computations on kernels of suitable Petri/multiplication maps. Correct computations are contained in the proof of Proposition \ref{lem:i=2.1} below, which focuses on a wider range of new cases (cf. also Remark \ref{rem:corrColl}).

\smallskip

\noindent
(iii) Theorem \ref{thm:i=2} exhibits also the irreducibility of $B^{k_2}_d\cap U_C^s(d)$ as postulated in Theorem \ref{TeixidorRes}-(i).  Indeed, apart from $4g-6 \leq d \leq 4g-4$ where $B^{k_2}_d\cap U_C^s(d) = \emptyset$, it turns out that $B^{k_2}_d\cap U_C^s(d)$ is irreducible and regular, for any $2g-2 \leq d \leq 2g-6+2\nu$ and $4g-4-2\nu \leq d \leq 4g-7$. These are exactly the ranges of $d$ for which no integer $n$ (as required in Theorem \ref{TeixidorRes}-(i) guaranteeing reducibility) can exist;  cf. the proof of Proposition \ref{prop:caso2spec2} and see Remark \ref{rem:finale2} for explicit computations.  

\smallskip

\noindent
(iv) Finally, Theorem \ref{thm:i=2} also shows that Teixidor I Bigas' component $\mathcal B$ as in Theorem \ref{TeixidorRes}-(i) (existing on any smooth, irreducible projective curve $C$) coincides with our component $\mathcal B_{\rm reg,2}$ (cf. \S\,\ref{ss:reg2}), for which we moreover prove its {\em regularity} and {\em uniruledness} when $C$ is a general $\nu$-gonal curve (but its simple construction still holds for any smooth, projective curve $C$). We furthermore deduce suitable {\em minimal presentation} as line-bundle extension for the bundle corresponding to a general point of such a component.}
\end{remark}

\medskip

As a direct consequence of Theorem \ref{thm:i=2} one has also further information on some other {\em Brill-Noether loci}, e.g. those   parametrizing (isomorphism classes of) stable vector bundles with fixed determinant, instead of fixed degree.  Indeed, in general one has the so called {\em determinantal map}, simply defined as: 
$$U_C(d) \stackrel{\det}{\longrightarrow} {\rm Pic}^d(C), \;\;\; [\Ff] \stackrel{det}{\longrightarrow} \det(\Ff),$$which turns out to be surjective onto $ {\rm Pic}^d(C)$. For any given $M \in {\rm Pic}^d(C)$ one can therefore consider 
$$U_C(M) := (\det)^{-1} (M) = \{[\Ff] \in U_C(d)\,|\, \det(\Ff) = M\},$$which is well-known to be irreducible and of dimension $3g-3$, as well as for any $i\geq 1$ the {\em Brill-Noether loci with fixed determinant} namely 
$$B_M^{k_2}:= B_d^{k_2} \cap U_C(M),$$whose {\em expected dimension} is $\rho_M^{k_2} := 3g-3 - 2\,k_2$, cf. \cite{LNS}. 

As a direct consequence of Theorem \ref{thm:i=2} above, one has also:

\medskip 

\begin{corollary}\label{cor:i=2} With assumptions as in Theorem \ref{thm:i=2}, let $M \in {\rm Pic}^d(C)$ be general. Consider the moduli space 
$$U_C(M)= (\det)^{-1} (M) = \{[\Ff] \in U_C(d)\,|\, \det(\Ff) = M\},$$which is irreducible and of dimension $3g-3$, and in it consider the {\em Brill-Noether locus} $$B_M^{k_2}= B_d^{k_2} \cap U_C(M),$$of {\em expected dimension} $\rho_M^{k_2} = 3g-3 - 2\,k_2 = 7g-11-2d$. 
Then one has the following situation. 

\smallskip

\noindent
$(a)$ For any $2g-5+2\nu \leq d \leq 4g-5 - 2\nu$, $B_M^{k_2}\cap U_C^s(M)$ is not empty even when its expected dimension $\rho_M^{k_2} = 7g-11-2d$ is negative which, under condition $d\leq  4g-5 - 2\nu$, occurs for   $d > \frac{7g-11}{2}$ and $\nu<\frac{g+1}{4}$. 

\smallskip

\noindent
$(b)$ Furthermore, for  $d \geq 3g-3$, $B_M^{k_2}\cap U_C^s(M)$ consists only of $\mathcal B_{\rm sup,2} \cap U_C^s(M)$, which is {\em superabundant}, being of dimension $5g-6 - d - 2 \nu > \rho_M^{k_2} $, where $g-1 \leq 5g-6 - d - 2 \nu \leq 
3g-1 - 4 \nu$, and it is moreover birational to $\Pp(\ext^1(K_C-A, N))$, where $M = K_C - A + N \in {\rm Pic}^d(C)$.

\smallskip

\noindent
$(c)$ For $2g-5 + 2 \nu \leq d \leq {3g-4}$, we have the extra component $\mathcal B_{\rm reg,2} \cap U_C^s(M)$, which  is of the {\em expected dimension} $\rho_M^{k_2}=7g-11-2d$.

\smallskip

\noindent
$(d)$ When $d = 2g-5+2\nu$, also the component $\mathcal B_{\rm sup,2} \cap U_C^s(M)$ has {\em expected dimension}  $\rho_M^{k_2}=7g-11-2d$. 

\smallskip

\noindent
$(e)$ At last, for $2g-2 \leq d \le 2g-6+2\nu$, $B_M^{k_2}\cap U_C^s(M)= \mathcal B_{\rm reg,2} \cap U_C^s(M)$ is irreducible and of {\em expected dimension} $\rho_M^{k_2} = 7g-11-2d$.
\end{corollary}

\begin{proof} The proof follows the lines of that in \cite[Corollary\,3.2,\,Remark\,3.3]{CFK2}, for the sub-cases studied therein. One simply takes $\mathcal B$ to be any of the two components of $B_d^{k_2} \cap U^s_C(d)$ as in Theorem \ref{thm:i=2}, namely either $\mathcal B = \mathcal B_{\rm sup,2}$ or $\mathcal B =  \mathcal B_{\rm reg,2}$, and studies the determinantal map restricted to any such a $\mathcal B$.

If, e.g. in $(a)$ and $(b)$ we take $[\mathcal F] \in \mathcal B_{\rm sup,2}$ general as in Theorem \ref{thm:i=2} then, since the corresponding bundle $\Ff$ arises as an element of $\ext^1(K_C-A, N)$, with $N \in {\rm Pic}^{d-2g+2+\nu}(C)$ general, one has that \linebreak $\det(\mathcal F) \cong K_C-A+N$ is general too in ${\rm Pic}^d(C)$. Therefore, the restricted determinantal map $det_{|_{\mathcal B_{\rm sup,2}}}$ is certainly dominant onto ${\rm Pic}^d(C)$ thus, for $M \in {\rm Pic}^d(C)$ general, the Brill-Noether locus $B_M^{k_2}(C)$ contains $\mathcal B_{\rm sup, 2} \cap U_C(2,M)$. Notice further that $$\dim\;(\mathcal B_{\rm sup, 2} \cap U_C(2,M)) = 5g-6-2\nu-d > 0,$$ since $M= K_C - A + N \simeq M':= K_C - A + N'$ if and only if $N \simeq N'$. Therefore, from the construction of $B_{\rm sup, 2}$ conducted in Sect.\;\ref{ss:sup2}, it will be clear that $\mathcal B_{\rm sup,2} \cap U^s_C(M)$ is birational to $\mathbb{P}(\ext^1(K_C-A, N))$, where $M \simeq K_C-A+N \in {\rm Pic}^d(C)$ is general.

One can apply the same reasoning also to the other component $\mathcal B_{\rm reg,2}$ as in Theorem \ref{thm:i=2}. In case $(c)$, $\mathcal B_{\rm reg,2}$ also dominates ${\rm Pic}^d(C)$ via the determinantal map, as $\dim (\mathcal B_{\rm reg,2}) \geq g = \dim({\rm Pic}^d(C))$ and $\det(\Ff) = K_C-p + N = M$ is general by the generality of $N$; this gives rise to a component of the expected dimension. When otherwise $d \geq 3g-3$, then  $\mathcal B_{\rm reg,2}$ cannot dominate ${\rm Pic}^d(C)$ as $\dim\;(\mathcal B_{\rm reg,2}) = 8g-2d-11 < g = \dim \; {\rm Pic}^d(C)$. 
\end{proof}

\medskip

\noindent
{\bf Acknowledgments:} The authors thank the Department of Mathematics of the University of Rome "Tor Vergata" (Italy) and the Complex Geometry Center in Daejeon (South Korea), for the friendly atmosphere and the warm hospitality during the collaboration. They moreover deeply thank Prof. G. Pareschi, responsible of the research funds AAAGH (CUP E83C22001710005) which have supported the staying of the first and third authors at the Department of Mathematics, University of Rome ``Tor Vergata", and Prof. D. S. Hwang, J.-M. Hwang, Y. Lee, organizers of the conference ``Workshop in Classical Algebraic Geometry", Daejeon-South Korea, 12-15 December 2023, for the support to the second author as an invited speaker, occasion in which this collaboration has been further developed. The authors thank KIAS for the warm hospitality when the first and the third authors were associate members in KIAS whereas when the second author was invited guest at KIAS. The first author was supported by the National Research Foundation of Korea(NRF) funded by the Korea government(MSIT) (RS-2024-00352592).  The second author acknowledges the support of the MIUR Excellence Department Project MatMod@TOV, MIUR CUP-E83C23000330006, 2023-2027,  awarded to the Department of Mathematics, University of Rome  "Tor Vergata" (Italy). The third  author  was supported by the National Research Foundation of Korea(NRF) funded by the Korea government(MSIT) (2022R1A2C100597713). Finally, the authors would like to thank Ali Bajravani for his careful reading of the first version of this article, for drawing our attention to reference \cite{Re}  which allowed us to slightly shorten the proof of Proposition \ref{prop:caso1spec2} and for other remarks/open questions  that will be the subject of possible future collaborations.


\section{Preliminaries}\label{S:pre} 
In this paper we work over $\mathbb C$. All schemes will be endowed with the Zariski topology. We will  interchangeably use the terms rank-$r$ vector bundle on a smooth, projective variety $X$ and rank-$r$ locally free sheaf; for this reason we may abuse notation and identify divisor classes with the corresponding line bundles, interchangeably using additive and multiplicative notation. The dual bundle of a rank-$r$ vector bundle $\mathcal E$ will be denoted by $\mathcal E^{\vee}$; if $L$ is of rank $1$, i.e. it is a line bundle, we will therefore  accordingly use $L^{\vee}$ or $-L$. 

Given a smooth, projective variety $X$, we will denote by $\sim$ the linear equivalence of Cartier divisors on $X$ and by $\sim_{alg}$ their algebraic equivalence. If $\mathcal M$ is a moduli space, parametrizing objects modulo a given equivalence relation, and if $Y$ is a representative of an equivalence class in $\mathcal M$,  we denote by $[Y] \in \mathcal M$ the point corresponding to $Y$. 

If $C$ denotes any smooth, irreducible, projective curve of genus $g$, as customary,  $W^r_d(C)$ will denote the  {\em Brill-Noether locus} as in the  Introduction,  parametrizing  line bundles $L \in {\rm Pic}^d(C)$ such that $h^0(L) \geq r+1$. The integer
\begin{equation}\label{eq:bnnumber}
\rho(g,r,d) := g - (r+1) (g+r-d)
\end{equation} denotes the {\em Brill-Noether number} and the multiplication map
\begin{equation}\label{eq:Petrilb}
\mu_L : H^0(C,L) \otimes H^0(\omega_C \otimes L^{\vee}) \to H^0(C, \omega_C)
\end{equation} denotes the so called {\em Petri map} of the line bundle $L$. As for the rest, we will use standard terminology and notation as in e.g. \cite{ACGH}, \cite{H}.

\subsection{Brill-Noether theory of line bundles on general $\nu$-gonal curves}\label{ss:BNlbgon} 
In this section, we briefly review some basic results concerning the Brill--Noether theory of line bundles on a general $\nu$-gonal curve, which will be used in the sequel.

From the Introduction, given a positive integer $g \geq 4$, let $\nu$ be an integer such that
\begin{equation}\label{eq:nu}
3 \leq \nu < \left\lfloor \frac{g+3}{2}\right\rfloor.
\end{equation}
We say that a smooth, irreducible, projective curve $C$ is a \emph{general $\nu$-gonal curve of genus $g$} if $[C]$ is a general point of the \emph{$\nu$-gonal stratum} $\mathcal{M}^1_{g,\nu} \subsetneq \mathcal{M}_g$, where $\mathcal{M}^1_{g,\nu}$ denotes the image of the natural modular map
\[
\Phi: \mathcal{H}_{g,\nu} \longrightarrow \mathcal{M}_g,
\]
with $\mathcal{H}_{g,\nu}$ being the \emph{Hurwitz scheme} of $\nu$-sheeted simply branched covers of $\mathbb{P}^1$ by curves of genus $g$, as already mentioned in Introduction. We begin by recalling the following result:

\begin{proposition}[cf. \cite{Ball}]\label{Ballico} Let $C$ be a general $\nu$-gonal curve of genus $g$, let $r$ be a non-negative integer and let $A \in {\rm Pic}^{\nu}(C)$ be the unique line bundle on $C$ associated to the  unique base-point-free pencil on $C$ of degree $\nu$, i.e. $|A| = g^1_{\nu}$. 

Then, if $g\geq r(\nu-1)$, one has $\dim (|A^{\otimes r}|)=r$ whereas if $g< r(\nu-1)$ then $A^{\otimes r}$ is non-special hence $\dim (|A^{\otimes r}|) =r\,\nu-g>r$.
\end{proposition}

Let therefore $C$ be a general $\nu$-gonal curve of genus $g$, with $f: C\to \mathbb P^1$ denoting the morphism of degree $\nu$, and 
let $L$ be a line bundle of degree $d$ on $C$. Thus $f_*(L)$ is a rank-$\nu$ vector bundle on $\mathbb P^1$. By Grothendieck theorem, there exist uniquely determined integers $e_1 \leq e_2 \leq \ldots \leq e_{\nu}$ for which $$f_*(L)=\mathcal O_{\mathbb P^1}(e_1)\oplus\cdots \oplus \mathcal O_{\mathbb P^1}(e_\nu)$$s.t., by Riemann-Roch theorem and Leray's isomorphism, 
\begin{equation}\label{eq:degf*L}
\deg(f_*(L)) = \sum_{i=1}^{\nu} e_i = d+1-g-\nu.
\end{equation} By projection formula, if we furthermore have $h^0(C,L) = r+1$, then we get 
\begin{equation}\label{eq:rf*L}
r +1 = h^0(C,L)= h^0\left(\Pp^1,  \mathcal O_{\mathbb P^1}(e_1)\oplus\cdots \oplus \mathcal O_{\mathbb P^1}(e_\nu)\right)= 
\sum_{i=1}^{\nu} {\rm max}\,\{e_i+1,\,0\}.   
\end{equation} Following \cite{CJ,Lar}, for any  collection $(e_1,\dots,e_\nu)$ of such integers, one sets $$\vec{e}:=(e_1,\dots,e_\nu),$$which is called the {\em splitting type vector} of the bundle $f_*(L)$, and sets moreover $$\mathcal O_{\mathbb P^1}(\vec{e}): =\mathcal O_{\mathbb P^1}(e_1)\oplus\cdots \oplus\mathcal O_{\mathbb P^1}(e_\nu).$$The splitting type obviously gives rise to a more refined invariant than the rank and the degree of the bundle $f_*(L)$, as it encodes information about all the bundles $f_*\left(L \otimes f^*\left(\mathcal O_{\mathbb P^1}(m)\right)\right)$, for any $m \in \mathbb Z$; indeed, by projection formula we have:  
\[
h^0\left(C, L \otimes f^*\left(\mathcal O_{\mathbb P^1}(m)\right)\right) = h^0(\Pp^1, 
\mathcal O_{\mathbb P^1}(\vec{e}) \otimes 
\mathcal O_{\mathbb P^1}(m)) = \sum_{i=1}^{\nu} {\rm max} \{m+ e_i+1,\,0\}, \; \forall \; m \in \mathbb Z. 
\]For a given splitting type vector $\vec{e}$, one considers the {\em Brill-Noether splitting locus}
\[
W^{\vec{e}} (C) : =\{L\in Pic^d(C)\, |\, f_*(L)=\mathcal O_{\mathbb P^1}(e_1)\oplus\cdots\oplus \mathcal O_{\mathbb P^1}(e_\nu)\},
\] which is locally-closed in ${\rm Pic}^d(C)$ (the above notation is closer to standard Brill-Noether notation and it is taken from \cite{CJ,LLV}, whereas in \cite{Lar} the aforementioned locus is denoted by $\Sigma_{\vec{e}}(C,f)$). The {\em expected codimension} of $W^{\vec{e}} (C) $ in $Pic^d(C) $  is given by 
the {\em magnitude} of $\vec{e}$ (c.f.\,\cite[Def.\,2.3]{CJ}), which is 
\begin{equation}\label{eq:magnitude}
u(\vec{e}) : = h^1(\Pp^1, \mathcal O_{\mathbb P^1}(\vec{e})) = \sum_{1 \leq i < j \leq \nu} {\rm max}\,\{0,\; e_j-e_1-1\}.
\end{equation} Since $\vec{e}$ can be thought as a partition (with possibly negative parts), splitting type vectors follow 
certain {\em ``specialization" rules} as follows: if $\vec{e}$ and $\vec{e}\,'$ are two splitting type vectors, one has a partial ordering $\vec{e}\,' \leq \vec{e}$ if $ e_1'+ e_2' + \ldots + e_j' \leq e_1 + e_2 + \ldots + e_j$, for any 
$1 \leq j \leq \nu$, which in turn implies that $\mathcal O_{\mathbb P^1}(\vec{e})$ 
specializes to $\mathcal O_{\mathbb P^1}(\vec{e}\,')$ (cf. e.g. \cite[\S\,2]{Lar}).  

With this reminded, one defines the {\em Brill-Noether splitting degeneracy loci} 
\[
\overline{W^{\vec{e}}}:= \left\{ L\in Pic^d(C)\, | \, h^0(C, L \otimes f^*(\mathcal O_{\mathbb P^1}(m))) \geq 
h^0(\Pp^1, \; \mathcal O_{\mathbb P^1}(\vec{e}) \otimes \mathcal O_{\mathbb P^1}(m)), \;\; \forall \; m \in \mathbb Z\right\},
\]and gets that $$\overline{W^{\vec{e}}} = \bigcup_{\vec{e}\,' \leq \vec{e}} W^{\vec{e}\,'},$$cf. \cite[(2.1)]{Lar}. A first set of results is given by the following:

\begin{theorem}\label{thm:Lar1} Let $C$ be a general $\nu$-gonal curve of genus $g \geq 4$. Let $\vec{e}$ be a splitting type vector,  satisfying \eqref{eq:degf*L} and let $u(\vec{e})$ be its magnitude as in \eqref{eq:magnitude}. Consider the integer
\begin{equation}\label{eq:expdimLar}
\rho' (g, \vec{e}) := g - u (\vec{e}),
\end{equation} which is the {\em expected dimension} of $W^{\vec{e}}$. Then:

\smallskip

\noindent
(i) $W^{\vec{e}} \neq \emptyset$ if and only if $\rho' (g, \vec{e}) \geq 0$. Moreover, when not empty, 
$W^{\vec{e}}$ is of pure dimension ${\rm min}\{g, \, \rho' (g, \vec{e})\}$ (cf. \cite{CJ}, \cite[Thm.\,1.2]{Lar}); 

\smallskip

\noindent
(ii) When $\rho' (g, \vec{e}) \geq 0$, $W^{\vec{e}}$ is smooth (cf. \cite[Thm.\,1.2]{Lar}). More precisely, when 
$\overline{W^{\vec{e}}} \subsetneq {\rm Pic}^d(C)$, then \linebreak $\overline{W^{\vec{e}}}$ is normal, Cohen-Macaulay and it is smooth 
away from the union of the splitting loci $W^{\vec{e}\,'}$, with $\vec{e}\,' \leq \vec{e}$, having codimension more than $2$ (cf.\cite[Thm.\,1.2-(2')]{LLV}); 

\smallskip

\noindent
(iii) $\overline{W^{\vec{e}}}$ is irreducible when $\rho' (g, \vec{e}) > 0$. When otherwise 
$\rho' (g, \vec{e}) = 0$, the number of points of $\overline{W^{\vec{e}}}$ is computed via Coxeter groups 
(cf. \cite[Thm.\,1.2]{LLV}). 
\end{theorem}

Such Brill-Noether splitting degeneracy loci provide a refinement of the stratification of ${\rm Pic}^d(C)$ by Brill-Noether loci $W_d^r(C)$. Indeed, fixed a rank $r'$ and a degree $d'$ for vector bundles over $\Pp^1$, there exist a unique 
splitting type vector which is {\em maximal} with respect to the previously defined partial order and this is called the {\em balanced splitting type vector}, which is characterized 
by the condition $| e_i - e_j| \leq 1$, for all $1 \leq i,j\leq r'$. The corresponding {\em balanced bundle} is denoted by $B(r',d')$ whereas the corresponding splitting type vector is said to be {\em maximal} (cf. \cite[p.\,773--774]{Lar}).  Thus, posing $|\vec{e}| := \sum_{i=1}^{\nu} e_i$, one gets (cf. \cite[(2.2)]{Lar}) 
\begin{equation}\label{eq:LarRef2}
W_d^r(C) = \bigcup_{\begin{array}{c} 
\vec{e} \; {\rm s.t.} \; |\vec{e}| = d-g+1+\nu \\ 
h^0(\mathcal O_{\mathbb P^1}(\vec{e})) \geq r+1
\end{array}} W^{\vec{e}} = \bigcup_{\begin{array}{c} 
\vec{e} \; {\rm maximal \; s.t.} \; |\vec{e}| = d-g+1+\nu \\ 
h^0(\mathcal O_{\mathbb P^1}(\vec{e})) \geq r+1
\end{array}} \overline{W^{\vec{e}}}. 
\end{equation} Therefore, characterizing contributions of splitting loci to the Brill-Noether locus $W_d^r(C)$ is equivalent to 
determining those splitting type vectors that are maximal (with respect to the partial order given above) among those satisfying: 
\begin{equation}\label{eq:maximal}
\deg(\mathcal O_{\mathbb P^1}(\vec{e})) = \sum_{i=1}^{\nu} e_i = d+1-g-\nu \;\; {\rm and} \;\; h^0(\mathcal O_{\mathbb P^1}(\vec{e})) = \sum_{i=1}^{\nu} {\rm max}\,\{0,\,e_i+1\} \geq r+1,
\end{equation} as in \eqref{eq:degf*L} and \eqref{eq:rf*L}. These are the so-called  
{\em ``balanced plus balanced"} splitting type vectors in \cite[p.\,769]{Lar}, uniquely determined by the number of their non-negative components. The following is a second set of results:

\begin{theorem}\label{thm:Lar2} Let $C$ be a general $\nu$-gonal curve of genus $g \geq 4$ and let $r$ and $d$ be non-negative integers. 

\noindent
$(1)$ If $r \leq d-g$, then $W_d^r(C) = {\rm Pic}^d(C)$;

\medskip

\noindent
$(2)$ If $r > d-g$, set $d' := d+1-g-\nu$. Then:
\begin{itemize}
\item[(2-i)]  a splitting type vector $\vec{e}$, which is maximal among those satisfying \eqref{eq:maximal}, 
satisfies the condition $h^0(\mathcal O_{\mathbb P^1}(\vec{e}))= \sum_{i=1}^{\nu} {\rm max}\,\{0,\,e_i+1\} = r+1$ (cf. \cite[Lemma\,2.7]{CJ}); 

\item[(2-ii)] maximal splitting type vectors as in (2-i) are vectors $\vec{w}_{r,\ell}$ such that 
$$\mathcal O_{\mathbb P^1}(\vec{w}_{r,\ell}) = B(\nu-r-1+\ell, \; d'-\ell) \oplus B(r+1-\ell, \;\ell),$$where $B(- ,\; -)$ 
as above denotes the unique balanced bundle on $\Pp^1$ of given rank and degree, respectively, for 
${\rm max} \; \{0,\; r+2-\nu\} \leq \ell \leq r$ an integer such that either $\ell=0$ or $\ell \leq g+2r+1-d - \nu$. Moreover 
$$u (\vec{w}_{r,\ell}) = g - \rho' (g, \vec{w}_{r,\ell}) = g - \left(\rho(g, r - \ell, d) - \ell\,\nu\right),$$where $\rho' (g, \vec{w}_{r,\ell})$ as in \eqref{eq:expdimLar} (cf. \cite[Lemma\,2.2]{Lar}); 

\item[(2-iii)] for any integer $\ell$ as in (2-ii), the corresponding splitting type vector $\vec{w}_{r,\ell}$ is the unique 
degree--$d'$, rank--$\nu$ vector with $r+1-\ell$ non-negative coordinates, which is maximal among those with $r+1$ linearly independent global sections (cf. \cite[Cor.\;1.3]{Lar});

\item[(2-iv)] Every component of $W_d^r(C)$ is generically smooth, of (expected) dimension $\rho' (g, \vec{w}_{r,\ell}) = 
g - u (\vec{w}_{r,\ell}) = \rho(g, r - \ell, d) - \ell\,\nu$, for some integer ${\rm max} \; \{0,\; r+2-\nu\} \leq \ell \leq r$  such that either $\ell=0$ or  $\ell \leq g + 2r+1 - d - \nu$. Such a component exists for each integer $\ell$ as above satisfying furthermore $\rho' (g, \vec{w}_{r,\ell}) \geq 0$ (cf. \cite[Cor.\;1.3]{Lar}). 
\end{itemize}
\end{theorem}

Using this framework,  we prove the following easy result which will be essential for our purposes:

\begin{lemma}\label{lem:Lar} Let $g \geq 4$ and $3 \leq \nu < \lfloor \frac{g+3}{2}\rfloor$ be integers. Let $C$ be a general $\nu$-gonal curve of genus $g$ and let $A \in {\rm Pic}^{\nu}(C)$ be the unique line bundle on $C$ such that $|A| = g^1_{\nu}$. 
Let $t \geq \nu$ be an integer and, for a given integer $0 \leq \ell \leq 1$, let 
$\vec{w}_{1,\ell}$ be the corresponding {\em maximal splitting type vector} and 
$\mathcal O_{\Pp^1}(\vec{w}_{1,\ell})$ be the associated {\em balanced-plus-balanced bundle} on $\Pp^1$ as in Theorem \ref{thm:Lar2}. The one has: 
$$\mathcal O_{\Pp^1}(\vec{w}_{1,0}) = B(\nu-2, \,t+1-g-\nu) \oplus B(2,\,0)\;\; {\rm and} \,\;\mathcal O_{\Pp^1}(\vec{w}_{1,1}) = B(\nu-1, \,t+1-g-\nu) \oplus B(1,\,1),$$ where 
$B(-,-)$ denotes the most balanced bundle on $\Pp^1$ of given $({\rm rank},\;{\rm degree})$, so $B(2,0) \cong \mathcal O_{\Pp^1}^{\oplus 2}$, \linebreak $B(1,1) \cong \mathcal O_{\Pp^1}(1)$ whereas the other summands are bundles of negative degrees. Moreover, 
\[W^1_t(C) = 
\begin{cases}
\overline{W^{\vec{w}_{1,1}}} & \mbox{ if $\nu\leq t< \frac{g+2}{2}$}\\
\overline{W^{\vec{w}_{1,1}}} \ \cup \ \overline{W^{\vec{w}_{1,0}}} &  \mbox{ if $\frac{g+2}{2}\leq t\leq g+2-\nu$ } \\
\overline{W^{\vec{w}_{1,0}}} &  \mbox{ if $g-\nu+2< t < g+1$}\\
{\rm Pic}^t(C) & \mbox{if $t \geq g+1$}
\end{cases}
\]and, when $W^1_t(C) \subsetneq {\rm Pic}^t(C) $, any of its irreducible components is generically smooth of dimension
$$\dim(\overline{W^{\vec{w}_{1,1}}} ) = t-\nu \;\; {\rm and} \;\; \dim(\overline{W^{\vec{w}_{1,0}}} ) = \rho(g,1,t) = 2t-g-2.$$Furthermore:

\begin{itemize}
\item[(i)]  a general element in $\overline{W^{\vec{w}_{1,1}}}$ is given by $\mathcal L_{1,1}:= A \otimes \mathcal O_C(B_{t-\nu})$, where $|A| = g^1_{\nu}$ and $B_{t-\nu} = p_1+ \cdots + p_{t-\nu}$ the effective divisor of base points of $|\mathcal L_{1,1}|$; 

\item[(ii)] a general element in $\overline{W^{\vec{w}_{1,0}}}$ is given by $\mathcal L_{1,0}$, where $|\mathcal L_{1,0}|$ is a base-point-free, complete $g^1_t$.
\end{itemize} 
\end{lemma}

\begin{proof} By Riemann-Roch theorem, if $t \geq g+1$, i.e. $1 \leq t-g$, then $W^1_t(C) = {\rm Pic}^t(C)$ (cf. also Theorem \ref{thm:Lar2}-(1)). 

When otherwise $t < g+1$, to compute possible maximal splitting type vectors $\vec{w}_{1,\ell}$ as in Theorem \ref{thm:Lar2}-(2-ii), we need to consider integers $\ell$ such that  $0 = {\rm max} \{0, 3-\nu\} \leq \ell \leq r=1$ with either $\ell = 0$ or $\ell \leq g+2+1-t-\nu = g+3-t-\nu$. Notice that $g+3 - t - \nu \geq 1=r$ iff $t \leq g+2-\nu$. Therefore, for $t \leq g+2-\nu$, one has the two cases $\ell \in \{0,\,1\}$ whereas, for 
$g+2-\nu < t < g+1$, only $\ell =0$ occurs. Moreover, using Riemann-Roch theorem, one gets $\rho(g,1,t) = g - 2 (g-t+1) = 2t-g-2$, which is non-negative as soon as $t \geq \frac{g+2}{2}$. 

Similarly as in Theorem \ref{thm:Lar2}, setting $t':=  t+1-g-\nu$, one has the following:

\smallskip

\noindent
$\bullet$ for $(r,\ell) = (1,0)$, the maximal splitting type vector $\vec{w}_{1,0}$ gives $ \mathcal O_{\Pp^1}(\vec{w}_{1,0}) = B(\nu-2, \,t') \oplus B(2,\,0)$. Notice that $t'= t+1-g-\nu <0$, as $t < g+1$, and that $B(2,\,0) = 
\mathcal O^{\oplus\,2}_{\Pp^1}$. Moreover, in this case, $\rho'(g, \vec{w}_{1,0}) = \rho(g,1,t) = 2t-g-2 \geq 0$ iff $t \geq \frac{g+2}{2}$; 

\smallskip

\noindent
$\bullet$ for $(r,\ell) = (1,1)$, the maximal splitting type vector $\vec{w}_{1,1}$ gives $ \mathcal O_{\Pp^1}(\vec{w}_{1,1}) = B(\nu-1, \,t'-1) \oplus B(1,\,1)$, where $t'-1 = t-g-\nu <0$ and where $B(1,\,1) = \mathcal O_{\Pp^1}(1)$. Moreover, $\rho'(g, \vec{w}_{1,1}) = g - u( \vec{w}_{1,1}) = \rho(g,0,t) - \nu = t-\nu \geq 0$ iff $\nu \leq t$. 

\noindent
Therefore, the first part of the statement is a consequence of the previous computations and of Theorem \ref{thm:Lar2}.

\smallskip 

For what concerns the description of the general element of $\overline{W^{\vec{w}_{1,1}}}$ as in (i), this is a consequence of a simple dimension count; namely 
$$\dim\{g^1_{\nu} + p_1 + \cdots + p_{t-\nu}\} = \dim(C^{(t-\nu)}) = \dim(\overline{W^{\vec{w}_{1,1}}} ).$$To prove (ii) instead, let $\mathcal L_{1,0}$ denote a general element of  
$\overline{W^{\vec{w}_{1,0}}}$; thus one has  $h^0 (\mathcal L_{1,0})=2$ as 
$\mathcal L_{1,0}$ is, by definition, associated  to $B(\nu-2, \,t') \oplus B(2,\,0)$ 
on $\mathbb P^1$.  We set $g^1_t := |\mathcal L_{1,0}|$. 

Suppose by contradiction that  such a pencil  $g^1_t$  has a base point $p$; so $g^1_t (-p)\in W^1_{t-1} (C)$ has the same  global section space as  $g^1_t$. Since $g^1_t $  is associated  to $B(\nu-2, \,t') \oplus B(2,\,0)$, the linear series
$g^1_t (-p)$ must be  therefore associated to  $B(\nu-2, \,t'-1) \oplus B(2,\,0)$ on $\mathbb P^1$, that is    $g^1_t (-p) \in \overline{W^{\vec{w}_{1,0}}}_* \subset W^1_{t-1} (C)$,  where $\overline{W^{\vec{w}_{1,0}}}_*$ denotes the irreducible component of $W^1_{t-1} (C)$ whose  general element is associated  to $B(\nu-2, \,t'-1) \oplus B(2,\,0)$.  This implies that 
$$\dim ( \overline{W^{\vec{w}_{1,0}}} ) \leq 
\dim (\overline{W^{\vec{w}_{1,0}}}_*) +1,$$which is contrary to 
 $\dim(\overline{W^{\vec{w}_{1,0}}} ) = \rho(g,1,t) = 2t-g-2 > \dim (\overline{W^{\vec{w}_{1,0}}}_*) +1=\rho(g,1,t-1) +1 = 2t -g-3$. Therefore the  general element $\mathcal L_{1,0}$ of $ \overline {W^{\vec{w}_{1,0}}} $ gives rise to a base-point-free linear series. 
\end{proof}

\bigskip

\subsection{Ruled surfaces and unisecants}\label{ss:unisec} Here we recall some preliminaries which are fundamental to construct {\em geometric presentations} $(*)$ as in Introduction of bundles $\Ff$ on $C$ which turn-out to be candidates to correspond to general points of possible irreducible components of $B_d^{k_2} \cap U^s_C(d)$.

Let $C$ be any smooth, irreducible, projective curve of genus $g$, and let $\Ff$ be a rank-$2$, degree $d$ vector bundle on $C$. Then  $F:= \Pp(\Ff) \stackrel{\rho}{\to} C$ will denote the {\em (geometrically) ruled
surface} (or simply the {\em surface scroll}) associated to  the pair $(\Ff, C)$; $f$ will denote the general $\rho$-fibre and 
$\Oc_F(1)$ the {\em tautological line bundle}. 

If $H$ denotes a divisor in the {\em tautological linear system} $|\Oc_F(1)| \neq \emptyset$ and if $\widetilde{\Gamma}$ is another divisor on $F$, we obvioulsy set 
$${\rm deg}(\widetilde{\Gamma}) := \widetilde{\Gamma} H;$$in particular one has therefore
$d=\deg(\Ff) = \deg(\det(\Ff)) = H^2 = \deg(H)$. The integer $$i(\Ff):= h^1(C,\Ff)$$is called the {\em speciality} of   the bundle $\Ff$, simply denoted by $i$ if there is no danger of confusion. Thus $\Ff$ (and also the surface scroll $F$) is said to be {\em non-special} if $i= 0$,  {\em special} otherwise.

We briefly remind some basics on unisecant curves on the surface scroll $F$; for more details the reader is referred to e.g. \cite[V-2]{H} or to \cite[\S\,2]{CF}. Let ${\rm Div}^1_F$ be the subscheme of ${\rm Div}_F$ formed by all divisors 
$\widetilde{\Gamma}$ of $F$ such that $\Oc_F(\widetilde{\Gamma}) \simeq \Oc_F(1) \otimes \rho^*(N^{\vee})$, for some $N \in {\rm Pic} (C)$. One has a natural morphism 
$$\Psi_1: {\rm Div}^1_F \to {\rm Pic}(C)\; \; \mbox{defined by}\;\; \widetilde{\Gamma} \stackrel{\Psi_1}{\longrightarrow} N.$$If $D \in {\rm Div} (C)$, then $\rho^*(D)$ will be denoted by $f_D$, so $\widetilde{\Gamma} \in {\rm Div}^1_F$ 
if and only if $\widetilde{\Gamma} \sim H - f_D$, for some $D \in {\rm Div}(C)$, and $\deg(\widetilde{\Gamma}) = \deg(\Ff) - \deg(D)$. Curves in ${\rm Div}^1_F$ are called {\em unisecants}  of the surface scroll $F$ and irreducible unisecants are isomorphic to $C$ and are called  {\em sections} of $F$. For any positive integer $\delta$, we consider 
$${\rm Div}_F^{1,\delta} : = \{\widetilde{\Gamma} \in {\rm Div}_F^1 \; | \; {\rm deg}(\widetilde{\Gamma}) = \delta \},$$which is the {\em Hilbert scheme of unisecants of $F$ of degree $\delta$} (with respect to $H$).

Let  $\widetilde{\Gamma} = \Gamma + f_A$ be a unisecant on $F$, with  $\Gamma$ a section of $F$ and $A $ an effective divisor of $C$. Then, the standard exact sequence defining 
$\widetilde{\Gamma}$ as a Cartier divisor on $F$ gives rise to the exact sequence 
\begin{equation}\label{eq:Fund2}
0 \to N ( - A) \to \Ff \to L \oplus \Oc_{A} \to 0;
\end{equation} in particular if $A =0$, i.e. 
$\widetilde{\Gamma} = \Gamma$ is a section of $F$, then $\Ff$ fits in 
\begin{equation}\label{eq:Fund}
0 \to N \to \Ff \to L \to 0
\end{equation} and the normal bundle of the section $\Gamma$ in $F$ is: 
\begin{equation}\label{eq:Ciro410}
\N_{\Gamma/F} \simeq L \otimes N^{\vee}, \;\; {\rm so} \;\; \Gamma^2 = \deg(L) - \deg(N) = 2\delta - d  
\end{equation} (cf. e.g. \cite[Rem.\,2.1]{CF}). 

Accordingly, if we consider the restriction of the map $\Psi_1$ to the Hilbert scheme ${\rm Div}_F^{1,\delta}$, we get a map 
$$\Psi_{1,\delta}: {\rm Div}^{1,\delta}_F \to {\rm Pic}^{d-\delta}(C)$$ 
which endows ${\rm Div}_F^{1,\delta}$ with the structure of a {\em Quot scheme}, namely as in \cite[\S\,4.4]{Ser}, one has 
\begin{equation}\label{eq:isom1}
\begin{array}{rrcc}
\Psi_{1,\delta}: & {\rm Div}_F^{1,\delta} & \stackrel{\simeq}{\longrightarrow} & {\rm Quot}^C_{\Ff,\delta+t-g+1} \\
 & \widetilde{\Gamma} & \longrightarrow & \left\{\Ff \to\!\!\to L \oplus \Oc_A \right\}.  
\end{array}
\end{equation} From standard results (cf. e.g. \cite[\S\,4.4]{Ser}), \eqref{eq:isom1} allows description of tangent and obstruction spaces; namely using \eqref{eq:Fund2} we have 
$$H^0(\N_{\widetilde{\Gamma}/F}) \simeq T_{[\widetilde{\Gamma}]} ({\rm Div}_F^{1,\delta}) \simeq {\rm Hom} (N (-A), L \oplus \Oc_A)\;\;\; {\rm and} \;\;\; H^1(\N_{\widetilde{\Gamma}/F})  \simeq {\rm Ext}^1 (N (-A), L \oplus \Oc_A).$$Finally, if $\widetilde{\Gamma} \sim H - f_D$ on $F$, for some effective divisor $D$ on $C$, then one has
\begin{equation}\label{eq:isom2}
|\Oc_F(\widetilde{\Gamma})| \simeq \Pp(H^0(\Ff (-D))).  
\end{equation}

\begin{definition}(cf. \cite[Def.\;2.7]{CF})\label{def:ass0} A unisecant $\widetilde{\Gamma} \in {\rm Div}^{1,\delta}_F$ is said to be:  

\noindent
(a) {\em linearly isolated}, (li, for short)  if $\dim(|\Oc_F(\widetilde{\Gamma} )|) =0$; 

\noindent
(b) {\em algebraically isolated}, (ai, for short)  if  $\dim({\rm Div}^{1,\delta}_F) =0$.  

\noindent
(c) Setting $\Oc_{\widetilde{\Gamma}} (1) := \Oc_F(1) \otimes \Oc_{\widetilde{\Gamma}}$, one calls the {\em speciality} of $\widetilde{\Gamma}$ the number $i(\widetilde{\Gamma}) := h^1(\Oc_{\widetilde{\Gamma}} (1))$ and $\widetilde{\Gamma}$ is said to be {\em special} if $i(\widetilde{\Gamma}) >0$. 
\end{definition} If $\widetilde{\Gamma}$ is given by \eqref{eq:Fund2}, by \eqref{eq:isom2} one has 
$\widetilde{\Gamma} \in |\Oc_F(1) \otimes \rho^*(N^{\vee}(A))|$ so, applying $\rho_*$, one gets
\begin{equation}\label{eq:iLa}
i(\widetilde{\Gamma}) = h^1(L \oplus \Oc_{A}) = h^1(L) = i(\Gamma), 
\end{equation}where $\Gamma$ denotes the unique section contained in  $\widetilde{\Gamma}$.

In general, {\em speciality} is not constant either in linear systems or in  (flat) algebraic families of unisecants (cf. e.g. \cite[Examples\,2.8,\,2.9]{CF}). Nevertheless, since ${\rm Div}^{1,\delta}_F$ has a Quot-scheme structure there is a universal quotient $\mathcal Q_{1,\delta} \to {\rm Div}^{1,\delta}_F$ so, taking  $\mathbb P (\mathcal Q_{1,\delta}):= {\rm Proj} (\mathcal Q_{1,\delta}) \stackrel{p}{\to} {\rm Div}^{1,\delta}_F$, one can consider 
\begin{equation}\label{eq:aga}
\mathcal S_F^{1,\delta} := \{\widetilde{\Gamma} \in {\rm Div}^{1,\delta}_F \;\; | \;\; R^1p_*(\Oc_{\Pp(\mathcal Q_{1,\delta})}(1))_{\widetilde{\Gamma}} \neq 0\} \; \;\;  {\rm and} \; \;\; a_F(\delta) := \dim (\mathcal S_F^{1,\delta}), 
\end{equation} i.e.  $\mathcal S_F^{1,\delta}$ is the support of $R^1p_*(\Oc_{\Pp(\mathcal Q_{1,\delta})}(1))$, which parametrizes degree-$\delta$, special unisecants of $F$.

\begin{definition}(cf. \cite[Def. 2.10]{CF})\label{def:ass1} Let $\widetilde{\Gamma}$ be a special unisecant of the surface scroll $F$. 
Assume $\widetilde{\Gamma} \in \mathfrak{F}$, where $\mathfrak{F} \subseteq {\rm Div}^{1,\delta}_F$ is an irreducible subscheme. 

\smallskip

\noindent
$(1)$ We will say that $\widetilde{\Gamma}$ is: 

\noindent
(1-i) {\em specially unique (su)} in $\mathfrak{F}$, if $\widetilde{\Gamma}$ is  the only special unisecant in $\mathfrak{F}$, or   

\noindent
(1-ii) {\em specially isolated (si)} in $\mathfrak{F}$, if $\dim_{\widetilde{\Gamma}} \left(\mathcal S_F^{1,\delta} \cap \mathfrak{F} \right) = 0$.

\smallskip

\noindent
$(2)$ In particular:

\noindent
(2-i) when $\mathfrak{F} = |\Oc_F(\widetilde{\Gamma})|$ is a {\em linear system}, then  $\widetilde{\Gamma}$ is said to be {\em linearly specially unique (lsu)} in case (1-i) and  {\em linearly specially isolated (lsi)} in case (1-ii);

\noindent 
(2-ii) when $ \mathfrak{F} =  {\rm Div}^{1,\delta}_F$,  $\widetilde{\Gamma}$ is said to be {\em algebraically specially unique (asu)} in case (1-i) and 
{\em algebraically specially isolated (asi)} in case (1-ii).  

\smallskip

\noindent
$(3)$ When a section $\Gamma \subset F$ is (asi), $\Ff$ is said to be {\em rigidly specially presented (rsp)} by the sequence \eqref{eq:Fund} corresponding to $\Gamma$. When $\Gamma$ is ai (cf. Def. \ref{def:ass0}), we will also say that 
$\Ff$ is {\em rigidly presented (rp)} via \eqref{eq:Fund}.  
\end{definition} We conclude with the following results, which will be sometimes used in the sequel.

\begin{proposition}\label{prop:CFliasi} (1) Let $\Gamma \subset F$ be a section corresponding to the exact sequence \eqref{eq:Fund}. A section $\Gamma'$, corresponding to a quotient $\Ff \to\!\!\!\!\! \to L'$, is such that  $\Gamma \sim \Gamma'$ on the surface scroll $F$ if and only if the corresponding quotient line bundles are such that $L \simeq L'$. In such a case, in particular one has $i(\Gamma) = i(\Gamma')$. 

Moreover, in such a case, the section $\Gamma$ is (lsu) on $F$ if and only if it is (lsi) on $F$ equivalently if and only if it is (li) on $F$.

\smallskip

\noindent
(2) Assume the vector bundle $\Ff$ to be non-splitting and let $\Gamma \in \mathfrak{F} \subseteq \mathcal{S}^{1,\delta}_F$ be a section of the surface scroll $F$ associated to $\Ff$, where  $\mathfrak{F}$ denotes any irreducible, projective scheme of dimension $k$, parametrizing a flat family of special unisecants. Assume further that either:  

\smallskip

\noindent
(2-i) $k \geq 1$, when $\mathfrak{F}$ is a linear system, or, 

\smallskip

\noindent
(2-ii) when $\mathfrak{F}$ is not a linear family, assume either $k \geq 2$ or  $k=1$ and $\mathfrak{F}$ with base points.

\smallskip

Then, in all the aforementioned cases, the family $\mathfrak{F}$ contains reducible unisecants $ \widetilde{\Gamma} $ with $i(\widetilde{\Gamma}) \geq i(\Gamma)$.
\end{proposition}

\begin{proof} The proofs are taken from \cite[Lemma\,2.11, Prop.\;2.12]{CF}; they are briefly included here, for the reader convenience, to let the paper be to some reasonable extent self-contained. As for (1), the first assertion follows from  \eqref{eq:isom2}, whereas the latter ones are both clear.

Concerning the proof of (2), if $k \geqslant 2$ let $t$ be the unique integer such that $0 \leqslant k':= k-2t \leqslant 1$. Let $f_1, \ldots, f_t$ be $t$ general $\rho$-fibres of $F$. Since $k' \geqslant 0$, by imposing to the curves in $\mathfrak{F}$ to contain fixed general pairs of points on $f_1, \ldots, f_t$,  we see that 
$$\mathfrak{F}' := \mathfrak{F} \left(- \sum_{i=1}^t f_i\right) \subset \mathfrak{F}$$is non-empty, all components of it have dimension $k'$, and they all parametrize  unisecants 
$\Gamma'  \sim_{alg} \Gamma -  \sum_{i=1}^t f_i$. Then $\mathfrak{F}$ contains reducible elements $\widetilde{\Gamma}$, and they verify $i(\widetilde{\Gamma}) \geq i(\Gamma)$ by upper-semicontinuity in irreducible, flat families. This proves the assertion when $k \geqslant 2$.  

So, to complete the proof of (2), we are left with the case $k=1$. 

Assume first that $\mathfrak{F}$ is a linear pencil. Since $\mathfrak{F} \subseteq |\Oc_F(\Gamma)|$, from the exact sequence $$0 \to \Oc_F \to \Oc_F(\Gamma) \to \Oc_{\Gamma}(\Gamma) \to 0,$$the line bundle  $\Oc_{\Gamma}(\Gamma)$ is effective which implies $\Gamma^2 \geqslant 0$. Let ${\rm Bs}(\mathfrak{F})$ be the base locus of $\mathfrak{F}$. If  $\Gamma^2 >0$ we take $p \in {\rm Bs}(\mathfrak{F})$ and we can clearly split-off the fibre through $p$ with one condition, thus proving the result. If otherwise $\Gamma^2 = 0$, $\mathfrak{F}$ is a base-point-free pencil, so $F$ contains two disjoint sections and this implies that $\Ff$ is decomposable, a contradiction to our assumptions.  
 
Finally, if $\mathfrak{F}$  is assumed to be non-linear, then ${\rm Bs} (\mathfrak{F}) \neq \emptyset$ and we can argue similarly as in the linear case with $\Gamma^2 >0$. 
\end{proof}

\bigskip

\subsection{Moduli spaces of rank-two (semi)stable bundles and Brill Nother loci}\label{ss:modspaces} Here we briefly recall some fundamental facts on (semi)-stable vector bundles on curves.

Given a rank-$2$ vector bundle $\Ff$ on a smooth, irreducible projective curve $C$ of genus $g$, recall that $\Ff$ is said to be {\em stable} (resp. {\em semistable}) if for every sub-line bundle $N\subset \Ff$ one has $\mu(N)<\mu(\Ff)$ (resp. $\mu(N)\le\mu(\Ff)$), where
 $$\mu(\mathcal E):=\frac{\deg (\mathcal E)}{{\rm rk}(\mathcal E)}$$is the so-called \emph{slope} of any vector bundle $\mathcal E$ on $C$. As in Introduction, $U_C(d)$ will denote the moduli space of (semi)stable, degree $d$, rank-$2$ vector bundles on  $C$, which is irreducible, projective and of dimension $4g-3$. 
The subset $U_C^s(d)\subseteq U_C(d)$ parametrizing (isomorphism classes of) stable bundles, is an open, dense subset. The points in  ${U}^{ss}_C(d):=U_C(d)\setminus U_C^s(d)$ correspond to (S-equivalence classes of) \emph{strictly semistable} bundles on $C$ (cf. e.g. \cite{Ram,Ses}). One has: 

\begin{proposition}\label{prop:sstabh1} Let $C$ be a smooth, irreducible, projective curve of  genus $g \geq 1$ and let $d$ be an integer.

\noindent
(i) If $d \geq 4g-3$, then for any $[\Ff] \in U_C(d)$, one has $i(\Ff) = h^1(\Ff) = 0$.

\noindent 
(ii) If $g \geq 2$ and $d \geq 2g-2$, for $[\Ff] \in U_C(d)$ general,  one has $i(\Ff) =h^1(\Ff) = 0$.  
\end{proposition}
 
\begin{proof} For (i), see  \cite[Lemma 5.2]{New}; for (ii) see \cite[p.\,100]{L} or  \cite[Rem.\,3]{Ballico}
\end{proof}

From Proposition \ref{prop:sstabh1}, Serre duality and invariance of stability under operations like tensoring with a line bundle or passing to the dual bundle, for $g \geq 2$ it therefore makes sense to consider the proper sub-loci of $U_C(d)$ parametrizing classes  $[\Ff]$ such that  $i(\Ff) > 0$ in the following range for $d$:    
\begin{equation}\label{eq:congd}
2g-2 \leq d \leq 4g-4. 
\end{equation}

\begin{definition}\label{def:BNloci} Given $C$ any smooth, irreducible, projective curve of genus $g \geq 2$ and considered  $d$ and $ i $ two non-negative integers, we set $k_i := d - 2 g + 2 + i$ and define 
\begin{equation}\label{eq:BdKi}
B^{k_i}_d := \left\{ [\Ff] \in U_C(d) \, | \,  h^0(\Ff) \geq k_i \right\} = 
\left\{ [\Ff] \in U_C(d) \, | \, h^1(\Ff) = i(\Ff) \geq i \right\},
\end{equation} the $k_i^{th}$--{\em Brill-Noether locus}, 
simply {\em Brill-Noether locus}, when no confusion arises.
\end{definition}

\begin{remark}\label{rem:BNloci}  {\normalfont Brill-Noether loci $B^{k_i}_d$ have a natural structure of (locally) determinantal subschemes of $U_C(d)$: 

\medskip

\noindent
$(a)$ When $d$ is an odd integer, one has $U_C(d)=U^s_C(d)$ so $U_C(d)$  is a fine moduli space and the existence of a universal bundle on $C \times U_C(d)$  allows one to construct $B^{k_i}_d$ as the degeneracy locus of a morphism between suitable  vector bundles on $U_C(d)$ (see, e.g. \cite{GT} for details). Accordingly, the  {\em expected dimension} of $B^{k_i}_d$ is 
${\rm max} \{ -1, \ \rho_d^{k_i}\}$, where 
\begin{equation}\label{eq:bn}
\rho_d^{k_i}:= 4g - 3 - i k_i
\end{equation} is  the \emph{Brill-Noether number} in rank-two.  If  $\emptyset \neq B^{k_i}_d \subsetneq U_C(d)$, 
then $B^{k_i+1}_d \subseteq {\rm Sing}(B^{k_i}_d)$. Since any $[\Ff] \in B^{k_i}_d$ is stable, it is a smooth point of $U_C(d)$ and the Zariski tangent space $T_{[\Ff]}(U_C(d))$ at the point $[\Ff]$ to $U_C(d)$ can be identified with $H^0( \omega_C \otimes \Ff \otimes \Ff^{\vee})^{\vee}$. 

If $[\Ff] \in B^{k_i}_d \setminus B^{k_i+1}_d$, the tangent space to $B^{k_i}_d$ at $[\Ff]$ is the  annihilator of the image of the cup--product, i.e. of the {\em Petri map} of $\Ff$ (see, e.g. \cite{Teixidor2}) 
\begin{equation}\label{eq:petrimap}
\mu_{\Ff} : H^0(C, \Ff) \otimes H^0(C, \omega_C \otimes \Ff^{\vee})
\longrightarrow H^0(C, \omega_C \otimes \Ff \otimes \Ff^{\vee}).
\end{equation} 
Thus, if $[\Ff] \in B^{k_i}_d \setminus B^{k_i+1}_d$, one gets  
$\rho_d^{k_i} = h^1(C, \Ff \otimes \Ff^{\vee}) - h^0(C, \Ff) h^1(C, \Ff)$ and $B^{k_i}_d$ is non--singular, of the expected dimension at $[\Ff]$ if and only if the Petri map $\mu_{\Ff}$ is injective.

\medskip 

\noindent
$(b)$ When otherwise $d$ is even, $U_C(d)$ is not a fine moduli space (because $U^{ss}_C(d)\neq \emptyset$). There is a suitable open, non-empty subscheme $ \mathcal Q^{ss} \subset \mathcal Q$ of a certain Quot scheme $\mathcal Q$ defining $U_C(d)$ via a GIT-quotient map $\pi$ (cf. e.g. \cite{Tha} for details) and one can define $B^{k_i}_C(d)$ as the image via $\pi$ of the degeneracy locus of a morphism between suitable vector bundles on $\mathcal Q^{ss}$. It may happen for a component   $\mathcal B $ of a Brill--Noether locus to be  totally contained in $U_C^{ss}(d)$.  In this case the lower bound  $\rho_d^{k_i}$ for the expected dimension of $\mathcal B$ may be no longer valid  (cf. e.g. 
\cite[Remark 7.4]{BGN}).  Nonetheless, the lower bound $ \rho_d^{k_i}$ is still valid if $\mathcal B\cap U^s_C(d)\neq \emptyset$. } 
\end{remark} With the above reminded, we will use the following standard terminology.

\begin{definition}\label{def:regsup} Assume $B_d^{k_i} \neq \emptyset$. A component $ \mathcal B \subseteq B_d^{k_i}$ 
such that $\mathcal B \cap U_C^s(d) \neq \emptyset$ will be called {\em regular},  if it is generically smooth and of (expected) dimension $\dim(\mathcal B) = \rho_d^{k_i}$, 
{\em superabundant}, otherwise. 
\end{definition}

\bigskip

\subsection{Segre invariant and extensions of line bundles}\label{ss:extsegre} The present section reminds some important ingredients and results ensuring stability of vector bundles we are going to contruct later on.

In what follows, $C$ will always denote any smooth, irreducible, projective curve of genus $g \geq 3$. Given a rank-$2$ vector bundle $\mathcal{F}$ on $C$, its \emph{Segre invariant} $s(\mathcal{F}) \in \mathbb{Z}$ is defined as
\[
s(\mathcal{F}) := \min_{N \subset \mathcal{F}} \left\{ \deg(\mathcal{F}) - 2 \deg(N) \right\},
\]
where $N$ runs over all sub-line bundles of $\mathcal{F}$. It is well known that the Segre invariant is preserved under tensoring with line bundles and under dualization; that is,
\[
s(\mathcal{F}) = s(\mathcal{F} \otimes L), \quad \text{for any line bundle } L, \quad \text{and} \quad s(\mathcal{F}) = s(\mathcal{F}^{\vee}).
\]

A sub-line bundle $N \subset \mathcal{F}$ is called a \emph{maximal-degree sub-line bundle} of $\mathcal{F}$ if $\deg(N)$ is maximal among all sub-line bundles of $\mathcal{F}$. In this case, the quotient $\mathcal{F}/N$ is referred to as a \emph{minimal-degree quotient line bundle} of $\mathcal{F}$, i.e., it has minimal degree among all quotient line bundles of $\mathcal{F}$.

In particular, if $\Gamma \in \mathrm{Div}^{1,\delta}_F$ is a section on the surface scroll $F = \mathbb{P}(\mathcal{F})$ such that $\Gamma^2 = s(\mathcal{F})$, then by \eqref{eq:Ciro410} we have
\[
s(\mathcal{F}) = 2\delta - d,
\]
and $\Gamma$ is called a section of \emph{minimal degree} (equivalently, of \emph{minimal self-intersection}) on $F$.
Thus the bundle $\mathcal{F}$ is \emph{semistable} (resp.\ \emph{stable}) if and only if $s(\mathcal{F}) \ge 0$ (resp.\ $s(\mathcal{F}) > 0$).

Let now $\delta$ be a positive integer and consider $L\in \pic^\delta(C)$ and $N\in\pic^{d-\delta}(C)$.  Any vector $u\in\ext^1(L,N)$ gives rise to a degree-$d$, rank-$2$ vector bundle $\mathcal F_u$, fitting in:
\begin{equation}\label{degree}
(u):\;\; 0 \to N \to \mathcal F_u \to L \to 0
\end{equation} (which actually coincides with sequence $(*)$ as in Introduction). If $\Gamma$ is the section on the surface scroll $F_u = \mathbb P(\Ff_u)$ corresponding to the surjection $\mathcal F_u \to\!\!\!\to L$ then, 
since $\Gamma^2 = \deg(L-N)$, a necessary condition  for $\mathcal F_u$ to be semistable is therefore 
\begin{equation}\label{eq:neccond}
\Gamma^2= 2\delta-d \ge s(\mathcal F_u)\ge 0.
\end{equation} In such a case, the Riemann-Roch theorem gives 
\begin{equation}\label{eq:m}
\dim (\ext^1(L,N))=
\begin{cases}
\ 2\delta-d+g-1\ &\text{ if } L\ncong N \\
\ g\ &\text{ if } L\simeq  N.
\end{cases}
\end{equation}

\color{black}

\begin{lemma}\label{lem:1e2note} (cf. \cite[Lemma\,4.1]{CF}) Let $\Ff$ be a (semi)stable, special, rank-$2$ vector bundle on $C$ of degree $d \geq 2g-2$. 

Then $\Ff = \Ff_u$, for a special and effective line bundle $L \in {\rm Pic}^{\delta}(C)$, for some $\delta \geq \frac{d}{2}$, $N = \det(\Ff) \otimes L^{\vee}\in {\rm Pic}^{d-\delta}(C)$ and a vector $u \in {\rm Ext}^1(L,N)$ as in \eqref{degree}.
\end{lemma}

\begin{proof} The proof is short and taken from \cite[Lemma\,4.1]{CF}; it is briefly included here, for the reader convenience, to let the paper be to some reasonable extent self-contained. 

By Serre duality, the fact that $i(\Ff) >0 $ implies there exists a non-zero morphism $\Ff \stackrel{\sigma^{\vee}}{\to} \omega_C$. The line bundle $L:= {\rm Im}(\sigma^{\vee}) \subseteq \omega_C$ is therefore special. 

Set $\delta:= \deg(L)$. Since $\Ff$ is (semi)stable, then  \eqref{eq:neccond} holds hence  $\delta \geqslant \frac{d}{2} \geqslant g-1$, the latter inequality form assumptions on $d$. Therefore one has $\chi(L) \geqslant 0$ which implies that $L$ is also effective.
\end{proof}

\begin{remark}\label{rem:rigid} {\normalfont With assumptions as in Lemma \ref{lem:1e2note}, any (semi)stable, rank-$2$, special vector bundle of degree $d$ must be {\em ``presented"} as an extension \eqref{degree}, where the quotient line bundle $L$ must be effective, with speciality $ 1 \leq h^1(C,L) \leq i(\Ff)$, of some degree $\delta \geq \frac{d}{2}$ (with strict inequality when {\em stability} holds). In terms of the surface scroll  $F= \Pp(\Ff)$, if we let $\Gamma \subset F$ to be the section corresponding to $L$, and if we moreover assume that it is of minimal degree among quotient line bundles of $\Ff$ then, from Proposition \ref{prop:CFliasi}-(2), if $\Ff$ is indecomposable (as it occurs if e.g. $\Ff$ is {\em not polystable}), then $\Gamma  \subset F$ is linearly isolated (li) with $a_{F}(\delta) \leqslant 1$, where $a_{F}(\delta)$ as in \eqref{eq:aga}. Moreover, the bundle $\Ff$ is rigidly specially presented (rsp) via the quotient line bundle $L$ if $a_F(\delta)=0$, and it is even rigidly presented (rp) if $\Gamma$ is algebraically isolated (ai) on $F$.}
\end{remark}

From Lemma \ref{lem:1e2note}, for every (semi)stable, special, rank-$2$ vector bundle $\mathcal F$ on $C$ of degree $d\geq 2g-2$, there exist a special and  effective line bundle $L$ on $C$ such that, considered $N:=\det(\mathcal F)\otimes L^{\vee}$, there exists $u\in Ext^1(L,N)$ for which  $\mathcal F=\mathcal F_u$. In this set-up, we recall the following:

\begin{theorem} \label{thm:barj} (cf. \cite[Prop.\;1]{Barj}) Let $C$ be a smooth, irreducible, projective curve of genus $g \geq 3$. Let $\Ff$ be a rank-$2$ vector bundle on $C$ of degree $d\geq 2g-2$, with speciality $i = h^1(\Ff) \geq 2$. Then $\Ff$ fits into an exact sequence of the form 
\begin{equation}\label{eq:barj-1}
0 \to N \to \Ff \to \omega_C(-D) \to 0,
\end{equation}where $D$ is an effective divisor on $C$ such that either $h^1(\omega_C(-D))=1$ or $h^1(\omega_C(-D)) = i$. 
\end{theorem}
\begin{proof} This is the ``Serre-dual" version of \cite[Lemma\,1]{Barj} or even of \cite[Cor.\,1.2]{Teixidor2}, cited therein 
as a direct consequence on an unpublished result of B. Feinberg. This result has been completely proved in \cite[Prop.\;1]{Barj}. 
\end{proof} 
\color{black}

\begin{definition}\label{def:fstype}
A bundle $\Ff$, fitting in \eqref{eq:barj-1},  of speciality $i = h^1(\Ff) \geq 2$ and satisfying $h^1(\omega_C(-D)) = i$ is said to be {\em a bundle of first type}; otherwise if $h^1(\omega_C(-D)) =1$ then $\Ff$ is said to be {\em a bundle of second type} (cf. \cite[Def.\;1]{Barj}). 
\end{definition} 

In the following sections, we are going to construct and classify all the irreducible components $\mathcal B$ of $B_d^{k_2} \cap U^s_C(d)$ with the use of the above subdivision in presentations of the rank-$2$, degree $d$ vector bundle $\Ff$ corresponding to a general point $[\Ff]$ of such a component $\mathcal B$, namely a bundle $\Ff$ which is either of {\em first type} or of {\em second type} in the sense of Definition \ref{def:fstype}.

\smallskip

We conclude this section by recalling a result by Lange and Narasimhan, cf. \cite{LN}, which allows (via projective-geometry approach) to sometimes deduce when a bundle $\Ff_u$ as in \eqref{degree} turns out to be (semi)stable. When $u \in \ext^1(L,N)$ corresponds to a non-trivial sequence \eqref{degree}, it defines a point $[u] \in \mathbb P (\ext^1(L,N)):= \mathbb{P}$. When the natural map $$\varphi:=\varphi_{|K_C+L-N|}: C\to\mathbb P$$ is a morphism (namely $|K_C+L-N|$ is base-point-free), set $X:=\varphi(C)\subset \mathbb P$ and consider, for any positive integer $h$, the so-called $h^{th}$-{\em secant variety} of $X$, denoted by $\Sec_h(X) \subseteq \mathbb{P}$ and defined as the closure of the union of all linear subspaces $\langle \varphi(D)\rangle\subseteq\mathbb P$,
for general effective divisors $D$ of degree $h$ on $C$. One obviously has $$\dim(\Sec_h(X)) =\min\,\{\dim(\mathbb P)\; ,\;2h-1\}.$$Within this framework, the following result holds:

\begin{proposition} (\cite[Proposition 1.1]{LN})\label{LN} Let $C$ be any smooth, irreducible, projective curve of genus $g \geq 3$. Let $\delta$ be a positive integer and $L\in \pic^\delta(C)$, $N\in\pic^{d-\delta}(C)$ be line bundles on $C$ such that $\deg(L-N) = 2\delta-d\ge 2$. Then, the map 
$$\varphi=\varphi_{|K_C+L-N|}: C\to\mathbb P = \mathbb P (\ext^1(L,N))$$is a morphism.

Set $X = \varphi(C) \subset \mathbb P$; for any point $[u] \in \mathbb P$, let $\Ff_u$ denote the corresponding rank-$2$, degree $d$ vector bundle on $C$ fitting in \eqref{degree}. Then, for any integer
$\sigma \equiv 2\delta-d\pmod 2  \;\; \mbox{such that} \;\; 4+ d-2\delta\le \sigma \le 2\delta-d,$ one has $$s (\mathcal F_u)\ge \sigma \;\; \mbox{if and only if} \;\; [u]\notin \Sec_{\frac{1}{2}(2\delta-d+\sigma-2)}(X).$$In such a case, if $\sigma \geq 1$ then  $\Ff_u$ is stable. 
\end{proposition}

\bigskip

\section{Brill-Noether loci in speciality 2: proof of Theorem \ref{thm:i=2}}\label{S:BN2} Here we focus on the Brill-Noether locus $B_d^{k_2} \cap U^s_C(d)$ for which, from Remark \ref{rem:BNloci}, the expected dimension of any of its irreducible component is  $\rho_d^{k_2} = 8g-11-2d$. 

Therefore from now on, unless otherwise stated, we will consider $C$ to be a general $\nu$-gonal curve of genus $g \geq 4$ and, as in \eqref{eq:nu} and \eqref{eq:congd}, we will set $$3 \leq \nu < \lfloor \frac{g+3}{2}\rfloor\;\;\; {\rm and} \;\;\; 2g-2 \leq d \leq 4g-4.$$

Given more generally $i >0 $ an integer recall that, as in Definition \ref{def:BNloci}, we set $k_i= d-2g+2+i$ and we may consider $B_d^{k_i} \cap U^s_C(d)$ the corresponding Brill-Noether locus, whose expected dimension is $\rho_d^{k_i}$ as in \eqref{eq:bn}. To prove Theorem \ref{thm:i=2} we will make use of Lemma \ref{lem:1e2note} (equivalently, of Theorem \ref{thm:barj}) from which we know that, for any possible irreducible component $\mathcal B$ of $B_d^{k_2} \cap U^s_C(d)$, its general point $[\Ff] \in \mathcal B$ corresponds 
to a rank-$2$, degree $d$ vector  bundle $\mathcal F$ arising as an extension \eqref{degree}, equivalently \eqref{eq:barj-1}, with quotient line bundle $L = \omega_C(-D)$ being effective, special, of bounded degree $\deg(L) := \delta = 2g-2-\deg(D) > \frac{d}{2}$ from stability condition \eqref{eq:neccond}.

\bigskip

\subsection{{\bf Preliminary results to study components of $B_d^{k_2} \cap U^s_C(d)$ }}\label{ss:prespec2} This section presents results and constructions that either constrain our analysis or enable the construction of the irreducible components of interest.

Firstly, within the aforementioned set-up, the subsequent result restricts the speciality of the quotient line bundle $L = \omega_C(-D)$ for such a presentation.

\begin{lemma}\label{specialityhigh} (cf.\,e.g.\,\cite[Lemme\,2.6 and pp.101--102]{L}) Let $C$ be any smooth, irreducible, projective curve of genus $g$. For any integer $2g-2 \leq d \leq 4g-4$ and $i >0$ there is no irreducible component of $B_d^{k_i} \cap U^s_C(d)$ whose general member $[\mathcal F]$ corresponds to  rank-$2$ vector bundle $\Ff$ of speciality $j := h^1(\mathcal F) > i$. 
\end{lemma}
\begin{proof} The reasoning is as follows: if $[\mathcal F] \in B_d^{k_i}\cap U^s_C(d)$ is such that $h^1(\mathcal F) = j > i $
then $h^0(\mathcal F) = d-2g+2+j = k_j > k_i$, so $[\mathcal F] \in {\rm Sing} (B_d^{k_i}\cap U^s_C(d))$ (cf. e.g. \cite[p.\;189]{ACGH}). Therefore the statement follows from \cite[Lemme 2.6]{L}, from which one deduces that no component 
of  $B_d^{k_i} \cap U^s_C(d)$ can be entirely contained or coincide with a component of $B_d^{k_j}\cap U^s_C(d)$, for $j>i$ (the proof is identical to that in \cite[pp.101-102]{L} for $B^0_d$, $1 \leq d \leq 2g-2$, which uses {\em elementary transformations} of vector bundles). 
\end{proof}

Thus, Lemmas \ref{lem:1e2note} and \ref{specialityhigh} imply that a general point $[\mathcal F]$ of any possible irreducible component $\mathcal B$ of the Brill-Noether locus $B_d^{k_2}\cap U^s_C(d) $ corresponds therefore to a rank-$2$ vector bundle $\Ff$ of degree $d$, presented as \eqref{degree} via an effective quotient $L$ s.t. $ 1 \leq h^1(L) \leq i = i (\Ff)=2$. Namely the components $\mathcal B$ of $B_d^{k_2}\cap U^s_C(d)$ will be constructed just in terms of presentations of the bundle $\Ff$ as either of {\em first type} or of {\em second type} in the sense of Definition \ref{def:fstype}, where $L = \omega_C(-D)$.

In what follows we will show in particular that, for any $2g-2 \leq d \leq 4g-4$ as in \eqref{eq:congd}, the case $h^1(L) =1$ (resp., $h^1(L) =2$) only produces either the empty set or the component $\mathcal B_{\rm reg,2}$ (resp., $\mathcal B_{\rm sup,2}$) as in Theorem \ref{thm:i=2}, extending our previous results in \cite{CFK,CFK2} to the whole range $2g-2 \leq d \leq 4g-4$ of interest for $d$ (cf. Proposition \ref{prop:sstabh1} and \eqref{eq:congd} above). Thus, when $B_d^{k_2} \cap U^s_C(d)\neq \emptyset$  it turns out that its irreducible components contain dense open subsets filled up by bundles either of {\em first} or of {\em second type}, as in Definition \ref{def:fstype},  the former case for {\em superabundant} components, the latter for {\em regular} ones.

To do this, recall that for any exact sequence $(u)$ as in \eqref{degree},  if one sets $$H^0(L) \stackrel{\partial_u}{\longrightarrow} H^1(N)$$the associated {\em coboundary map} then, for any integer $t>0$, one can define the corresponding {\em degeneracy locus}: 
\begin{equation}\label{W1}
\mathcal W_t:=\{u\in\ext^1(L,N)\ |\ {\rm corank} (\partial_u)\ge t\}\subseteq \ext^1(L,N), 
\end{equation} which has obviously a natural structure of (locally) determinantal scheme. As such, it has an {\em expected codimension}, which is 
\begin{equation}\label{eq:clrt}
c(t):= {\rm max} \{0,\; t(h^0(L) -h^1(N) +t)\}
\end{equation} (cf. \cite[(5.5)]{CF}). We will remind some details, which are taken from \cite[\S\,5.2]{CF} and which are included here for reader's convenience, to let the paper be as self-contained as possible. 

Put $m := \dim( {\rm Ext}^1(L,N))$. Thus, if $m >0$ and if $\mathcal W_t \neq \emptyset$, any irreducible component $\Lambda_t \subseteq \mathcal W_t$ is such that
\begin{equation}\label{eq:expdimwt} 
\dim(\Lambda_t) \geqslant {\rm min} \left\{m , m- c(t) \right\}, 
\end{equation}where the right-hand-side is therefore the {\em expected dimension} of $\mathcal W_t$. These loci have been considered also in \cite[\S\,2,\,3]{BeFe}, \cite[\S\;6,\,8]{Muk} for low genus and for vector bundles with {\em canonical determinant}.

\begin{remark}\label{rem:wt} {\normalfont The coboundary map $\partial_u$  can be interpreted in terms of multiplication maps  
among global sections of suitable line bundles on $C$. Indeed,  consider $h^1(N) \geqslant t$ and $h^0(L)\geqslant {\rm max} \{1,h^1(N)-t\}$. Denote 
by
$$\cup : H^0(L) \otimes H^1(N - L) \longrightarrow H^1(N),$$the cup-product: for any $u \in H^1(N - L) \cong {\rm Ext}^1(L,N)$, one has $\partial_u (-) = - \cup u.$ 

By Serre duality and by the functorial properties of $\otimes$ and ${\rm Hom}$, considering the map $\cup$ is equivalent to considering the multiplication map
\begin{equation}\label{eq:mu}
\mu := \mu_{L,K_C-N} : H^0(L) \otimes H^0(K_C-N)  \to H^0(K_C+L-N)
\end{equation} (when $N=L$, such a map $\mu$ simply coincides with the rank-$1$ {\em Petri map} $\mu_0(L)$ as defined in \eqref{eq:Petrilb}).  
For any sub-vector space $W \subseteq H^0(K_C-N)$, we set
\begin{equation}\label{eq:muW}
\mu_W:= \mu|_W : H^0(L) \otimes W \to H^0(K_C+L-N).
\end{equation}Imposing $\cork(\partial_u) \geqslant t$ is equivalent to 
$$V_t := {\rm Im}(\partial_u)^{\perp}  \subset H^0(K_C - N) $$ having dimension at least $t$. Therefore one can describe $\mathcal W_t$ also as: 
$$\mathcal W_t = \left\{u \in  H^0(K_C+L-N)^{\vee} \mid 
\exists\,V_t \subseteq H^0(K_C-N),\; {\rm s.t.} \; \dim(V_t) \geqslant t \;{\rm and}\; {\rm Im}(\mu_{V_t}) \subseteq \{u = 0 \}
\right\},$$where $\{u = 0 \} \subset H^0(K+L-N)$ is the hyperplane defined by $u \in H^0(K_C+L-N)^{\vee}$. 
}
\end{remark}

Thus, one has:

\begin{theorem}\label{CF5.8} Let $C$ be a smooth, irreducible, projective curve of genus $g\ge 3$. Let $0 \leq \delta \leq d $ be integers and let $L \in \pic^{\delta}(C)$ be special and effective and $N \in \pic^{d-\delta}(C)$. 
Set $$l := h^0(L),\;\;\; r:= h^1(N)\;\; {\rm and}\;\; m:=\dim(\ext^1(L,N))$$and assume $ r\ge 1,\,l\ge \max\{1,r-1\},\, m \ge l+1$. Then:

\smallskip

\noindent
(i) $\mathcal W_1$ as in \eqref{W1} is irreducible of the {\em expected dimension} 
$m-c(1) = m- (l-r+1)$; 

\smallskip

\noindent
(ii) if $l \geq r$, then $\mathcal W_1 \subsetneq \ext^1(L,N)$ and, for general $u \in  \ext^1(L,N)$,  the associated coboundary map $\partial_u$ is surjective whereas, for general $w \in  \mathcal W_1$, one has  ${\rm corank} (\partial_w)=1$. 
\end{theorem}

\begin{proof} The proof is taken from \cite[Thm.\,5.8, Cor.\,5.9]{CF} and included here, for reader's convenience, to let the paper be as self-contained as possible.

Notice that, by numerical assumptions, $c(1) \geq 0$ and $m - c(1) \geq r$. Let us prove $(i)$. Since $l, r \geqslant 1$, both line bundles $L$ and $K_C-N$ are effective. One has an inclusion $$0 \to L \to K_C+L-N$$obtained by 
tensoring by $L$ the injection $\Oc_C \hookrightarrow K_C-N$ given by a given non--zero  section of $K_C-N$. 
Thus, for any $V_1 \in \Pp(H^0(K_C-N))$, the $\mu_{V_1}$ as in \eqref{eq:muW} is injective. Since $m \geqslant \, l + 1$, one has $\dim({\rm Im} (\mu_{V_1})) = l \leqslant m-1 $, i.e. ${\rm Im}(\mu_{V_1})$ is certainly  contained in some hyperplane of 
$H^0(K_C+L-N)$.  Let  
$$\Sigma:= \left\{\sigma := H^0(L) \otimes V_1^{\sigma} \subseteq H^0(L) \otimes H^0(K_C-N) \; \mid \; V_1^{\sigma} \in \Pp(H^0(K_C-N)) \right\}.$$Thus 
$\Sigma \cong \Pp(H^0(K_C-N))$,  so it is irreducible of dimension $r-1$. Set $\mathbb P:= \Pp(H^0(K_C+L-N)^{\vee})$; we can define the incidence variety 
$$\Jj := \left\{(\sigma, \pi) \in \Sigma \times \Pp \; | \; \mu_{V_1^{\sigma}}(\sigma) \subseteq \pi \right\} \subset \Sigma \times \Pp.$$ Let 
$$ \Sigma \stackrel{pr_1}{\longleftarrow} \Jj \stackrel{pr_2}{\longrightarrow}  \Pp$$be the two  projections. By the above arguments, the map $pr_1$ is surjective. In particular $\Jj \neq \emptyset$ and, for any $\sigma \in \Sigma$, 
$$pr_1^{-1} (\sigma) = \left\{ \pi \in \Pp\,|\,\mu_{V_1^{\sigma}}(\sigma) \subseteq \pi\right\} \cong 
|\Ii_{\widehat{\sigma}\vert \Pp^{\vee}} (1)|,$$where we have denoted 
$\widehat{\sigma} := \Pp(\mu_{V_1^{\sigma}}(\sigma))$, $\Ii_{\widehat{\sigma}\vert \Pp^{\vee}}$ the ideal sheaf of $\widehat{\sigma}$ in $ \Pp^{\vee}$ and by 
$|\Ii_{\widehat{\sigma}\vert \Pp^{\vee}} (1)|$ the complete linear system of hyperplanes 
of  $ \Pp^{\vee}$ passing through the linear subspace $\widehat{\sigma}$.

Since $\dim(\widehat{\sigma}) = l -1$, then $\dim(pr_1^{-1} (\sigma)) = m -1 - l \geqslant 0$.  This shows that  $\Jj$ is irreducible and $\dim(\Jj) = m-1 - c(1) \leqslant m-1$. Then, $$\widehat{\mathcal W}_1 := \Pp(\mathcal W_1) = \overline{pr_2(\Jj)}.$$Recalling \eqref{eq:expdimwt}, $\mathcal{W}_1 \neq \emptyset$ is irreducible, of the expected dimension $m - c(1)$, which completes the proof of $(i)$. 

As for $(ii)$, since $l \geq r$ then, by definition of $c(1)$ as in \eqref{eq:clrt}, one gets $c(1) \geq 1$ which therefore implies $\mathcal W_1 \subsetneq \ext^1(L,N)$. Thus, for general $u \in  \ext^1(L,N)$,  the associated coboundary map $\partial_u$ must be surjective whereas, for general $w \in  \mathcal W_1$, one has  ${\rm corank} (\partial_w)=1$.  
\end{proof}

We conclude this sub-section with the following auxiliary result which will be frequently used in what follows. 

\begin{lemma}\label{lem:technical} Let $C$ be a smooth, irreducible projective curve of genus $g$, let $L$ and $N$ be line bundles on $C$ such that $h^0(N-L) = 0$. Assume further there exist $u, u' \in {\rm Ext}^1(L,N)$ such that: 
 
 \smallskip

\noindent
(i) the corresponding rank-two vector bundles $\Ff_u$ and $\Ff_{u'}$ are stable, and

\smallskip

\noindent
(ii) there exists an isomorphism $\varphi$ 
\[
\begin{array}{ccccccl}
0 \to & N & \stackrel{\iota_1}{\longrightarrow} & \Ff_{u'} & \to & L & \to 0 \\ 
 & & & \downarrow^{\varphi} & & &  \\
0 \to & N & \stackrel{\iota_2}{\longrightarrow} & \Ff_u & \to & L & \to 0
\end{array}
\]such that $\varphi \circ \iota_1 = \lambda \iota_2$, for some $\lambda \in \mathbb C^*$. 

Then $\Ff_u = \Ff_{u'}$, i.e.  $u, u'$ are proportional vectors in ${\rm Ext}^1(L,N)$. 
\end{lemma}

\begin{proof} The proof has been inspired by that in \cite[Lemma 1]{Ma3}. If $\{ U_i\}_{1 \leqslant i \leqslant n}$ is a sufficiently fine open covering of $C$, on any $U_i$ we can choose local coordinates 
$$\left(\begin{array}{c}
 u_i \\
 v_i 
 \end{array}\right) \; \mbox{for} \; \Ff_{u}|_{U_i}\; \mbox{and} \; \left(\begin{array}{c}
 u'_i \\
 v'_i 
 \end{array}\right) \; \mbox{for} \; \Ff_{u'}|_{U_i}, \; 1 \leqslant i \leqslant n,$$such that: 
 
 \smallskip 
 
 \noindent
 $\bullet$ the transition matrices on $U_i \cap U_j$ are
 $$\left(\begin{array}{cc}
 a_{ij} & c_{ij} \\
 0 & b_{ij}  
 \end{array}\right)  \;  \mbox{for} \;\; \Ff_{u}|_{U_i} \;\;\;\; \mbox{and} \;\;\;\; \left(\begin{array}{cc}
 a'_{ij} & c'_{ij} \\
 0 & b'_{ij}  
 \end{array}\right) \; \mbox{for} \;\; \Ff_{u'}|_{U_i},$$where 
 $a_{ij}, a'_{ij}, b_{ij}, b'_{ij} \in \Oc^*_C(U_i \cap U_j)$, $c_{ij}, c'_{ij}\in \Oc_C(U_i \cap U_j)$, for
 $1 \leqslant i, j \leqslant n$,

 \smallskip

 \noindent
 $\bullet$ $N$ is defined by 
 $$ v_i = 0 \;\;  \mbox{for} \;\; \Ff_u|_{U_i}, \;\; \mbox{and by} \;\; v'_i = 0 \;\;  \mbox{for} \; \Ff_{u'}|_{U_i}$$and its transition functions on $U_i \cap U_j$ 
are given by 
$$ a_{ij} \;  \mbox{when} \; N \subset \Ff_u \; \mbox{and by} \; a'_{ij} \;  \mbox{when} \; N \subset \Ff_{u'}, \;  \mbox{for } 1 \leqslant i, j \leqslant n.$$

 \noindent
 In the above setting, transition functions on $U_i \cap U_j$  for $L$ are given by 
 $$ b_{ij} \;  \mbox{when} \; \Ff_u \to \!\! \to L \; \mbox{and by} \; b'_{ij} \;  \mbox{for} \; 
 \; \Ff_{u'}\to \!\! \to L,\;  \mbox{for } \;  1 \leqslant i, j \leqslant n.$$ 
 The map $\varphi$ induces isomorphisms 
 $$\varphi_i := \varphi|_{U_i} : \Ff_{u'}|_{U_i} \stackrel{\cong}{\longrightarrow} \Ff_{u}|_{U_i}, \; \mbox{for any} \;\; 1 \leqslant i \leqslant n.$$By $(ii)$, one has 
 \begin{equation}\label{eq:fi}
 \varphi_i = \left(\begin{array}{cc}
 \lambda & \gamma_i \\
 0 & \beta_i  
 \end{array}\right), \; 1 \leqslant i \leqslant n,
 \end{equation}where $\lambda \in \mathbb C^*$, $\beta_i \in \Oc^*_{C}(U_i)$, $\gamma_i \in \Oc_{C}(U_i)$. Moreover, any $\varphi_i$ has to commute with the transition matrices, i.e. 
$$\left(\begin{array}{cc}
 a_{ij} & c_{ij} \\
 0 & b_{ij}  
 \end{array}\right)\left(\begin{array}{cc}
 \lambda & \gamma_j \\
 0 & \beta_j  
 \end{array}\right) 
 = \left(\begin{array}{cc}
 \lambda & \gamma_i \\
 0 & \beta_i  
 \end{array}\right) \left(\begin{array}{cc}
 a'_{ij} & c'_{ij} \\
 0 & b'_{ij}  
 \end{array}\right), \;\;  \mbox{for}  \;  1 \leqslant i,j \leqslant n,$$ which  read 
 \begin{equation}\label{eq:1}
a_{ij} = a'_{ij}, \;  b'_{ij} = \frac{\beta_j}{\beta_i} b_{ij}, \; \mbox{and} \; \gamma_{j}  a_{ij} -  \gamma_i \frac{\beta_j}{\beta_i} b_{ij} =  \lambda c'_{ij} - \beta_j c_{ij}, 
\;\;  \mbox{for}  \;  1 \leqslant i, j \leqslant n.
\end{equation} 
The second equality in \eqref{eq:1} gives \begin{equation}\label{eq:2} b'_{ij} = \frac{\beta_j}{\beta_i} b_{ij}, \; \mbox{for} \; 1 \leqslant i, j \leqslant n;
\end{equation}correspondingly, the third equality in \eqref{eq:1} becomes  $$\gamma_{j}  a_{ij} -  \gamma_i \frac{\beta_j}{\beta_i} b_{ij} =  \lambda c'_{ij} - \beta_j c_{ij}, \; \mbox{for} \; 1 \leqslant i, j \leqslant n.$$
 Since  $b_{ij}, b'_{ij} \in \Oc^*_C(U_i \cap U_j)$ are both transition functions for $L$, the second equality in
\eqref{eq:1} implies that, on $U_i \cap U_j$, $\beta_i$ and $\beta_j$ differ by a coboundary, i.e. there exist 
$h_i \in  \Oc_C^*(U_i)$, for $1 \leqslant i \leqslant n$, such that $\frac{\beta_j}{\beta_i} = \frac{h_j}{h_i}$ on $U_i \cap U_j$, for $ 1 \leqslant i,  j \leqslant n$. 
Therefore, $$\frac{\beta_i}{h_i} \in \Oc_C^*(U_i) \;\; \mbox{and} \;\; 
\frac{\beta_i}{h_i} = \frac{\beta_j}{h_j} \in \Oc_C^*(U_i \cap U_j), \; \mbox{for}  \; 1 \leqslant i, j \leqslant n,$$i.e. 
\begin{equation}\label{eq:scalar}
\frac{\beta_i}{h_i} = \mu \in \C^*, \; \mbox{for any}  \; 1 \leqslant i \leqslant n.
\end{equation} Make the following change of local coordinates on $\Ff_{u'}|_{U_i}$ 
$$u'_i = x'_i, \;\; v_i' = \frac{1}{h_i} y'_i, \;\mbox{for any}\; 1 \leqslant i \leqslant n.$$
In these coordinates, one  has that: 

\noindent
$\bullet$ from \eqref{eq:fi} and \eqref{eq:scalar}, the representation of $\varphi_i$  becomes   
\begin{equation}\label{eq:fibis}
\varphi_i = \left(\begin{array}{cc}
 \lambda & \widetilde{\gamma}_i \\
 0 & \mu 
 \end{array}\right), \;\mbox{for}  \;  1 \leqslant i \leqslant n
 \end{equation}where $\widetilde{\gamma}_i := \frac{\gamma_i}{h_i} \in \Oc_{C}(U_i)$; 
 
 \noindent
 $\bullet$ the transition matrices for $\Ff_{u'}$ on $U_i \cap U_j$ become 
 $$\left(\begin{array}{cc}
 a'_{ij} & \widetilde{c'}_{ij} \\
 0 & \widetilde{b}'_{ij}  
 \end{array}\right), \; \mbox{for}  \;  1 \leqslant i,j\leqslant n$$ 
 where $ \widetilde{c'}_{ij} := \frac{c'_{ij}}{h_i} \in \Oc_C(U_i \cap U_j)$, 
 $\widetilde{b'}_{ij} := b'_{ij} \frac{h_i}{h_j} \in \Oc_C^*(U_i \cap U_j)$;

\noindent
$\bullet$  the compatibility conditions as in \eqref{eq:1} become
\begin{equation}\label{eq:*}
a_{ij} = a'_{ij}, \;\; b_{ij} = \widetilde{b'}_{ij} \; \; {\rm and} \;\; \widetilde{\gamma}_{j} a_{ij}  - \widetilde{\gamma}_i b_{ij} = \lambda \widetilde{c'}_{ij} - \mu c_{ij}, \;  1 \leqslant i, j \leqslant n.
\end{equation}For the third equality in \eqref{eq:*}, two cases have to be discussed.

\noindent
(a) Assume first $\lambda \widetilde{c'}_{ij} - \mu c_{ij} = 0$. Thus,  
$$\widetilde{\gamma}_i = \frac{a_{ij}}{b_{ij}} \widetilde{\gamma}_j, \; 
 \mbox{for} \; 1 \leqslant i, j \leqslant n,$$i.e. the collection $\{ \widetilde{\gamma}_i\}_{i\leqslant i \leqslant n}$ defines a global section of $N-L$. Since 
 $h^0(N-L) = 0$, from \eqref{eq:fibis} one has 
 $$ \varphi = \left(\begin{array}{cc}
 \lambda & 0\\
 0 & \mu  
 \end{array}\right), \; \lambda, \mu \in \mathbb C^*.$$If $\lambda= \mu$, then $\varphi = \lambda\,{\rm Id}$ so $\Ff_u = \Ff_{u'}$ as vector bundles, 
 $\varphi \in {\rm End}(\Ff_u)$ and  $u, u' \in {\rm Ext}^1(L,N)$ are proportional. 
 
 Next we exclude that  $\lambda \neq \mu$. Indeed, in this case we  get the diagram
 \[
\begin{array}{ccccccl}
0 \to & N \otimes L^{\vee} & \longrightarrow & \Ff_{u'} \otimes L^{\vee} & \to & \Oc_C & \to 0 \\ 
 & \downarrow^{\lambda \mu^{-1}}& & \downarrow^{\widehat{\varphi}} & & \downarrow^{id} &  \\
 0 \to & N \otimes L^{\vee} & \longrightarrow & \Ff_u \otimes L^{\vee} & \to & \Oc_C & \to 0,
 \end{array}
\]where $\widehat{\varphi} = \mu^{-1} \varphi$. Taking the coboundary maps 
$$\widehat{\partial}_{u'} : H^0(\Oc_C) \to H^1(N-L) \;\; \mbox{and} \;\; \widehat{\partial}_u : H^0(\Oc_C) \to H^1(N-L),$$ we  
get $$\widehat{\partial}_{u'}(1) = (\lambda \mu^{-1})  \widehat{\partial}_u(1),$$which  implies again that
$u, u' \in {\rm Ext}^1(L,N)$ are proportional vectors, i.e. $\Ff_u = \Ff_{u'}$ as vector bundles. In this case, 
$\varphi \in {\rm End}(\Ff_u) \setminus \mathbb C^*$, contradicting assumption $(i)$ ($\Ff_u$ stable implies that $\Ff_u$ is simple, cf. e.g. \cite[Prop. 6-c, p.17]{Ses}).     So the case  $\lambda \neq \mu$ cannot occur, and we are done. 

\medskip

\noindent
(b) Assume now $\lambda \widetilde{c'}_{ij} - \mu c_{ij} \neq 0$. In this case, we argue as in \cite[Lemma 1]{Ma3}. Indeed, 
for $1 \leqslant i \leqslant n$, we consider the following further change of coordinates 
\[\left(\begin{array}{c}
 x'_i \\
 y'_i 
 \end{array}\right) := \left(\begin{array}{cc}
 1 & - \frac{2\widetilde{\gamma}_i}{\lambda} \\
 0 & 1 
 \end{array}\right) \left(\begin{array}{c}
 \xi'_i \\
 \eta'_i 
 \end{array}\right) \;\; \mbox{for} \;\; \Ff_{u'}|_{U_i}\;\; {\rm and} \;\; 
 \left(\begin{array}{c}
 u_i \\
 v_i 
 \end{array}\right) := \left(\begin{array}{cc}
 1 & - \frac{\widetilde{\gamma}_i}{\mu} \\
 0 & 1 
 \end{array}\right) \left(\begin{array}{c}
 \xi_i \\
 \eta_i 
 \end{array}\right) \; \mbox{for} \;\; \Ff_u|_{U_i}.
\] Taking into account  \eqref{eq:*}, in these  coordinates the transition matrices become
$$\left(\begin{array}{c}  \xi'_i \\  \eta'_i 
 \end{array}\right) 
 = \left(\begin{array}{cc}
a_{ij} & \frac{1}{\lambda} (2 \mu c_{ij} - \lambda \widetilde{c'}_{ij}) \\
 0 & b_{ij}  
 \end{array}\right) \left(\begin{array}{c}
 \xi'_j \\
 \eta'_j 
 \end{array}\right), \; \mbox{for} \;\; \Ff_{u'}|_{U_i}$$
 and 
 $$\left(\begin{array}{c}
 \xi_i \\
 \eta_i 
 \end{array}\right) =
  \left(\begin{array}{cc}
 a_{ij} & \frac{1}{\mu} (2 \mu c_{ij} - \lambda \widetilde{c'}_{ij}) \\
 0 & b_{ij}  
 \end{array}\right) \left(\begin{array}{c}
 \xi_j \\
 \eta_j 
 \end{array}\right), \; \mbox{for} \;\; \Ff_u|_{U_i}$$whereas $\varphi_i$ reads as $$\left(\begin{array}{cc}
 \lambda & 0 \\
 0 & \mu 
 \end{array}\right).$$One can then conclude as 
 in case (a) above.  
 \end{proof}

\medskip

\subsection{{\bf Regular components in speciality 2}}\label{ss:reg2} In this subsection we prove that, for any $2g-2 \leq d \leq 4g-4$ as in \eqref{eq:congd}, $B_d^{k_2}  \cap U_C^s(d)$ is either empty or it contains only one irreducible component whose general point $[\Ff]$ correspond to a rank-$2$ stable bundle $\mathcal F$ of degree $d$ and of {\em second type} in the sense of Definition \ref{def:fstype}. 

When not empty, such a component will turn out to be exactly the component $\mathcal B_{\rm reg,2}$ as in Theorem \ref{thm:i=2}; we moreover deduce several properties of the bundle $\Ff$ corresponding to $[\Ff] \in \mathcal B_{\rm reg,2}$ general  (e.g., speciality, bounds on its Segre invariant, minimal presentation of $\Ff$ as an extension \eqref{degree}, etcetera)  as well as information on the {\em birational structure} of the component $\mathcal B_{\rm reg,2}$.  For values of $d$ near the upper-bound $4g-4$, some of these properties have already been (partially) proved in \cite{CFK}. 


\bigskip

\bigskip

\noindent
$\boxed{4g-6 \leq d \leq 4g-4}$ In this range the naive count $\rho_d^{k_2} = 8g - 11-2d \geq 0$ (which holds true if and only if $d \leq 4g - \frac{11}{2}$) suggests that $B_d^{k_2}  \cap U_C^s(d) = \emptyset$ and this is exactly what occurs as stated in Theorem \ref{thm:i=2}-(i).

\begin{proposition}\label{prop:caso1spec2} Let $C$ be any smooth, irreducible, projective curve of genus $g\geq 4$. Then, for any integer $4g-6 \leq d \leq 4g-4$ one has $B_d^{k_2}  \cap U_C^s(d) = \emptyset$. 
\end{proposition}

\begin{proof} For $4g-5 \leq d \leq 4g-4$, if by contradiction we had $B_d^{k_2}  \cap U_C^s(d) \neq \emptyset$, then $[\Ff]$ general in any component $\mathcal B \subseteq B_d^{k_3}  \cap U_C^s(d)$ has speciality exactly $i=2$, 
from Lemma \ref{specialityhigh}. On the other hand, by refinements of Clifford's Theorem as in \cite[Propositions\,2,3,4]{Re}, we would have also $h^0(\Ff) \leq \frac{d}{2} +1$. Therefore, we would get $h^0(\Ff) = k_2 = d - 2 g +4 \leq \frac{d}{2} +1$ which would give $d \leq 4g-6$, against the assumptions.

When otherwise $d = 4g-6$, if once again by contradiction we had $B_d^{k_2}  \cap U_C^s(d) \neq \emptyset$ and if $[\Ff]$ were general in any component $\mathcal B \subseteq B_d^{k_3}  \cap U_C^s(d)$, we could consider the stable and effective residual bundle $\omega_C \otimes \Ff^{\vee}$, which is of degree $2$. Then one concludes with same arguments as in \cite[end of page 123]{Teixidor1}. 
\end{proof}

\bigskip

\bigskip

\noindent
$\boxed{4g-4-2\nu \leq d \leq 4g-7}$ This range for $d$ was partially considered in \cite[Rem.\;3.14-(ii)]{CFK} and \cite[Rem.\;3.1]{CFK2} where, as a consequence 
of  Teixidor I Bigas' Theorem  Theorem \ref{TeixidorRes}-$(i)$, we observed that $B_d^{k_2}  \cap U_C^s(d)$ is irreducible, consisting only of the component $\mathcal B_{\rm reg,2}$ of the expected dimension $\rho_d^{k_2} = 8g -11-2d$ which coincides with the component 
$\mathcal B$ mentioned in Theorem \ref{TeixidorRes}, existing for any smooth, irreducible, projective curve $C$ of genus $g$. 

Here we strongly improve our results in \cite{CFK,CFK2}. Indeed, first of all we show the {\em regularity} of such a component when $C$ is a general $\nu$-gonal curve (recall that Theorem \ref{TeixidorRes} shows regularity of her component $\mathcal B = \mathcal B_{\rm reg,2}$ only when $C$ is with {\em general moduli}); moreover we prove further properties concerning the {\em birational structure} of $\mathcal B_{\rm reg,2}$ as well as a suitable {\em parametric representation}  of such a component. Precisely, we prove (cf. Theorem \ref{thm:i=2}-(ii) in Introduction) the following:

\begin{proposition}\label{prop:caso2spec2} Let $g \geq 4$, $3 \leq \nu < \lfloor \frac{g+3}{2} \rfloor$ and $4g-4-2\nu \leq d \leq 4g-7$ be integers and let $C$ be a general $\nu$-gonal curve of genus $g$. Then one has:  

\begin{itemize}
\item[(i)] $B_d^{k_2}  \cap U_C^s(d)$ is {\em irreducible}, consisting only of one  component, denoted $\mathcal B_{\rm reg,2}$, which is of the expected dimension $\rho_d^{k_2} = 8g-11-2d$; 
\item[(ii)] $\mathcal B_{\rm reg,2}$ is {\em uniruled} and generically smooth, so in particular it is {\em regular}; 
\item[(iii)] if $[\Ff] \in \mathcal B_{\rm reg,2}$ denotes a general point, the corresponding rank-$2$, degree-$d$ vector bundle $\Ff$ on $C$ is {\em of second type} (as in Definition \ref{def:fstype}), arising as an extension of line bundles of the following type: 

\smallskip

\noindent
$(iii-1)$ for either $4g-4-2\nu \leq d \leq 4g-7$, when $3 \leq \nu \leq \frac{g}{2}$, or for $3g-4 \leq d \leq 4g-7$, when $\frac{g+1}{2}\leq \nu < \lfloor\frac{g+3}{2} \rfloor$, the bundle $\Ff$ fits in an exact sequence of the form: 
$$0 \to \omega_{C}(-D) \to \Ff \to \omega_C (-p) \to 0,$$where 
$D \in C^{(4g-5-d)}$  and $p \in C$ are general; 

\smallskip

\noindent
$(iii-2)$  for $4g-4-2\nu \leq d\leq 3g-5$, when $\frac{g+1}{2}\leq \nu < \lfloor\frac{g+3}{2} \rfloor$, the bundle $\Ff$ fits in an exact sequence of the form: $$0 \to N \to \Ff \to \omega_C (-p) \to 0,$$where $p \in C$ is general and  $N \in {\rm Pic}^{d-2g+3}(C)$ is general, i.e. special and non-effective. 

\end{itemize}
\end{proposition}

\begin{proof} (i) This is an easy  consequence of Theorem \ref{TeixidorRes}-$(i)$, where for any smooth, irreducible, projective curve $C$ of genus $g$, it is proved the existence of a component $\mathcal B$, of expected dimension $\rho_d^{k_2}$. Moreover conditions therein, guaranteeing reducibility of $B_d^{k_2}  \cap U_C^s(d)$, are:  
$$2n < 4g-4-d, \;\; W^1_n (C) \neq \emptyset \;\; {\rm and} \; \; \dim(W^1_n (C)) \geq 2g + 2n - d - 5.$$Since on a general $\nu$-gonal curve one must have $\nu \leq n$ (cf. e.g. \cite{AC}), then $2 \nu \leq 2 n < 4g-4-d$ forces $d \leq 4g-5-2\nu$, which explains why $B_d^{k_2} \cap U^s_C(d)$ must be irreducible for $d \geq 4g-4-2\nu$.  

\smallskip

\noindent
(ii)-(iii) In order to give an explicit parametric representation of the component $\mathcal B_{\rm reg,2}$, as in \cite[Proofs of $(1)$ and of Claim 2]{Teixidor1} related to the ``Serre dual" version of these cases, the bundle $\Ff$ corresponding to a general point in $\mathcal B_{\rm reg,2}$ is proved to have a {\em presentation} of the form $$0 \to N \to \Ff \to \omega_C \to 0,$$where $N$ of degree $2g-2-2\nu \leq \deg (N) = d - 2g+2 \leq 2g-5$, necessarily special, as $i(\Ff)=2$. We will use this fact when $C$ is a general $\nu$-gonal curve to prove (ii) and (iii). 

Notice that, in the above exact sequence, one may take the kernel line bundle $N$ as $N := K_C-D_t$, with $D_{t} \in C^{(t)}$ a general effective divisor such that $t = 4g-4-d$, $h^1(K_C-D_t) =1$ and $h^0(K_C-D_t) = g-t \geq 1$, if and only if $1 \leq t \leq g-1$. This is therefore possible if $t = 4g-4-d \leq g-1$, namely $d \geq 3g-3$. When otherwise $d=3g-4$ and $N$, of degree $d-2g-2$, is general of its degree then $\deg (N)= d-2g+2 = g-2$ so $h^1(N)=1$ by generality, which gives $N=K_C-D_{g}$, with $D_g \in C^{(g)}$ general. On the other hand, since $d \geq 4g-4-2\nu$, notice that $4g-4-2\nu \geq 3g-3$ iff $3 \leq \nu \leq \frac{g-1}{2}$.

To sum-up, to study the cases as in $(iii-1)$ in the statement, we take 
$N := K_C-D_{4g-4-d}$, where  $D_{4g-4-d} \in C^{(4g-4-d)}$ is general (cases as in $(iii-2)$ will be treated later on). Thus \linebreak $h^0(N) = h^0(K_C-D_{4g-4-d}) = d+4-3g \geq 1$. Using notation as in  Theorem \ref{CF5.8}, one has $r:= h^1(N) = h^1(K_C-D_{4g-4-d}) = 1$. Moreover $m:= \dim(\ext^1(K_C, K_C-D_{4g-4-d})) = h^0(K_C+D_{4g-4-d}) = 5g-5-d$. 

Notice that the degeneracy locus $\mathcal W_1$ as in \eqref{W1} is such that 
$\mathcal W_1 \neq \emptyset$, as one obviously has \linebreak $\omega_C \oplus \omega_C(-D_{4g-4-d}) \in \mathcal W_1$, and also that $u \in \mathcal W_1$ if and only if 
$\partial_u = 0$, since $h^1(K_C-D_{4g-4-d}) = 1$. Recalling  from above that the coboundary maps $\partial$ are induced by the cup-product 
$$\cup: H^0(K_C) \otimes H^1(-D_{4g-4-d}) \to H^1(K_C-D_{4g-4-d}),$$we have the natural induced morphism 
\[
\begin{array}{ccc}
\ext^1(K_C, K_C-D_{4g-4-d}) \cong H^1(-D_{4g-4-d}) & \stackrel{\Phi}{\longrightarrow} & {\rm Hom} \left(H^0(K_C),\; H^1(K_C-D_{4g-4-d})\right) \\
u & \longrightarrow & \partial_u
\end{array}
\]which shows that $\mathcal W_1 \subset \ext^1(K_C, K_C-D_{4g-4-d})$ turns out to be 
\begin{equation}\label{eq:vaff931}
{\small {\mathcal W}_1 = \left(\coker \left(H^0(K_C+D_{4g-4-d}) \stackrel{\Phi^{\vee}}{\leftarrow} H^0(K_C) \otimes H^0(D_{4g-4-d})\right)\right)^{\vee},}
\end{equation} namely $\mathcal W_1 = \left(\im \; \Phi^{\vee}\right)^{\perp}$ is a sub-vector space of $\ext^1(K_C, K_C-D_{4g-4-d})$.  Since 
$h^0(D_{4g-4-d}) =1$, the map $\Phi^{\vee}$ is injective and so $\mathcal W_1$ has codimension $g$, namely $\dim(\mathcal W_1) = m-g = 4g-5-d \geq 2$; moreover, for $u \in \mathcal W_1$ general, the corresponding rank-two vector bundle $\Ff_u$ has speciality $2$.

We let $\widehat{\mathcal W}_1 := \Pp(\mathcal W_1)$ denote the corresponding linear sub-space in $\Pp := \Pp(\ext^1(K_C, K_C-D_t))$. The stability of the bundle $\Ff_u$, for $D_{4g-4-d} \in C^{(4g-4-d)}$ and $[u] \in \widehat{\mathcal W}_1 $ general, can be proved using its ``Serre dual version" 
$\mathcal E_u := \omega_C \otimes \Ff_u^{\vee}$ as in \cite[proof of Claim\,2]{Teixidor1}. Since for $D_{4g-4-d}$ varying in a suitable dense subset $S \subseteq C^{(4g-4-d)}$ each corresponding $\widehat{\mathcal W}_1$ is a linear space of (constant) dimension $4g-6-d$, these linear spaces fill-up an irreducible scheme, denoted by $\widehat{\mathcal W}_1^{Tot}$, which is ruled over $S$ by the linear spaces $\widehat{\mathcal W}_1$'s as $D_{4g-4-d}$ varies in $S$ and which is of dimension $$\dim(\widehat{\mathcal W}_1^{Tot}) = (4g-4-d) + (4g-6-d) = 8g-10-2d.$$By stability and speciality $2$ of the general bundle, $\widehat{\mathcal W}_1^{Tot}$ is endowed with a natural (rational) modular map 
\[
 \begin{array}{ccc}
 \widehat{\mathcal W}_1^{Tot}& \stackrel{\pi}{\dashrightarrow} & U^s_C(d) \\
 (D_{4g-4-d}, [u]) & \longrightarrow &[\mathcal F_u]
 \end{array}
 \] for which $\im(\pi) \subseteq B^{k_2}_d \cap U_C^s(d)$. 
 
\begin{claim}\label{cl:20feb816} The general fiber of $\pi$ is $1$-dimensional, isomorphic to $\Pp^1$.
\end{claim}
\begin{proof}[proof of Claim \ref{cl:20feb816}] Take $[\Ff] \in \im(\pi)$ general, so $\Ff= \Ff_u$ for 
$(D_{4g-4-d}, [u]) \in  \widehat{\mathcal W}_1^{Tot}$ general. Then one has that 
$$\pi^{-1}([\Ff_u]) = \left\{ (D_{4g-4-d}', u') \in  \widehat{\mathcal W}_1^{Tot} \, | \; \Ff_{u'} \simeq \Ff_u \right\}.$$Since $2 K_C - D_{4g-4-d} = \det(\Ff_u) \simeq \det(\Ff_{u'}) = 2 K_C- D'_{4g-4-d}$ and since both bundles are presented via the canonical quotient, one gets that $D_{4g-4-d} = D'_{4g-4-d}$. 

For $[u'] \in \pi^{-1}([\Ff_u]) \setminus \{[u]\}$, taking into account that 
$\Ff_u \simeq \Ff_{u'}$ and that both of them are stable (so simple), there exists therefore an isomorphism $\varphi : \Ff_{u'} {\to} \Ff_u$ (which is not an endomorphism) giving rise to the following diagram: 
\[
\begin{array}{ccccccl}
0 \to & K_C-D_{4g-4-d} & \stackrel{\iota_1}{\longrightarrow} & \Ff_{u'} & \to & K_C & \to 0 \\ 
 & & & \downarrow^{\varphi} & & &  \\
0 \to & K_C-D_{4g-4-d} & \stackrel{\iota_2}{\longrightarrow} & \Ff_u & \to & K_C & \to 0.
\end{array}
\]Thus the maps $\varphi \circ \iota_1$ and $\iota_2$ determine two non-zero global sections $ s_1 \neq s_2 \in H^0(\Ff_u (-K_C+D_{4g-4-d})) $. 

If these sections are linearly independent, they give rise to two non-proportional independent ways to inject $K_C-D_{4g-4-d}$ as a sub-line bundle of $\Ff_u$; in other words the fiber $\pi^{-1}([\Ff_u])$ is determined by the family of possible injections (up to scalar multiplication) of $K_C-D_{4g-4-d}$ as a sub-line bundle of $\Ff_u$. Thus, if $F_u = \mathbb P(\Ff_u)$ denotes the surface scroll associated to $\Ff_u$ and if $\Gamma$ denotes a canonical section of $F_u$ corresponding to the canonical quotient of $\Ff_u$, from \eqref{eq:isom2} one has therefore $\pi^{-1}([\Ff_u]) = \Pp(H^0(\Ff(-K_C+D_{4g-4-d})) \simeq |\mathcal O_{F_u} (\Gamma)|$. From the fact that $\Ff_u$ is of rank two, with $\det(\Ff_u) = 2K_C - D_t$ and $h^1(\Ff_u)=2$, by Serre duality one has 
$\dim(|\mathcal O_{F_u} (\Gamma)|) = h^0(\Ff_u (-K_C + D_{4g-4-d})) -1 = h^1(\Ff_u) -1 = 1$. 

If otherwise the two sections were linearly dependent, since $\deg(N-L) = d - 2 \delta <0$, 
one has $h^0(N-L) = 0$ and we would get a contradiction by Lemma \ref{lem:technical}. 
\end{proof}

From Claim \ref{cl:20feb816}, we have that$$\dim(\im(\pi)) = 8g-11-2d = \rho_d^{k_2}$$so it is dense in the unique component $\mathcal B_{\rm reg,2}$, namely the component $\mathcal B$ as in Theorem \ref{TeixidorRes}-(i) existing for any curve $C$, which is of expected dimension and also the only component by (i). 

Since $\im(\pi)$ is dominated by $\widehat{\mathcal W}_1^{Tot}$, which is {\em ruled}  over $S$ by the linear spaces $\widehat{\mathcal W}_1$'s, it is clear that $ \im(\pi)$ and so also $\mathcal B_{\rm reg,2}$ are {\em uniruled}. Moreover, since the general fiber of $\pi$ is isomorphic to $\Pp^1$ by Claim \ref{cl:20feb816}, it also follows that the bundle $\Ff_u$ corresponding to $[\Ff_u] \in \mathcal B_{\rm reg,2}$ general is not rigidly specially presented (not rsp), as in Definition \ref{def:ass1}, via the exact sequence 
$$0 \to \omega_{C}(-D_{4g-4-d}) \to \Ff_u \to \omega_C \to 0.$$To prove the generic smoothness  of $\mathcal B_{\rm reg,2}$, we need to study the Petri map $\mu_{\Ff}$ for  general $[\mathcal F ]\in \im(\pi)$. 

To do this, one can apply arguments similar to those in \cite[Prop.\,3.9]{CFK}. Indeed, from the fact that $\mathcal F = 
\mathcal F_u$, for some $[u]$ in some fiber $\widehat{\mathcal W}_1$ over $S$ of $\widehat{\mathcal W}_1^{Tot}$, it follows that the corresponding coboundary map $\partial_u$ is the zero-map; in other words 
$$H^0(\mathcal F) \cong H^0(K_C-D_{4g-4-d}) \oplus H^0(K_C) \;\;\;{\rm and} \;\;\; H^1(\mathcal F) \cong  H^1(K_C-D_{4g-4-d}) \oplus H^1(K_C).$$This means that, for any such bundle, the domain of the Petri map $\mu_{\mathcal F}$ coincides with that 
of $\mu_{{\mathcal F}_0}$, where  the bundle $\mathcal F_0 := (K_C-D) \oplus (K_C)$ corresponds to the zero vector in $\mathcal W_1 \subset \ext^1(K_C, K_C-D_{4g-4-d})$. We will concentrate in computing $\mu_{{\mathcal F}_0}$. To ease notation, we will set
$D:= D_{4g-4-d}$.  Observe that 
\[\begin{array}{ccl}
H^0(\mathcal F_0) \otimes H^0(\omega_C \otimes \mathcal F_0^{\vee}) & \cong &  \left(H^0(K_C-D) \otimes H^0(D)\right)  \oplus \left(H^0(K_C-D) \otimes H^0(\mathcal O_C)\right) \oplus\\
& & \left(H^0(K_C) \otimes H^0(D)\right) \oplus \left(H^0(K_C) \otimes H^0(\mathcal O_C)\right). 
\end{array}
\]Moreover we have $\omega_C \otimes \mathcal F_0 \otimes \mathcal F_0^{\vee} \cong K_C \oplus (K_C-D) \oplus (K_C+D) \oplus K_C$. Therefore, for Chern classes reason,  
$$\mu_{\mathcal F_0} = \mu_{D} \oplus \mu_{K_C-D, \mathcal O_C} \oplus \mu_{K_C, D} \oplus \mu_{K_C, \mathcal O_C}$$where the maps 
\begin{eqnarray*}
\mu_{D}: H^0(D)\otimes H^0(K_C-D)\to H^0(K_C) & {\rm and} & \mu_{K_C-D,\mathcal O_c}: H^0(K_C-D)\otimes H^0(\mathcal O_C)\to H^0(K_C-D) \\
 \mu_{K_C,D}: H^0(K_C)\otimes H^0(D)\to H^0(K_C+D) & {\rm and} &
  \mu_{K_C, \mathcal O_C}: H^0(K_C)\otimes H^0(\mathcal O_C)\to H^0(K_C)
 \end{eqnarray*}are natural multiplication maps. Since $h^0(D)=1$, all the maps are injective, so it is $\mu_{\mathcal F_0}$. By semicontinuity on $\mathcal W_1$, it then follows that also $\mu_{\mathcal F}$ is injective, proving that $\mathcal B_{\rm reg,2}$ is {\em regular} and {\em uniruled} in the above cases.

We now focus on the left cases, namely   
$$4g-4-2\nu \leq d \leq 3g-5,\;{\rm when} \; \frac{g+1}{2}\leq \nu < \lfloor\frac{g+3}{2} \rfloor.$$In this situation we take the kernel line bundle $N \in {\rm Pic}^{d-2g+2}(C)$ general of its degree; from the bounds on $d$ one has $g-4 \leq 2g-2-2\nu \leq \deg(N) \leq g-2$, so $N$ general is non-effective with speciality $h^1(N) := r \in \{1,2,3\}$. More precisely, we have 
\[
\left\{\begin{array}{cccc}
{\rm for} \; \nu = \frac{g+1}{2} & & 3g-5= 4g-4-2\nu\leq d \leq 3g-4 \;\; \mbox{and} & 1 \leq r \leq 2 \\
{\rm for} \; \nu = \frac{g+2}{2} & &  3g-6= 4g-4-2\nu\leq d \leq 3g-4 \;\; \mbox{and} & 1 \leq r \leq 3
\end{array}
\right.
\] On the other hand, the case $ \nu = \frac{g+2}{2}$ cannot occur as $g$, being necessarily even, gives $\frac{g+2}{2} = \lfloor \frac{g+3}{2} \rfloor$ and, by assumption we have $\nu <  \lfloor \frac{g+3}{2} \rfloor$. Moreover, when $\nu = \frac{g+1}{2}$, then $3g-5= 4g-4-2\nu\leq d \leq 3g-4$, neverthless the case $d=3g-4$ (so $\deg(N) = g-2$) has been already studied in cases $(iii-1)$. 

Therefore we are left with $d = 3g-5$, with $N$ general of degree $\deg(N) = g-3$ as in $(iii-2)$, and by generality $r : = h^1(N) = 2$. Using notation as in Theorem \ref{CF5.8}, one has $l:= h^0(K_C) = g$, $m := 5g-5-d$ and it is easy to observe that all the numerical assumptions therein are satisfied. Therefore, $\mathcal W_1 \subset \ext^1(K_C, N)$ is irreducible of dimension 
$\dim(\mathcal W_1) = 7g-9-d$ and, for general $u \in \mathcal W_1$, the corresponding 
coboundary map  $\partial_u$ has corank $1$, so $h^1(\Ff_u)=2$. 

As above, we let $\widehat{\mathcal W}_1$ be the projectification $\Pp(\mathcal W_1)$ in $\Pp := \Pp(\ext^1(\omega_C, N))$ and, for $N$ varying in a suitable open dense subset $S \subseteq {\rm Pic}^{d-2g+2}$, these  $\widehat{\mathcal W}_1$'s are all of (constant) dimension $7g-10-2d$. Thus, they fill-up an irreducible scheme, denoted once again by $\widehat{\mathcal W}_1^{Tot}$, as $N$ varies in $S$, which is of dimension $$\dim(\widehat{\mathcal W}_1^{Tot}) = g + (7g-10-2d) = 8g-10-2d.$$Furthermore $\widehat{\mathcal W}_1^{Tot}$ turns out to be {\em uniruled} indeed, from the proof of Theorem \ref{CF5.8}, any $\widehat{\mathcal W}_1$ is {\em unirational}, being dominated by the incidence variety $\mathcal J$ therein which was a projective bundle, over $\Sigma \cong \mathbb P(H^0(K_C-N))$, whose fibers were projective spaces $|\Ii_{\widehat{\sigma}\vert \Pp^{\vee}} (1)|$ of dimension $m-1 -l >0$.

The stability of the bundle $\Ff_u$, for $N \in {\rm Pic}^{d-2g+2}(C)$ and $[u] \in \widehat{\mathcal W}_1 $ general, can be proved once again via the ``Serre dual version" 
$\mathcal E_u := \omega_C \otimes \Ff_u^{\vee}$ using \cite[proof of Claim\,2]{Teixidor1}. By stability and speciality $2$ of the general bundle in it, $\widehat{\mathcal W}_1^{Tot}$ is once again endowed with a (rational) modular map 
\[
 \begin{array}{ccc}
 \widehat{\mathcal W}_1^{Tot}& \stackrel{\pi}{\dashrightarrow} & U^s_C(d) \\
 (N, [u]) & \longrightarrow &[\mathcal F_u]
 \end{array}
 \] and $\im(\pi) \subseteq B^{k_2}_d \cap U_C^s(d)$. 
 
Similarly as in Claim \ref{cl:20feb816}, one shows also in this case that the general fiber of $\pi$ is $\Pp^1$, being isomorphic to the complete linear system $|\mathcal O_{F_u} (\Gamma)|$  of the canonical section $\Gamma$ on the surface scroll $F_u$ corresponding to $\Ff_u$ general in $\im(\pi)$. Hence, as above, $ \im(\pi)$ has dimension $8g-11-2d = \rho_d^{k_2}$ so it is dense in the unique component $\mathcal B_{\rm reg,2}$, which is therefore {\em uniruled} as $\im(\pi)$ is dominated by $\widehat{\mathcal W}_1^{Tot}$ (which is unirational from above); moreover $[\Ff_u] \in \mathcal B_{\rm reg,2}$ general is  not rigidly specially presented (not rsp), as in Definition \ref{def:ass1}, via 
 \begin{equation}\label{eq:Petri20feb}
 0 \to N \to \Ff_u \to \omega_C \to 0.
\end{equation} To study the Petri map, consider the cohomology sequences arising from \eqref{eq:Petri20feb}, namely:
$${\small 0 \to H^0(\Ff_u) \to H^0(K_C) \to {\rm Im}(\partial_u) \to 0 \;\; {\rm and} \;\; 
0 \to {\rm Coker}(\partial_u) \simeq \mathbb C \to H^1(\Ff_u) \to H^1(K_C) \cong \mathbb C \to 0.}$$ Set $W \simeq H^0(\Ff_u)$ the image inside $H^0(K_C)$ via the above inclusion 
$H^0(\Ff) \hookrightarrow H^0(K_C)$; thus $W$ is of codimension $r-1 =1$ in $H^0(K_C)$. If we take the dual of the second exact sequence and tensor it by $W$, we get the following diagram:
\begin{equation}\label{eq:Feb201010}
\begin{array}{ccccll}
\\[1ex]
  & 0 & & & & \\
 &\downarrow & & &  & \\[1ex]
0 \to & W \otimes H^0(\mathcal O_C) & \stackrel{\mu_{W, \mathcal O_C}}{\longrightarrow} & W & \subset & H^0(K_C)  \\[1ex]
 &  \downarrow &  & & &  \\[1ex]
& W \otimes H^0(\omega_C \otimes \Ff_u^{\vee}) &  \stackrel{\alpha}{\longrightarrow} & H^0(\Ff_u^{\vee} (2K_C)) & \simeq & H^0(\Ff_u(K_C-N)) \\[1ex]
&\downarrow & & & & \\[1ex]
0 \to & W \otimes {\rm Coker}(\partial_u)^{\vee}  & \stackrel{{\mu_{K_C, K_C-N}}_|}{\longrightarrow} & H^0(2K_C -N) & = & H^0(2K_C-N)    \\[1ex]
& \downarrow & & & & \\
& 0 & & & &  
\end{array}
\end{equation}where the first map $\mu_{W, \mathcal O_C}$ is the natural multiplication map, which is injective, the isomorphism in the second row follows from $\Ff_u^{\vee} \simeq \Ff_u (- \det(\Ff_u))$ and $\det(\Ff_u) \simeq K_C+N$, whereas the latter horizontal map is given by the restriction $ {\mu_{K_C, K_C-N}}_|$ to the vector space $W \otimes {\rm Coker}(\partial_u)^{\vee}$ of the natural multiplication map \linebreak $H^0(K_C) \otimes H^0(K_C -N) \stackrel{{\mu_{K_C, K_C-N}}}{\longrightarrow} H^0(2K_C -N)$, since one clearly has inclusions $W \subset H^0(K_C)$ and \linebreak ${\rm Coker}(\partial_u)^{\vee} \subset H^0(K_C-N)$ as proper sub-vector spaces. 

This restricted map is also injective as
$  {\rm Coker}(\partial_u)^{\vee} \simeq \mathbb C$ and because of the composition of 
injections $W \hookrightarrow H^0(K_C) \hookrightarrow H^0(2K_C-N)$, the latter inclusion following from the effectiveness of $K_C-N$. Observe that one can complete the right-most column of the previous diagram; indeed, tensoring \eqref{eq:Petri20feb} by its dual sequence and then tensoring by $K_C$ gives the following exact diagram:
\begin{equation}\label{eq:20Feb1055}
\begin{array}{ccccccccccccccccccccccc}
&&0&&0&&0&&\\[1ex]
&&\downarrow &&\downarrow&&\downarrow&&\\[1ex]
0&\lra& N  & \rightarrow & \mathcal F_u& \rightarrow & K_C &\rightarrow & 0 \\[1ex]
&&\downarrow && \downarrow && \downarrow& \\[1ex]
0&\lra & \mathcal F_u^{\vee} (N+K_C) \simeq \mathcal F_u& \rightarrow & \omega_C \otimes \mathcal F_u\otimes \mathcal F_u^{\vee} & \rightarrow &
\mathcal F_u^{\vee} (2K_C) \cong \Ff_u(K_C-N) &\rightarrow & 0 \\[1ex]
&&\downarrow &&\downarrow &&\downarrow  &&\\[1ex]
0&\lra & K_{C}& \lra & \mathcal F_u(K_C -N)&\lra& 2K_C-N &\lra& 0 \\[1ex]
&&\downarrow &&\downarrow&&\downarrow&&\\[1ex]
&&0&&0&&0&&
\end{array}
\end{equation} From the cohomology sequence of the right-most column of the previous diagram, we get that \eqref{eq:Feb201010} completes to the following commutative diagram:
\[
\begin{array}{ccccll}
\\[1ex]
  & 0 & & & & 0 \\
 &\downarrow & & &  & \downarrow \\[1ex]
0 \to & W \otimes H^0(\mathcal O_C) & \stackrel{\mu_{W, \mathcal O_C}}{\longrightarrow} & W & \subset & H^0(K_C)  \\[1ex]
 &  \downarrow &  & & & \downarrow \\[1ex]
 & W \otimes H^0(\omega_C \otimes \Ff_u^{\vee}) &  \stackrel{\alpha}{\longrightarrow} & H^0(\Ff_u^{\vee} (2K_C)) & \simeq & H^0(\Ff_u(K_C-N)) \\[1ex]
&\downarrow & & & & \downarrow\\[1ex]
0 \to & W \otimes {\rm Coker}(\partial_u)^{\vee}  & \stackrel{{\mu_{K_C, K_C-N}}_|}{\lra} & H^0(2K_C -N) & = & H^0(2K_C-N)    \\[1ex]
& \downarrow & & & & \downarrow \\
& 0 & & & &  0
\end{array}
\]which proves the injectivity of the map $\alpha$. 
Taking into account the cohomology sequence obtained from the middle column in \eqref{eq:20Feb1055}, the fact that $W \simeq H^0(\Ff_u)$, the Petri map $\mu_{\Ff_u}$ and the map $\alpha$ previously defined, one gets the following diagram:
\[
\begin{array}{cccc}
\\[1ex]
  &  &  & 0 \\
 & &   & \downarrow \\[1ex]
 &  &    & H^0(\Ff_u)  \\[1ex]
 &   &  & \downarrow \\[1ex]
 & W \otimes H^0(\omega_C \otimes \Ff_u^{\vee}) &  \stackrel{\mu_{\Ff_u}}{\longrightarrow} & H^0(\omega_S \otimes \Ff_u \otimes \Ff_u^{\vee})) \\[1ex]
&|| & &  \downarrow^{\pi}\\[1ex]
0 \to & W \otimes H^0(\omega_C \otimes \Ff_u^{\vee})  & \stackrel{\alpha}{\lra} &  H^0(\Ff_u(K_C-N))    
\end{array}
\]Take $x \in {\rm Ker}(\mu_{\Ff_u})$ so that $\mu_{\Ff_u}(x) = 0$; since 
$\pi \circ \mu_{\Ff_u} = \alpha$, one must have  
$0 = \pi (0) = \pi(\mu_{\Ff_u}(x)) = \alpha(x)$, namely $x \in \ker(\alpha) = \{0\}$ as $\alpha$ was injective. Therefore $\mathcal B_{\rm reg,2}$ is {\em regular} and {\em uniruled} also in these cases.

To conclude with the description of a {\em minimal presentation} of $[\Ff] \in \mathcal B_{{\rm reg},2}$ general as stated, we observed that Teixidor's presentation of $\Ff_u$ general in such a component was not rigidly special (not rp), since the map $\pi$ has a $1$-dimensional general fiber. Nevertheless, since this general fiber is isomorphic to the linear pencil of canonical sections $\Gamma$ on the corresponding surface scroll $F_u$ (recall the proof of Claim \ref{cl:20feb816}) and since $\Gamma^2 = 4g-4-d >0$, this linear pencil certainly has some base points. If $p$ is any of such base points, as in Proposition \ref{prop:CFliasi}-(2-i), one can split off the fiber $f_p$ and gets a section $\Gamma'$, of degree $2g-3$ and of speciality $1$, since it corresponds to the quotient line bundle $\Ff_u \to \omega_C(-p) \to 0$. In such a case, $\Ff_u$ turns out to be rigidly specially presented (rsp), as in Definition \ref{def:ass1}, via the exact sequence $0 \to N' \to \Ff_u \to \omega_C(-p) \to 0$ where by determinantal reasons in cases $(iii-1)$, where $N = K_C - D_{4g-4-d}$, one has $N' = K_C - D_{4g-4-d} + p$, 
whereas in cases $(iii-2)$ one has  $N' = N+p$. 
\end{proof}

\bigskip

\noindent
$\boxed{3g-5 \leq d \leq 4g-5-2\nu}$. These cases occur when $3 \leq \nu \leq \frac{g}{2}$ (otherwise, for $\frac{g}{2} < \nu < \lfloor \frac{g+3}{2}\rfloor$, we are in the previous case where $B_{d}^{k_2}\cap U_C^s(d) = \mathcal B_{\rm reg,2}$) and have been studied in \cite[Theorem\;1.5-(i)]{CFK2} where we precisely proved:

\begin{proposition}\label{prop:caso3spec2} (cf. \cite[Theorem\;1.5-(i)]{CFK2})  Let $g \geq 6$, $3 \leq \nu \leq \frac{g}{2}$ and $3g-5 \leq d \leq 4g-5-2\nu$ be integers and let $C$ be a general $\nu$-gonal curve of genus $g$. Then $B_{d}^{k_2}\cap U_C^s(d)$ contains a unique component which is {\em regular}, denoted by $\mathcal B_{\rm reg,2}$. Such a component is moreover {\em uniruled} and a general element $[\mathcal F] \in \mathcal B_{\rm reg,2}$ corresponds to a bundle fitting   in an exact sequence
\begin{equation}\label{eq:aiuto13feb1}
 0\to \omega_C(-D) \to \mathcal F\to   \omega_C(-p) \to 0,  
\end{equation} where $p\in C$ and $D\in C^{(4g-5-d)}$ are general. Moreover, $\omega_C(-p)$  is of minimal degree among special quotient line bundles of $\mathcal F$.
\end{proposition}

\begin{proof}[Sketch of the proof of  Proposition \ref{prop:caso3spec2}] The strategy used in \cite[Theorem\;1.5-(i)\,and\,\S\;3.2]{CFK2}  is the following: take a general effective divisor $D \in C^{4g-5-d}$, where $\deg (D) = 4g-5-d \leq g$, and  $p\in C$ general and consider $\mathcal W_1 \subseteq \ext^1(K_C-p,K_C-D)$ as in \eqref{W1}. Since we proved $\mathcal W_1$ to be a sub-vector space of $ \ext^1(K_C-p,K_C-D)$ and of dimension $\dim(\mathcal W_1) = 4g-6-d$, we deduced that, for $u \in \mathcal W_1$ general, the corresponding rank-$2$ vector bundle $\Ff_u$ as in \eqref{eq:aiuto13feb1} was stable (more precisely $s(\Ff_u) \geq 1$-- resp. $2$-- if $d$ is odd--resp. even), of speciality $i=h^1(\Ff)=2$; this has been done with the use of Proposition \ref{LN}. 

Precisely we first took, as in Proposition \ref{LN}, the morphism $\varphi$ associated to the complete linear system $|K_C+L-N| = |K_C+D-p|$ as in \eqref{eq:aiuto13feb1} and,  correspondingly, the image curve $X := \varphi(C) \subset \mathbb P = \mathbb P({\rm Ext}^1(K_C-p, K_C-D))$. Considering $\sigma = 1$ or $2$, according to the parity of $d$, we took the corresponding secant variety $${\rm Sec}_{\frac{1}{2} (2\delta-d+\sigma-2)}(X) = {\rm Sec}_{\frac{1}{2}(4g-8-d+\sigma)}(X),$$whose dimension is $\dim\left({\rm Sec}_{\frac{1}{2}(4g-8-d+\sigma)}(X) \right) = 4g-9-d+\sigma$, where $\sigma = 1,2$.  
Since we proved $\mathcal W_1 \subseteq \ext^1(K_C-p,K_C-D)$ to be a sub-vector space, then $\widehat{\mathcal W}_1:= \mathbb P(\mathcal W_1)$ was therefore a linear sub-space of $\mathbb P$ of dimension $\dim( \widehat{\mathcal W}_1) = \dim(\mathcal W_1) -1 = 4g-7-d$. Thus, for $\sigma =1,2$, we got that 
$$\dim\left({\rm Sec}_{\frac{1}{2}(4g-8-d+\sigma)}(X) \right) = \dim(\widehat{\mathcal W}_1) + 1 -\sigma.$$Using the fact that $X \subset \mathbb P$ is non-degenerate and in accordance with some results by \cite{Voisin} related to maximal linear sub-spaces contained in a secant variety of an embedded curve, we deduced that general $[u] \in \widehat{\mathcal W}_1$ did not belong to ${\rm Sec}_{\frac{1}{2}(4g-8-d+\sigma)}(X)$ so, from Proposition \ref{LN}, we got 
$s(\Ff_u) \geq \sigma \in \{1,2\}$ namely the stability of such a $\Ff_u$. Moreover, being $u \in \widehat{\mathcal W}_1$ general and $h^1(\omega_C(-p)) =1$, we also deduced $h^1(\Ff_u)=2$. 

We then proved the existence of a natural projective bundle $\mathbb P(\mathcal E_{d})\to S$, where $S \subseteq W^0_{4g-5-d}(C)\times C$ is an open dense subset: namely, such a projective bundle $\mathbb P(\mathcal E_{d})$ is determined by the family of 
projective spaces $\mathbb P(\ext^1(K_C-p,K_C-D))$'s  as  $(D, p) \in S$ varies. Since, for any such $(D, p) \in S$, any $\widehat{\mathcal W}_1$ as above is a linear sub-space of (constant) dimension $4g-7-d$, one deduced the existence of an irreducible projective variety $$\widehat{\mathcal W}_1^{Tot}:= \left\{ (D,p,[u]) \in \mathbb P(\mathcal E_{d}) \; | \; H^0(K_C-p) \stackrel{\partial_u = 0}{\longrightarrow} H^1(K_C-D)\right\} \subsetneq \mathbb P(\mathcal E_{d})$$which is therefore {\em ruled} over $S$, of 
dimension 
$$\dim \left(\widehat{\mathcal W}_1^{Tot} \right) = \dim (S) + 4g - 7 - d = 4g - d - 4 + 4g - 7 - d = 8 g - 2d - 11 = \rho_d^{k_2}.$$
From stability and speciality $2$ of the bundle $\Ff_u$ arising from the general triple $(D,p,u) \in \widehat{\mathcal W}_1^{Tot}$, we deduced a natural (rational) modular map
 \[
 \begin{array}{ccc}
 \widehat{\mathcal W}_1^{Tot}& \stackrel{\pi}{\dashrightarrow} & U^s_C(d) \\
 (D,p, [u]) & \longrightarrow & [\mathcal F_u]
 \end{array}
 \] s.t. $\im(\pi) \subseteq B^{k_2}_d \cap U_C^s(d)$.

We set $\mathcal B_{\rm reg,2}$ to be the closure of $\im(\pi)$ in $B^{k_2}_d \cap U_C^s(d)$; to show that $ \mathcal B_{\rm reg,2}$ is actually a generically smooth component of $B_d^{k_2} \cap U^s_C(d)$, with (expected) dimension $\rho_d^{k_2} = 8g-11-2d$, i.e. $\mathcal B_{\rm reg,2}$ is {\em regular}, we can apply arguments as in \cite[Prop.\;3.9]{CFK}.  From  the fact that $\im(\pi)$ contains stable bundles, any component of $B^{k_2}_d \cap U_C^s(d)$  containing it has dimension at least $\rho_d^{k_2}$. One can compute $\dim (T_{\mathcal F}(B^{k_2}_d))$, for  general $\mathcal F \in \im(\pi)$ by considering the Petri map 
$$\mu_\mathcal F: H^0(\mathcal F)\otimes H^0(\omega_C\otimes \mathcal F^{\vee})\to H^0(\omega_C\otimes \mathcal F\otimes \mathcal F^{\vee})$$for a general $\mathcal F\in \im(\pi)$. From the fact that $\mathcal F = 
\mathcal F_u$, for some $u$ in some fiber $\widehat{\mathcal W}_1$ of $\widehat{\mathcal W}_1^{Tot}$, the corresponding coboundary map $\partial_u$ is the zero-map; in other words 
$$H^0(\mathcal F) \cong H^0(K_C-D) \oplus H^0(K_C-p) \;\;\;{\rm and} \;\;\; H^1(\mathcal F) \cong  H^1(K_C-D) \oplus H^1(K_C-p).$$This means that, for any such bundle, the domain of the Petri map $\mu_{\mathcal F}$ coincides with that 
of $\mu_{{\mathcal F}_0}$, where  
$\mathcal F_0 := (K_C-D) \oplus (K_C-p)$ corresponds to the zero vector in $\mathcal W_1 \subset \ext^1(K_C-p, K_C-D)$. So one can focus on computing the kernel of $\mu_{{\mathcal F}_0}$ and then reason via semi-continuity. Observe that 
\[\begin{array}{ccl}
H^0(\mathcal F_0) \otimes H^0(\omega_C \otimes \mathcal F_0^{\vee}) & \cong &  \left(H^0(K_C-D) \otimes H^0(D)\right)  \oplus \left(H^0(K_C-D) \otimes H^0(p)\right) \oplus\\
& & \left(H^0(K_C-p) \otimes H^0(D)\right) \oplus \left(H^0(K_C-p) \otimes H^0(p)\right). 
\end{array}
\]Moreover 
$$\omega_C \otimes \mathcal F_0 \otimes \mathcal F_0^{\vee} \cong K_C \oplus (K_C+p-D) \oplus (K_C+D-p) \oplus K_C.$$Therefore, for Chern classes reason,  
$$\mu_{\mathcal F_0} = \mu_{0, D} \oplus \mu_{K_C-D, p} \oplus \mu_{K_C-p, D} \oplus \mu_{0, p}$$where the maps 
\begin{eqnarray*}
\mu_{0,D}: & H^0(D)\otimes H^0(K_C-D)\to H^0(K_C),\\
 \mu_{K_C-D,p}:& H^0(K_C-D)\otimes H^0(p)\to H^0(K_C-D+p)\\
 \mu_{K_C-p,D}: & H^0(K_C-p)\otimes H^0(D)\to H^0(K_C+D-p)\\
  \mu_{0,p}: & H^0(p)\otimes H^0(K_C-p)\to H^0(K_C)
 \end{eqnarray*}are natural multiplication maps. Since $h^0(D) = h^0(p) =1$, the maps $\mu_{0,D}, \; \mu_{K_C-D,p},\; \mu_{K_C-p,D}\;\mu_{0,p}$ are all injective and so is $\mu_{\mathcal F_0}$. By semicontinuity on $\mathcal W_1$, one therefore has that $\mu_{\mathcal F}$ is injective, for $\mathcal F$ general in $\widehat{\mathcal W}_1$. 
 
The previous argument shows that a general $\mathcal F \in \im(\pi) \subset \mathcal B_{\rm reg,2}$  is contained in only one irreducible component $\mathcal B$  of $B_d^{k_2} \cap U^s_C(d)$ for which 
\begin{eqnarray*}
\dim (\mathcal B)= \dim (T_{\mathcal F}(\mathcal B))&=&4g-3-h^0(\mathcal F)h^1(\mathcal F) \\
&=&4g-3-2(d-2g+4) = 8g-11-2d, 
\end{eqnarray*} i.e.  $\mathcal B$ is generically smooth and of dimension $\rho_d^{k_2}$. 

To conclude that $B_{\rm reg,2} = \mathcal B$ it suffices to observe that the rational map $\pi$ is generically finite onto its image. To do this, let $F_u = \mathbb{P} (\mathcal F_u)$ be the surface scroll, for general $u \in \widehat{\mathcal W}_1^{Tot} $, and let $\Gamma$ be the section corresponding to the 
quotient $\mathcal F_u \to \!\!\!\! \to K_C-p$; then its normal bundle is $N_{\Gamma/F} \simeq D-p$ which has no sections. Thus, one deduces the generically finiteness of $\pi$.

This implies that $\mathcal B_{\rm reg,2}$ is a generically  smooth component of  $B^{k_2}_d\cap U^s_C(d)$ with (expected) dimension $\rho_d^{k_2} = 8g-11-2d$, i.e. $\mathcal B_{\rm reg,2}$ is {\em regular}. Moreover, by the parametric description of $\mathcal B_{\rm reg,2}$, one deduces that it is {\em uniruled}, being finitely dominated by $\widehat{\mathcal W}_1^{Tot}$, whose general fiber over $S$ is a projective space, and the general point of $ \mathcal B_{\rm reg,2}$ arises as an extension in \eqref{eq:aiuto13feb1} with corank-one coboundary map.

Finally, in \cite[Prop.\;3.8]{CFK2}, we showed the uniqueness of $ \mathcal B_{\rm reg,2} $ among possible components of \linebreak $B_d^{k_2} \cap U_C^s(d)$, whose general point represents a bundle $\mathcal F$ of speciality $2$ and presented via an extension 
\eqref{degree} with $h^1(L) =1$.  As observed in \cite[Rem.\;3.9]{CFK2}, the proof of the uniqueness of $\mathcal B_{\rm reg,2} $ among such type of components also shows that, for $[\Ff] \in \mathcal B_{\rm reg,2}$ general, the line bundle $\omega_C(-p)$ is minimal among special quotient line bundles for $\mathcal F$ general in $\mathcal B_{\rm reg,2}$ and that it admits also a {\em presentation} via a canonical quotient, i.e. it fits in  $$0 \to \omega_C(-D-p) \to \mathcal F \to \omega_C\to 0,$$whose residual presentation coincides with that in the proof of \cite[Theorem]{Teixidor1}.

In other words, the component  $\mathcal B_{\rm reg,2}$ coincides with the component $\mathcal B$ stated in  Theorem \ref{TeixidorRes} (cf. \cite[Theorem]{Teixidor1}, which therein turns out to be also the only component when $C$ is with general moduli). The proof of generic smoothness of $\mathcal B = \mathcal B_{\rm reg,2}$ in \cite[Thm.\,1]{Teixidor2} (in her ``Serre dual notation" $W^1_{2,d}$) was done only under the assumption on $C$ with {\em general moduli}.

Above (and also in  \cite[Prop.\,3.9]{CFK}) we instead proved generic smoothness of $\mathcal B = \mathcal B_{\rm reg,2}$ for $C$ a general $\nu$-gonal curve. 
\end{proof}

\bigskip

\noindent
$\boxed{2g-2 \leq d \leq 3g-6}$ These cases are brand new and they appear in the statement of Theorem \ref{thm:i=2}. To approach these new cases, we start with the following preliminary results.

\begin{lemma}\label{lem:i=1.2} Let $d=3g-5-j$ be an integer, with $1 \leq  j\leq g-3$. Let $N\in {\rm Pic}^{d-2g+3}(C)$ be general and let $\mathcal W_1 \subseteq \ext^1(K_C-p, N)$  be the degeneracy locus as in \eqref{W1}. 

Then $\mathcal W_1$ is irreducible, of dimension $g+2j-1$ and {\em unirational}. Moreover, for $u\in \mathcal W_1$ general, the corresponding rank-two, degree-$d$ vector bundle $\mathcal F_u$ is of speciality $h^1(\mathcal F_u)=2$, with Segre invariant $s(\mathcal F_u) \ge 1$ (resp., $2$)  if $d$ is odd (resp., even).  In particular $\Ff_u$ is stable. 
\end{lemma}

\begin{proof} We first prove the assertions on $\mathcal W_1$ for any $d\leq 3g-6$. Notice that $\deg(N) = d-2g+3 = g-2-j \leq g-3$ so, by generality, $N$ is special and non-effective. Thus we have the following set-up: 
\begin{equation}\label{degree0}
\begin{CD}
&(u)& : \;\;\; 0\to &N&\to &\ \ \mathcal F_u\ \ & \to  \ &\ \ K_C-p\ \ &\to 0\\
&\deg&       &g-2-j&& d&& 2g-3&\\
&h^0&       &0&&  && g-1&\\
&h^1&       &j+1&&  && 1&.
\end{CD}
\end{equation}  Since $\deg (K_C-p) = 2g-3  > g-2-j =\deg (N)$, then 
$K_C-p \ncong N$ so, using  \eqref{eq:m} and notation as in Theorem \ref{CF5.8}, 
one has $l:=h^0(K_C-p) = g-1$, $r:= h^1(N) = j+1$ and $m:=\dim(\ext^1(K_C-p,N))=2g-2+j$;   
therefore $l = g-1> r = j+1$ and $m=2g-2+j \geq l+1=g$. 

We can therefore apply Theorem \ref{CF5.8} to get that $\mathcal W_1=\{u\in\ext^1(K_C-p, N)\ |\ \corank(\partial_u)\ge 1\}$ is irreducible, of  (expected) dimension  
\begin{equation}\label{dim:w1}
 \dim (\mathcal W_1) = m- 1(l-r+1)=g+2j -1.
\end{equation} Set $\mathbb{P}:= \mathbb{P}\left(\ext^1(K_C-p, N)\right)$ and 
$\widehat{\mathcal W}_1 := \mathbb{P}(\mathcal W_1)$; from  the proof of Theorem \ref{CF5.8}, $\widehat{\mathcal W}_1$ is {\em unirational}, being dominated by the incidence variety $\mathcal J$, which turned-out to be a projective bundle over $\Sigma \cong \mathbb P(H^0(K_C-N))$, whose fibers were projective spaces of dimension $m-1 -l >0$. This proves the assertions on $\mathcal W_1$. 

Concerning the speciality of the bundle $\mathcal F_u$ corresponding to $u \in \mathcal W_1$ general, by Theorem \ref{CF5.8}-(ii) and by \eqref{degree0}, one has $$h^1(\mathcal F_u) =h^1(K_C-p) + {\rm corank} (\partial_u) = 2,$$and we are done. Concerning the bounds on the Segre invariant (and so the stability, by \eqref{eq:neccond}) of the bundle $\Ff_u$ for $u \in \mathcal W_1$ general,  recall that $\widehat{\mathcal W}_1$ has dimension $g+2j-2$; posing $\delta := 2g-3$ and considering that $d= 3g-5-j$, for $1 \leq j \leq 2g-3$, one has $2\delta - d =g-1+j\geq g \geq 4$. We are therefore in position to apply  Theorem \ref{LN} and consider the natural morphism $ C\stackrel{\varphi}{\longrightarrow}\mathbb P$, given by the complete linear system $|2K_C-p-N|$. 
Set $X := \varphi\left(C\right)$ and let $\sigma$  be an integer such that  
$ \sigma \equiv 2\delta-d\  {\rm{ (mod \;2) \ \ and  }} $ $ 0 \le \sigma \le 2\delta - d$; then we have 
$$\dim \left(\Sec_{\frac{1}{2}(2\delta-d+\sigma-2)}(X) \right)=2\delta - d + \sigma - 3 = g - 1 + j + \sigma -3 \le g-2+2j = \dim (\widehat{\mathcal W}_1)$$
if and only if 
$\sigma \le 2+j$, where the equality holds if and only if $\sigma=2+j$. 
Thus, since $j\geq 1$,  we may take $\sigma =1$ for $d$ odd (resp., $\sigma=2$ for $d$ even) which shows in particular the stability of $\Ff_u$. 
\end{proof}

To construct the locus $\mathcal B_{\rm reg,2}$ and to show that it actually is an irreducible component of $B_d^{k_2}\cap U^s_C(d)$ as in Theorem \ref{thm:i=2} notice that, as in \cite[\S\;3.2]{CFK}, one has a projective bundle $\mathbb P(\mathcal E_{d})\to S$ over an open dense subset $S \subseteq {\rm Pic}^{d-2g+3}(C)\times C$ of dimension $g+1=4g-4-j-d$; namely, 
$\mathbb P(\mathcal E_{d})$ is the family of $\mathbb P(\ext^1(K_C-p, N))$'s  as  $(N, p) \in S$ varies. Since, 
for any such $(N, p) \in S$, $\widehat{\mathcal W}_1$ is irreducible, {\em unirational} and of (constant) dimension $g-2+2j=4g-7+j-d$, 
one has an irreducible projective variety 
$$\widehat{\mathcal W}_1^{Tot}:= \left\{ (N,p,u) \in \mathbb P(\mathcal E_{d}) \; | \; H^0(K_C-p) \stackrel{\partial_u}{\longrightarrow} H^1(N),\ \ \codim(\partial_u)=1\right\},$$
which is  therefore {\em uniruled} over $S$ and of dimension $$\dim(\widehat{\mathcal W}_1^{Tot}) = (4g -4 -j -d )+ (4g - 7 + j - d) = 8 g - 2d - 11 = 
\rho_d^{k_2}.$$From stability, proved in Lemma \ref{lem:i=1.2} for $(N,p,u) \in \widehat{\mathcal W}_1^{Tot}$ general, one has a natural (rational) modular map
\[
 \begin{array}{ccc}
 \widehat{\mathcal W}_1^{Tot}& \stackrel{\pi}{\dashrightarrow} & U^s_C(d) \\
 (N, p, [u]) & \longrightarrow & [\mathcal F_u]
 \end{array}
 \] and  $\im(\pi) \subseteq B^{k_2}_d \cap U_C^s(d)$. 

\begin{lemma}\label{lem:pigenfin} The map $\pi: \widehat{\mathcal W}_1^{Tot}{\dashrightarrow} U^s_C(d)$ is {\em birational}. In particular, $\im(\pi)$ is {\em uniruled} and 
$\dim(\im(\pi)) = \dim(\widehat{\mathcal W}_1^{Tot}) = 8 g - 2d - 11 = 
\rho_d^{k_2}$. 
\end{lemma}
\begin{proof} Once we show that $\pi$ is birational onto its image, the uniruledness and the stated dimension of $\im(\pi)$ will follow from the fact that it is birational 
to $\widehat{\mathcal W}_1^{Tot}$, which is uniruled and of computed dimension.  

To prove that $\pi$ is birational onto its image, let $\Ff_u$ be the bundle corresponding to the general point of $\im(\pi)$. If $F_u = \Pp(\Ff_u)$ denotes the associated surface scroll and if $\Gamma$ denotes the section of $F_u$ corresponding to the quotient line bundle $K_C-p$ of $\Ff_u$, from \eqref{eq:Ciro410} one has that $\mathcal N_{\Gamma/F_u} \simeq K_C-p-N$ is general, since $N$ is, and of degree $g-1 + j$. Therefore $h^0(\mathcal N_{\Gamma/F_u})= h^0(K_C-p-N)=j>0$ and $h^1(\mathcal N_{\Gamma/F_u})=0$, i.e. the Hilbert scheme ${\rm Div}^{1,2g-3}_{F_u}$ is unobstructed at $[\Gamma]$ and of dimension $j$. 

We will now prove that $[\mathcal F_u]\in \im (\pi)$ general is rigidly specially presented (rsp) via \eqref{degree0} (recall Definition 
\ref{def:ass1}) more precisely that the section $\Gamma \subset F_u$ is algebraically specially unique (asu).

\begin{claim}\label{lem:linearly_isolated} $\Gamma \subset F_u$ is linearly isolated (li) and  algebraically specially unique (asu), but not of minimal degree on $F_u$. 
 \end{claim} 

\begin{proof}[Proof of Claim \ref{lem:linearly_isolated}] First of all we show that $\Gamma$ is linearly isolated (li) in $F_u$ which, by \eqref{eq:isom2}, is equivalent to proving that $h^0(\mathcal F_u \otimes N^{\vee}) =1$. From Serre duality, the facts $\Ff_u^{\vee} \simeq \Ff_u (-{\rm det}(\Ff_u))$ 
and ${\rm det} (\Ff_u) = K_C - p + N$ and \eqref{degree0}, this is equivalent to proving that $h^1(\mathcal F_u(p))=1$. To do this, we tensor by $\mathcal O_C(p)$ the exact sequence \eqref{degree0} to get
\begin{equation}\label{w1_p}
0\  \to \ N+p \  \to \  \mathcal F_v:=\mathcal F_u(p)\   \to  \  K_C\  \to \ 0,
\end{equation} where $v\in \ext^1(K_C, N+p)$. Since $\Ff_u$ in \eqref{degree0} arises from a general pair $(N,p) \in {\rm Pic}^{d-2g+3}(C) \times C$ and for \linebreak $u\in \mathcal W_1 \subseteq \ext^1(K_C-p, N)$ general, where $\dim(\mathcal W_1)=g-1+2j$ by \eqref{dim:w1}, and since $$\ext^1(K_C, N+p) \simeq \ext^1(K_C-p, N)$$(being both isomorphic to $H^1(N+p-K_C)$), it follows that $v$ is a general point of an irreducible, unirational and $(g-1+2j)$-dimensional locus in $\ext^1(K_C, N+p)$ which, by small abuse of notation, we will always denote by $\mathcal W_1 \subseteq  \ext^1(K_C, N+p)$. 

Now we consider the relevant degeneracy sub-locus in $\ext^1(K_C, N+p)$ as in \eqref{W1}, namely: 
\begin{equation}\label{W1tilde}
 \widetilde{\mathcal W_1}:=\{u\in\ext^1(K_C,N+p)\ |\ {\rm corank} (\partial_u)\ge 1\}\subseteq \ext^1(K_C,N+p).
\end{equation} We check the validity of conditions in Theorem \ref{CF5.8}; note that from \eqref{w1_p} we have $l:=h^0(K_C)=g$ and $r:= h^1(N+p) = j$, as we had $h^1(N) = j+1$ and $N$ general of $\deg(N) \leq g-3$. 
Since $1\leq j\leq g-3$, we get therefore $l> j$; moreover, from above $m= \dim(\ext^1(K_C, N+p))= 2g-2+j \geq l+1=g+1$. Hence, by Theorem \ref{CF5.8}, a general element $w\in\widetilde{\mathcal W_1}$ gives corank($\partial_w)=1$ and 
$$\dim(\widetilde{\mathcal W_1})=m-(g-j+1)=g-3+2j.$$Since $v$ in \eqref{w1_p} was general in $\mathcal W_1 \subseteq \ext^1(K_C, N+p)$ such that $\dim(\mathcal W_1) = g-1+2j > g-3+2j = \dim(\widetilde{\mathcal W_1})$, then a general $v\in\mathcal W_1 \subseteq  \ext^1(K_C, N+p) $ gives rise to a surjective coboundary map $\partial_v$ and we get 
\begin{equation}\label{eq:linearly_isolated}
h^1(\mathcal F(p))=h^1(K_C)=1,
\end{equation} so $\Gamma$ is linearly isolated (li) in $F_u$ as desired.

Now we will show that, for $[\Ff_u] \in \im(\pi)$ general, the section  $\Gamma$ is even algebraically specially unique (asu) on the scroll $F_u$. Assume by contradiction that $\Gamma$ is (not asu) on $F_u$ then, by Proposition \ref{prop:CFliasi}-$(1)$,  we have 
the following diagram for some point  $p\neq q \in C$:
\begin{equation*}
\begin{array}{ccccccccccccccccccccccc}
&&&&0&&&&\\[1ex]
&& &&\downarrow&&&&\\[1ex]
&&  &  & N' &  &  & &  \\[1ex]
&& && \downarrow && & \\[1ex]
0&\lra & N & \rightarrow & \mathcal F_u& \rightarrow &
K_C-p &\rightarrow & 0 \\[1ex]
&& &&\downarrow && &&\\[1ex]
& & &  & K_C-q&& && \\[1ex]
&& &&\downarrow&&&&\\[1ex]
&&&&0&&&&
\end{array}
\end{equation*} where $N' := N + q - p$ since $\det(\Ff_u) = K_C + N -p \simeq K_C-q + N'$. Tensoring with the line bundle $\mathcal O_C(p)$, we get
\begin{equation*}
\begin{array}{ccccccccccccccccccccccc}
&&&&0&&&&\\[1ex]
&& &&\downarrow&&&&\\[1ex]
&&  &  & N + q &  &  & &  \\[1ex]
&& && \downarrow && & \\[1ex]
0&\lra & N+p & \rightarrow & \mathcal F_v:=\mathcal F(p) & \rightarrow &
K_C &\rightarrow & 0 \\[1ex]
&& &&\downarrow &&  &&\\[1ex]
& & &  & K_C+p-q&& && \\[1ex]
&& &&\downarrow&&&&\\[1ex]
&&&&0&&&&
\end{array}
\end{equation*} As above, the extension $v \in \ext^1(K_C+p-q,\; N+q)$ is a general point in the irreducible sub-locus $\mathcal W_1$ of dimension $g-1+2j$, with $N$ and $p$ general. On the other hand, consider the relevant degeneracy 
sub-locus in $ \ext^1(K_C+p-q,N+q)$: 
\[
\widetilde{\widetilde{\mathcal W_1}}:=\{u\in\ext^1(K_C+p-q,N+q)\ |\ {\rm corank} (\partial_u)\ge 1\}\subseteq \ext^1(K_C+p-q,N+q).
\]Once again we check the validity of conditions as in Theorem \ref{CF5.8}; 
in this case we have $l:=h^0(K_C +p-q)=g-1$, since $p \neq q$ because $\Gamma$ is linearly isolated (cf. Prop.\,\ref{prop:CFliasi}-(1)), and $r:= h^1(N+q) = j$, as 
$N$ was general of its degree $\deg(N) \leq g-3$ with speciality $j+1$. Thus, 
from $1\leq j\leq g-3$, we get $l=g-1> r=j$. 

Now set $m:=\dim ({\rm Ext}^1(K_C+q-p, N + q)) = h^0(2K_C+p-N-2q)$; since 
$\deg(2K_C+p-N-2q) = 3g-3 + j > 2g-2$, one has $m = 2g-2+j$, certainly greater than $l+1 = g$.  Hence, from Theorem \ref{CF5.8}, it follows that a general element $w\in\widetilde{\widetilde{\mathcal W_1}}$ gives 
corank($\partial_w)=1$. Notice that $$\dim(\widetilde{\widetilde{\mathcal W_1}})=m-(g-i)=g-2+2j.$$Since $v $ was general in $\mathcal W_1 \subset \ext^1(K_C+p-q, N+q)$ of dimension $\dim(\mathcal W_1) = g-1+2j$, which is bigger 
than $\dim(\widetilde{\widetilde{\mathcal W_1}})$, we deduce that a general $v\in\mathcal W_1$ gives rise to a surjective coboundary map $\partial_v$. This in turn gives $h^1(\mathcal F_u(p))=h^1(K_C+p-q)=0$, which contradicts 
\eqref{eq:linearly_isolated}. Therefore, $\Gamma$ is (asu) on $F_u$.

Finally, the fact that $\Gamma$ is not of minimal degree among sections of $F_u$ follows from $h^0(\mathcal N_{\Gamma/F_u}) = \dim_{[\Gamma]}({\rm Div}^{1,2g-3}_{F_u} ) = j > 0$, from $F_u$ being not decomposable and from Proposition \ref{prop:CFliasi}-(2).
\end{proof} 

We now show how Claim \ref{lem:linearly_isolated} implies that the map $\pi$ is birational; we argue as in the proof of Claim \ref{cl:20feb816}. Let $[\Ff] \in \im(\pi)$ be general, so $\Ff= \Ff_u$ for $(N,p,[u]) \in  \widehat{\mathcal W}_1^{Tot}$ general. By definition of the map $\pi$ we have that $\pi^{-1}([\Ff_u]) = \left\{ (N', p', [u']) \in  \widehat{\mathcal W}_1^{Tot} \, | \; \Ff_{u'} \simeq \Ff_u \right\}$. 

Assume by contradiction $(N',p',[u']) \in \pi^{-1}([\Ff_u]) \setminus \{(N,p,[u])\}$; thus $N' \simeq N+ p' - p$. If $p = p'$, then $N \simeq N'$; thus $[u] \neq [u']$ belong to the same $\widehat{\mathcal W}_1 \subset \mathbb{P} (\ext^1(K_C-p, N)) $. Since $\Ff_u$ is stable (so simple), this implies the existence of an isomorphism $\varphi : \Ff_{u'} {\to} \Ff_u$ (which is not an endomorphism) giving rise to the following diagram: 
\[
\begin{array}{ccccccl}
0 \to & N & \stackrel{\iota_1}{\longrightarrow} & \Ff_{u'} & \to & K_C -p & \to 0 \\ 
 & & & \downarrow^{\varphi} & & &  \\
0 \to & N & \stackrel{\iota_2}{\longrightarrow} & \Ff_u & \to & K_C-p & \to 0.
\end{array}
\]The maps $\varphi \circ \iota_1$ and $\iota_2$ determine, respectively, two non-zero global sections $ s_1 \neq s_2 \in H^0(\Ff_u \otimes N^{\vee}) $. 
They must be linearly dependent, otherwise one would have $h^0(\Ff_u \otimes N^{\vee}) > 1$, contradicting Claim \ref{lem:linearly_isolated} as $ h^0(\Ff_u \otimes N^{\vee}) -1= \dim (|\mathcal O_{F_u}(\Gamma)|)$ by \eqref{eq:isom2}. On the other hand, if they were linearly dependent, then Lemma \ref{lem:technical} would give $[u]= [u']$, a contradiction. If on the contrary $p \neq p'$, the sections $\Gamma$ and $\Gamma'$, respectively corresponding to the quotient line bundles $\omega_C(-p)$ and $\omega_C(-p')$ of $\Ff_u \simeq \Ff_{u'}$, would be such that $\Gamma \sim_{\rm alg} \Gamma'$, contradicting Claim \ref{lem:linearly_isolated} 

This completes the proof of Lemma \ref{lem:pigenfin}. 
\end{proof}

Let $\mathcal B$ be any component of $B^{k_2}_d \cap U_C^s(d)$ containing $\im(\pi)$. From \eqref{lem:pigenfin}, one must have therefore that $\dim(\mathcal B) \geq \dim(\im(\pi)) = \rho_d^{k_2}$. The next result will show that there exists only one of such a component, which turns out to be also {\em regular} and {\em uniruled}.

\begin{proposition}\label{thm:i=1.3} The dimension of  the tangent space to $B^{k_2}_d \cap U_C^s(d)$ at a general point 
$[\mathcal F_u] \in \im(\pi)$ equals the expected dimension $\rho_d^{k_2} = 8g-11-2d$. 

Hence there exists a unique irreducible component of $B^{k_2}_d \cap U_C^s(d)$, denoted by $\mathcal  B_{\rm reg,2}$,  which contains $\im(\pi)$ as an open dense subset. Moreover, $\mathcal B_{\rm reg,2}$ is generically smooth, of the expected dimension $\rho_d^{k_2} = 8g-11 -2d$, i.e. it is 
{\em regular}. In particular, the general point of 
$\mathcal B_{\rm reg,2}$ corresponds to a stable bundle, arising as in Lemma \ref{lem:i=1.2}, and $\mathcal B_{\rm reg,2}$ is {\em uniruled}.  
\end{proposition}

\begin{proof} The proof of the first part of the statement is similar to that in Proposition \ref{prop:caso3spec2} above (cf. also \cite[Prop.\;3.9]{CFK}) which deals with computations of the kernel of the Petri map as in \eqref{eq:petrimap} of $\Ff_u$ via suitable degenerations. The second part directly follows from Lemma \ref{lem:pigenfin} and from the aforementioned computation of the Petri map. 
\end{proof}

Proposition \ref{thm:i=1.3} shows that a general $[\Ff] \in \mathcal B_{\rm reg,2}$ certainly has a {\em special presentation} as in \eqref{degree0}, with 
$N$ and $p$ general. The next results will show in particular that $\omega_C(-p)$ is of {\em minimal degree} among special quotient line bundles of $\mathcal F$.

\begin{lemma}\label{reg_seq_1} Let $C$ be a general $\nu$-gonal curve of genus $g$ and let $\mathcal F$ be a rank-$2$ vector bundle of degree $d$, speciality $i = h^1(\Ff) = 2$ and of {\em second type} in the sense of Definition \ref{def:fstype}. Let $D_s \in C^{(s)}$, be an effective divisor of degree $s >0$ 
and $N \in {\rm Pic}^{d-2g+2+s}(C)$, such that 

\smallskip

\noindent
(i) $1 \leq h^1(N) \leq g-s$, 

\smallskip

\noindent
(ii) $h^0(2K_C-(N+D_s)) \geq g-s+2$,  

\smallskip

\noindent
(iii) $\Ff= \Ff_u$ for $u\in \mathcal W_1 \subset Ext^1(\omega_C(-D_s),\; N)$ general. 

\smallskip

Then $\Ff$ fits also in the exact sequence: $0 \to N(-D_s) \to \Ff \to \omega_C \to 0$. 

\end{lemma}
\begin{proof} Notice that from the assumptions on $\Ff$, being of speciality $2$ and of {\em second type}, one must have $h^1(\omega_C(-D_s)) =1$, so we have $\ell:= h^0(\omega_C(-D_s)) = g-s$ and $h^1(N) \geq 1$, as in (i). 

Notice further that assumption (i) implies $\ell = h^0(\omega_C(-D_s)) \geq h^1(N)$ whereas assumption (ii) implies that $m := \dim ({\rm Ext}^1(K_C-D_s, N)) \geq \ell +2$. Therefore, from Theorem \ref{CF5.8}, we have that $\mathcal W_1 \subset {\rm Ext}^1(\omega_C(-D_s),\; N)$ as in \eqref{W1} is non-empty, irreducible, of the expected codimension $c(1) = \ell - r + 1$ and $u \in \mathcal W_1$ general is such that ${\rm corank}(\partial_u)=1$. So the associated bundle $\mathcal F_u$ has degree $d$, speciality $2$ and it is of {\em second type}, which explains assumption (iii).

By assumption (iii), $\Ff= \Ff_u$ fits in a sequence of the form \eqref{eq:barj-1}, namely $$0 \to N \to \Ff_u \to \omega_C(-D_s) \to 0.$$Composing the line-bundle injection $N (-D_s) \hookrightarrow N$ with the above exact sequence shows that $N(-D_s)$ is also a sub-line bundle of $\mathcal F_u$. To prove the statement, we need to show that $\frac{\Ff_u}{N(-D_s)} \simeq \omega_C$.

Since by \eqref{eq:barj-1} one has $\det(\Ff_u) \simeq \omega_C \otimes N \otimes \mathcal O_C(-D_s)$, by determinantal reasons it is enough to show that the quotient sheaf $\frac{\Ff_u}{N(-D_s)}$ is  torsion-free on $C$, namely that the sub-line bundle $N(-D_s) \hookrightarrow \Ff_u$ is {\em saturated} in $\mathcal F_u$. 

Notice that the injection $N(-D_s) \hookrightarrow \Ff_u$ gives rise to a global section of the rank-$2$ vector bundle $\Ff_u \otimes N^{\vee}  \otimes \mathcal O_C(D_s) \simeq \omega_C\otimes\mathcal F_u^\vee$, where the isomorphism follows from the fact that $\Ff_u$ is of rank $2$ with determinant $N \otimes \omega_C(-D_s)$. Thus, to prove that the injection $N(-D_s) \hookrightarrow \Ff_u$ is saturated  it is enough to show that the general global section of $\omega_C\otimes\mathcal F_u^\vee$ is nowhere vanishing, namely that 
the intersections of the zero-loci of global sections of
$\omega_C\otimes\mathcal F_u^\vee$ is empty, i.e. $h^0(\omega_C\otimes\mathcal F_u^\vee(-q))=h^1(\mathcal F_u(q))\leq 1, \;\; \mbox{for every point}\; q\in C$. 

To do this, tensor  the exact sequence \eqref{eq:barj-1} by $N^{\vee} \otimes \mathcal O_C(D_s)$ to get 
\[
0\ \to\ \mathcal O_C(D_s) \ \stackrel{\sigma}{\to} \ \omega_C\otimes\mathcal F_u^\vee 
= N^{\vee}  \otimes \mathcal O_C(D_s) \otimes \Ff_u \ \to \omega_C \otimes N^{\vee} \ \ \to \ \ 0
\]The injection  $\sigma$ gives rise to a global section $\tilde{\sigma}\in H^0(\omega_C\otimes\mathcal F_u^\vee)$ which is not zero at any $q\notin D_s$. 

We claim that, for any $q \in D_s$, there also exists a global section in $H^0(\omega_C\otimes\mathcal F_u^\vee)$ which is not zero at $q$.  Thus, assume $q \in D_s$ and tensor \eqref{eq:barj-1} by $\mathcal O_C(q)$ to get 
\begin{equation}\label{eq:6gencogl}
0 \to N(q) \to \mathcal F_v := \mathcal F_u \otimes \mathcal O_C(q) \to K_C-(D_s-q) \to 0.
\end{equation}Since $q \in D_s$, then $\widetilde{\ell}:= h^0(K_C-(D_s-q)) = \ell + 1 = g-s + 1$, whereas $\widetilde{r}:= h^1(N(q))$ is either $h^1(N)$ or $h^1(N)-1$ and, in any case, assumption (i) implies $\widetilde{\ell} \geq \widetilde{r}$. Notice that $$\ext^1(K_C-(D_s-q)), N(q)) \simeq \ext^1(K_C-D_s, N)$$since they both are isomorphic to $H^1(N+D_s - K_C)$, whose dimension is $m := h^0(2K_C - (N+D_s))$ which, from assumption (ii), is such that 
$m \geq \widetilde{\ell} + 1 = h^0(K_C-(D_s-q)) + 1 = g-s+2$. Therefore, from  Theorem \ref{CF5.8}, the relevant degeneracy locus 
$$\widetilde{\mathcal W}_1 \:= \left\{ w \in \ext^1 (K_C-(D-q), N(q)) \; | \; {\rm corank}(\partial_w) \geq 1\right\} \subseteq \ext^1 (K_C-(D-q), N(q)) \simeq \ext^1(K_C-D_s, N) $$is not empty, 
irreducible, of the expected codimension $\widetilde{c(1)} := \widetilde{\ell} - \widetilde{r} +1$ and $w \in \widetilde{\mathcal W}_1$ general is such that ${\rm corank}(\partial_w) = 1$. 

Since both $\widetilde{\mathcal W}_1, \mathcal W_1$ are contained in $\ext^1(K_C-D_s, N)$ and $\widetilde{c(1)} > c(1)$, it follows that $$\dim(\widetilde{\mathcal W}_1) < \dim(\mathcal W_1).$$From the generality assumption of $u \in \mathcal W_1$, the point $v$ defined by  $\mathcal F_v = \mathcal F_u \otimes \mathcal O_C(q)$ as in \eqref{eq:6gencogl} is such that $v \notin \widetilde{\mathcal W}_1$, which implies therefore that the corresponding coboundary map $\partial_v$ is surjective. This gives therefore $1 = h^1(K_C-(D-q)) = h^1(\Ff_v) = h^1(\Ff_u(q))$, as desired. 
\end{proof}

\begin{proposition}\label{thm:forseutile} A general $[\mathcal F] \in  \mathcal B_{\rm reg,2}$ as above cannot have either 
an effective quotient line bundle $L$ with $h^1(L)=2$ or an effective quotient line bundle $L$ with $h^1(L) = 1$ and of degree $\delta < 2g-3$.
\end{proposition}

\begin{proof} We start by proving that speciality-$2$ line bundle quotients of $\Ff$ are avoided. Assume by contradiction that a general $[\mathcal F]\in  \mathcal B_{\rm reg,2}$ has an effective quotient line bundle $L$ with $h^1(L)=2$; set $L:=K_C-E$, where $h^0(E)=2$. Notice further that, from Lemma \ref{lem:i=1.2}, $[\mathcal F]\in  \mathcal B_{\rm reg,2}$ general satisfies all the assumptions as in Lemma \ref{reg_seq_1}, with $D_s = p \in C$ general. Then, one has the following commutative diagram:
\begin{equation*}
\begin{array}{ccccccccccccccccccccccc}
&&&&0&&&&\\[1ex]
&&&&\downarrow&&&&\\[1ex]
&&&& N-p & = & N-p & &  \\[1ex]
&&&& \downarrow \tilde{\alpha} && \downarrow {\beta} & \\[1ex]
0&\lra & N_{L} & \rightarrow & \mathcal F &  \stackrel{\phi}{\rightarrow} &
L=K_C-E &\rightarrow & 0 \\[1ex]
&& &&\downarrow &&  &&\\[1ex]
& &  &  & K_C && && \\[1ex]
&& &&\downarrow&&&&\\[1ex]
&&&&0&&&&
\end{array}
\end{equation*}where the vertical exact sequence comes from Lemma \ref{reg_seq_1}, whereas $N_L$ denotes the 
kernel of the surjection $\mathcal F\to\!\!\to K_C-E$ and where $\beta$ is the composition map $$\beta: N-p \to \mathcal F \to K_C-E.$$For determinantal reason, one has $N_L \simeq N + E -p$. We claim that $\beta$ is not the zero-map: otherwise, there would exist an injective line-bundle map $N-p\hookrightarrow N_{L} \simeq N+E-p$, so the composition $N-p\to N_{L}\to \mathcal F$ would not be a bundle homomorphism, which gives a contradiction. Hence $0 \neq \beta \in {\rm Hom}(N-p, K_C-E)$ and, considering the maps:
\[
{\small
{\rm Hom}(N-p, \mathcal F)\stackrel{\psi\rq{}}{\to} {\rm Hom}(N-p, K_C-E)\simeq H^0(K_C-N-E+p) \stackrel{\psi}{\to} {\rm Ext}^1(N-p, N_{L})\simeq H^1(N_{L}-N+p),
}
\]one gets $\psi(\beta) = 0$, since $\beta=\psi\rq{}(\tilde{\alpha})$. On the other hand, since $h^1(\mathcal F)=h^1(L)=2$, the coboundary map $H^0(K_C-E)\to H^1(N_{L})$ is surjective. Hence the dual homomorphism 
$$H^1(N_{L})^{\vee}\simeq H^0(K_C-N_{L})\simeq H^0(K_C-N-E+p)\stackrel{\psi}{\to} H^0(K_C-E)^{\vee}\simeq H^1(E)\simeq H^1(N_{L}-N+p)$$ is
injective. Hence $\psi(\beta)\neq 0$, which gives a contradiction.

Now assume by contradiction that a general $[\mathcal F]\in  \mathcal B_{\rm reg,2}$ has an effective quotient line bundle $L$ 
with $h^1(L)=1$ and $\deg(L) = \delta < 2g-3$. Thus, $K_C-L = D_{2g-2-\delta}$ is an effective divisor of degree $2g-2-\delta \geq 2$ and $L = K_C - D_{2g-2-\delta}$. Let $N'$ denote the kernel line bundle of the surjection $\Ff \to\!\!\to K_C - D_{2g-2-\delta}$. Since $\Ff$ fits also in \eqref{degree0}, then $N' \simeq N+D-p$ which is general of $\deg(N') = d - \delta$, as $N$ is general of $\deg(N) = d - 2g+3$.

If $d - \delta \geq g-1$, then $h^1(N') = 0$ and this is a contradiction since $h^1(\Ff) = 2$ and $h^1(K_C - D_{2g-2-\delta}) = 1$. If otherwise $d - \delta \leq g-2$, we take into account parameters for this construction. We first take into account $g$ parameters for $N'$ and $2g-2-\delta$ parameters for $K_C - D_{2g-2-\delta}$. Taking notation as in Theorem \ref{CF5.8}, we now consider $m:= \dim (\ext^1(K_C - D_{2g-2-\delta}, N')) = h^0(2K_C - N'- D_{2g-2-\delta})$; notice that $$\deg(2K_C - N'- D_{2g-2-\delta}) = 2g-2 + 2 \delta - d,$$where $2 \delta - d \geq s(\Ff) \geq 1$, by stability; thus the line bundle 
$2K_C - N'- D_{2g-2-\delta}$ is non-special so \linebreak $m =  \dim (\ext^1(K_C - D_{2g-2-\delta}, N'))= g-1 + 2 \delta -2$.  

Now $l := h^0(K_C- D_{2g-2-\delta}) = \delta - g + 2$, as $K_C- D_{2g-2-\delta} $ is of speciality $1$, whereas $r:= h^1(N') = g-1+ \delta - d$, by generality of $N'$ of degree $d-\delta \leq g-2$. From the bounds on $d$ and the fact that $d-\delta \leq g-2$, it follows that $l \geq r$ and $m \geq l+1$. Therefore, by Theorem \ref{CF5.8}-(ii), 
$\mathcal W_1 \subsetneq \ext^1(K_C - D_{2g-2-\delta}, N')$ is irreducible, of dimension 
$\dim(\mathcal W_1) = m - (l -r +1) = 3g-5+ 2 \delta - d$. Regardless stability, in this way we get as in the previous proofs an irreducible scheme 
$\widehat{\mathcal W}_1^{Tot}$ of dimension 
$$\dim(\widehat{\mathcal W}_1^{Tot}) = g + (2g-2-\delta) + (3g-6+ 2\delta - 2d) = 
6g - 8 + \delta - 2d$$and the latter is strictly less than $\rho_d^{k_2} = 8g-11-2d$, since 
$\delta < 2g-3$, and we are done also in this case. 
\end{proof}

As a direct consequence of the proof of Proposition \ref{thm:forseutile}, one has:

\begin{corollary}\label{cor:aiuto12feb1634} For $p \in C$ general and $[\Ff] \in \mathcal B_{\rm reg,2}$ general as above, the line bundle $\omega_C(-p)$ is of minimal degree among special and effective quotient line bundles of $\Ff$, even if it is not of minimal degree among all possible quotients of $\Ff$. Moreover $\Ff$ is {\em rigidly specially presented} (rsp) via the extension \eqref{degree0}. 

On the contrary, $\Ff$ is {\em not rigidly specially presented} (not rsp) via canonical quotient line bundle, i.e. via extensions as in Lemma \ref{reg_seq_1}.   
\end{corollary}

\begin{proof} The fact that  $\omega_C(-p)$ is of minimal degree among special and effective quotient line bundles of $\Ff$ follows from \eqref{degree0} and from Theorem \ref{thm:forseutile}. The fact that it is not of minimal degree among all possible quotients comes from Claim \ref{lem:linearly_isolated}. 

The fact that $\Ff$ general is rigidly specially presented via \eqref{degree0} is a consequence of Lemma \ref{lem:pigenfin} and of the density of $\im(\pi)$ in $\mathcal B_{\rm reg,2}$. The fact that $\Ff$ general admits also canonical quotients comes from Lemma \ref{reg_seq_1}, as the bundle $\Ff$ satisfies all the assumptions therein. Therefore, one can construct the relevant $\widehat{\mathcal W}_1^{Tot}$ 
as in Lemma \ref{lem:i=1.2} but using insetad the canonical quotient. 

The associated dominant, rational map $\pi$ onto $\mathcal B_{\rm reg,2}$ turns out to have general fiber isomorphic to $\mathbb P^1$ being identified, as e.g. in Claim \ref{cl:20feb816}, to the linear system $|\mathcal O_{F_u}(\Gamma)|$ on the surface scroll $F_u$. 
\end{proof}

\bigskip

\subsection{\bf Superabundant components in speciality 2}\label{ss:sup2} It remains to study possible irreducible components $\mathcal B$ of $B_d^{k_2} \cap U^s_C(d)$ for which $[\mathcal F]\in \mathcal B $ general is of {\em first type} in the sense of Definition \ref{def:fstype}, namely which fits into an exact sequence as in \eqref{degree}, with $L = \omega_C(-D)$ s.t. $h^1(\mathcal F) =h^1(L)=2$. 

As occurred for the {\em regular} component $B_{\rm reg,2}$ studied in \S\,\ref{ss:reg2}, some values of $d$ have been already partially considered in \cite{CFK,CFK2}. For reader's convenience and in order to present a complete picture of the whole range of interest, i.e. $2g-2 \leq d \leq 4g-4$ as in \eqref{eq:congd}, we will go brief when reminding the main results in \cite{CFK,CFK2}.

\bigskip

\bigskip

\noindent
$\boxed{4g-4-2\nu \leq d \leq 4g-4}$ From \S\;\ref{ss:reg2}, in this range there is no component $\mathcal B \subset B_d^{k_2} \cap U^s_C(d)$ for which $[\mathcal F]\in \mathcal B $ general can be s.t. $h^1(\mathcal F) =h^1(L)=2$ with $L = \omega_C(-D)$ effective, as it follows either from Proposition \ref{prop:caso1spec2} (from which $B_d^{k_2}  \cap U_C^s(d) = \emptyset$ when $4g-6 \leq d \leq 4g-4$) or from 
Proposition \ref{prop:caso2spec2} (from which $B_d^{k_2}  \cap U_C^s(d)$ is irreducible, for any $3 \leq \nu < \lfloor \frac{g+3}{2}\rfloor$, consisting only of the component $\mathcal B_{\rm reg,2}$ already studied in \S\;\ref{ss:reg2}).

\bigskip

\noindent
$\boxed{3g-3 \leq d \leq 4g-5-2\nu}$ These cases have already been studied in \cite[Theorem\;1.5-(ii)]{CFK2}, and they occur when the gonality $\nu$ is such that $3 \leq \nu \leq \frac{g-2}{2}$, in particular $g \geq 8$ (indeed, for $\frac{g-2}{2} < \nu < 
\lfloor \frac{g+3}{2} \rfloor $, one has $3g-3 >  4g-5-2\nu$ so we are in the previous case, where $B_d^{k_2}  \cap U_C^s(d)$ is either empty or it is irreducible, consisting only of the regular component $\mathcal B_{\rm reg,2}$). In \cite[Theorem\;1.5-(ii)]{CFK2} we precisely proved the following:

\begin{proposition}\label{prop:caso1sup2} (cf. \cite[Theorem\;1.5-(ii)]{CFK2}) Let $g \geq 8$, $3 \leq \nu \leq \frac{g-2}{2}$ and $3g-3 \leq d \leq 4g-5-2\nu$ be integers and let $C$ be a general $\nu$-gonal curve of genus $g$. Then $B_d^{k_2}  \cap U_C^s(d)$ contains an extra irreducible component, denoted by $\mathcal B_{\rm sup,2}$, which  is generically smooth, of dimension $6g-6- d- 2\nu > \rho_d^{k_2}$, i.e. $\mathcal B_{\rm sup,2}$ is 
{\em superabundant}. Moreover, it is  {\em ruled} and a general element $[\mathcal F] \in \mathcal B_{\rm sup,2}$ corresponds to a rank two vector bundle of degree $d$ and speciality $2$, fitting in the exact sequence
\begin{equation}\label{eq:aiuto12feb2044}
 0\to N \to \mathcal F\to   \omega_C\otimes A^{\vee} \to 0,  
\end{equation}where $A \in {\rm Pic}^{\nu}(C)$ is the unique line bundle on $C$ such that $|A| = g^1_{\nu}$ and where $N \in {\rm Pic}^{d-2g+2+\nu}(C)$ is general of its degree. In particular, $ s(\Ff) = 4g-4-d-2\nu$ and $\omega_C\otimes A^{\vee}$ is of minimal degree among quotient line bundles of $\Ff$.

Moreover, $\mathcal B_{\rm sup, 2}$ is the only component of $B_{d}^{k_2}\cap U_C^s(d)$ for which the general bundle $\mathcal F$ is such that $h^1(\mathcal F) = h^1(L) =2$, for effective and special quotient line bundles $L$ as in \eqref{degree}.
\end{proposition}

\begin{remark}\label{rem:erroreColl} {\normalfont  To be more precise,  in \cite{CFK2} we studied the cases $3g-5 \leq d \leq 4g-5-2\nu$; nevertheless, as mentioned in the Introduction, in \cite[Theorem 1.5-(ii)]{CFK2} it was wrongly stated that for $d=3g-4, 3g-5$ the component $\mathcal B_{\rm sup,2}$ therein was even {\em non-reduced}. This is due to tricky computations of kernels of suitable Petri/multiplication maps in 
\cite[Thm.\;3.12-(ii)]{CFK2}. We will fix here this mistake (cf. Proposition \ref{lem:i=2.1} and Remark \ref{rem:corrColl}-(a) below) by reproving the cases $d= 3g-4, 3g-5$ in the brand new part of the paper, whereas we briefly remind here the correct part of \cite[Theorem\;1.5-(ii)]{CFK2}, i.e. what was proved for $3g-3 \leq d \leq 4g-5-2\nu$. 
} 
\end{remark}

\begin{proof} [Sketch of the proof of Proposition \ref{prop:caso1sup2}] The strategy used in \cite[\S\;3.3]{CFK2} to study cases in the range \linebreak $3g-3 \leq d \leq 4g-5-2\nu$ was as follows: first, in \cite[Lemma\,3.10]{CFK2} we proved that if $\mathcal F$ is a rank-two vector bundle arising as a general extension in ${\rm Ext}^1(K_C-A, N)$, where $N$ any line bundle in ${\rm Pic}^{d-2g+2 + \nu}(C)$, then $\mathcal F$ is stable with $s(\mathcal F)= 4g-4-2\nu-d$, i.e. $K_C-A$ is a minimal quotient line bundle of  $\mathcal F$. Moreover, if $N$ is non-special, one obviously has $h^1(\mathcal F) = h^1(K_C-A) = 2$.

Next, to construct the component $\mathcal B_{\rm sup,2}$, in \cite[Thm.\;1.5-(ii)]{CFK2} we proved the existence of a vector bundle $\mathcal E_{d,\nu}$ on a suitable open, dense subset $S \subseteq \pic^{d-2g+2+\nu}(C)$,  whose rank is  exactly $m= \dim(\ext^1(K_C-A,N))= 5g-5- 2 \nu -d $ as in \eqref{eq:m} as follows: let $\mathcal N\to\pic^{d-2g+2+\nu}(C)\times C$ be  a Poincar$\acute{e}$ line bundle so we have the following diagram:
\begin{center}
  \begin{picture}(300,100)
    \put(58,13){$\pic^{d-2g+2+\nu}(C)$}
    \put(200,15){$C$}
    \put(115,55){$\pic^{d-2g+2+\nu}(C)\times C$}
    \put(145,50){\vector(-3,-2){40}}
    \put(160,50){\vector(3,-2){40}}
    \put(150,90){$\mathcal N$}
     \put(153,85){\vector(0,-1){15}}
    \put(112,35){$p_1$}
    \put(187,35){$p_2$}
      \put(240,50){$K_C-A$}
        \put(245,45){\vector(-3,-2){30}}
  \end{picture}
\end{center}Set $\mathcal E_{d,\nu}:={R^1p_1}_*(\mathcal N\otimes p_2^*(A-K_C))$; by \cite[pp.\;166-167]{ACGH}, $\mathcal E_{d,\nu}$ is the aforementioned vector bundle on a suitable open, dense subset $S \subseteq \pic^{d-2g+2+\nu}(C)$ of 
rank $m= \dim{\ext^1(K_C-A,N)}= 5g-5- 2 \nu -d $ as in \eqref{eq:m}, since $K_C-A \ncong N$.  

Taking the associated projective bundle $\mathbb P(\mathcal E_{d,\nu})\to S$ (consisting of the family of projective spaces $\mathbb P\left(\ext^1(K_C-A,N)\right)$'s varying as  $N$ varies in $S$) we got$$ \dim (\mathbb P(\mathcal E_{d,\nu})) 
=g+ (5g-5- 2 \nu -d) -1=6g-6-2 \nu  -d.$$We proved stability for the general bundle arising from the general pair 
$(N, u) \in \mathbb P(\mathcal E_{d,\nu})$ and, using the fact that such a bundle has speciality $2$, we deduce the existence of a natural (rational) modular map
 \begin{eqnarray*}
 &\mathbb P(\mathcal E_{d,\nu})\stackrel{\pi_{d,\nu}}{\dashrightarrow} &U_C(d) \\
 &(N, [u])\to &\mathcal F_u;
 \end{eqnarray*} for which $ \im (\pi_{d,\nu})\subseteq B^{k_2}_d \cap  U^s_C(d)$. We then proved that 
$\pi_{d,\nu}$ is birational onto its image (cf. \cite[Claim\;3.11]{CFK2}); this in particular showed that the closure 
of $\im (\pi_{d,\nu})$ in $U_C(d)$, denoted by $\mathcal B_{\rm sup,2}$, is {\em ruled},  being birational to $ \mathbb P(\mathcal E_{d,\nu})$, and such that $\dim (\mathcal B_{\rm sup,2}) = 6g - 6 - 2\nu - d$. 

The content of \cite[Thm.\;3.12-(i)]{CFK2} was that, for $3g-3 \leq d \leq 4g-5-2\nu$, one has $\dim (\mathcal B_{\rm sup,2}) > \rho_d^{k_2}$ and also that $\mathcal B_{\rm sup,2}$ is generically smooth; the latter has been proved via a careful study of suitable Petri/multiplication maps. At last, we also proved that $\mathcal B_{\rm sup, 2}$ is the only component of $B_{d}^{k_2}\cap U_C^s(d)$ for which the general bundle $\mathcal F$ is such that $h^1(\mathcal F) = h^1(L) =2$, for quotient line bundles $L$ as in \eqref{degree} (cf. \cite[Lemma\;3.14]{CFK2}). 
\end{proof}

\bigskip

\noindent
$\boxed{2g-2 \leq d \leq 3g-4}$ This range is brand new; in these cases $d$ can be also written as $d = 3g-4-j$, with $0 \leq j \leq g-2$. From the cases previously studied, notice that $3g-4-j < 4g-5 - 2\nu$ iff $\nu < \frac{g-1+j}{2}$. Thus, according to the choice of $d$, equivalently of the integer $j$, if $\frac{g-1+j}{2} \leq \nu < \lfloor \frac{g+3}{2} \rfloor$, these cases have been already discussed in the previous case. We will prove here the following result.

\begin{proposition}\label{lem:i=2.1} Let  $g \geq 8$ and $3 \leq \nu \leq \frac{g-2}{2}$ be integers and let $C$ be a general $\nu$-gonal curve of genus $g$. Then: 
\begin{itemize} 
\item[(i)] for any $2g-5+2\nu \leq d \leq 3g-4$,   $B_{d}^{k_2}\cap U_C^s(d)$ contains an extra-component, different from $\mathcal B_{\rm reg,2}$ as in Proposition \ref{thm:i=1.3}, denoted by $\mathcal B_{\rm sup,2}$, which is {\em ruled}, {\em generically smooth} and {\em superabundant}, being of dimension $\dim(\mathcal B_{\rm sup,2}) = 6g-6- d-2\nu > \rho_d^{k_2}$. Furthermore, $[\Ff] \in \mathcal B_{\rm sup,2}$ general is rigidly presented (rp) as an extension \eqref{eq:aiuto12feb2044}, where $A \in {\rm Pic}^{\nu}(C)$ is the unique line bundle on $C$ such that $|A|= g^1_{\nu}$ and $N \in {\rm Pic}^{d-2g+2+\nu}(C)$ is general;

\item[(ii)] for $d= 2g-5+2\nu$, $B_{d}^{k_2}\cap U_C^s(d)$ contains an extra-component, different from $\mathcal B_{\rm reg,2}$ as in Proposition \ref{thm:i=1.3}, denoted as above by $\mathcal B_{\rm sup,2}$, which is {\em ruled} and {\em regular}. Furthermore, as above,  $[\Ff] \in \mathcal B_{\rm sup,2}$ general is rigidly presented (rp) as an extension \eqref{eq:aiuto12feb2044}, where $A \in {\rm Pic}^{\nu}(C)$ is the unique line bundle such on $C$ that $|A|= g^1_{\nu}$ and $N \in {\rm Pic}^{d-2g+2+\nu}(C)$ is general;

\item[(iii)]  for $2g-2 \leq d \leq 2g-6+2\nu$, $B^{k_2}_d \cap U^s_C(d)$ is irreducible, consisting only of  $\mathcal B_{\rm reg,2}$ as in Proposition \ref{thm:i=1.3}.

\end{itemize}
\end{proposition}

\smallskip

To prove Proposition \ref{lem:i=2.1}, we start with the following more general result. 

\begin{lemma}\label{lem:i=2.2} Let $g \geq 8$, $2g-2 \leq d \leq 3g-4$ and $3 \leq \nu \leq \frac{g-2}{2}$ be integers and let $C$ be a general $\nu$-gonal curve. Let $\Ff = \mathcal F_u$ be  a degree-$d$, rank-two vector bundle, arising as a general extension  $u \in {\rm Ext}^1(K_C-A, N)$, where $A \in \pic^{\nu}(C)$ is the unique line bundle on $C$ such that $|A| = g^1_{\nu}$ and $N \in {\rm Pic}^{d-2g+2 + \nu}(C)$ is general. Then  $\mathcal F_u$ is stable with $h^1(\mathcal F_u) = h^1(K_C-A) = 2$, i.e. $\mathcal F_u$ is of {\em first type} in the sense of Definition \ref{def:fstype}. 
\end{lemma}

\begin{proof} Notice that, from the bounds on $d$ and $\nu$, one has $\deg(N) = d-2g+2 + \nu < 2g-2-\nu = \deg(K_C-A) $; thus 
$N \ncong K_C-A$. Setting $d =3g-4-j$, for $0\leq j \leq g-2$, and $\delta:= \deg(K_C-A) = 2g-2 - \nu$, 
by \eqref{eq:m} one gets $m:=\dim (\ext^1(K_C-A,N))=2g-2\nu+1+j$, which certainly is positive by the bounds on $\nu$. 
Let $u \in \ext^1(K_C-A,N)$ be general and let $\Ff_u$ be the corresponding rank-two vector bundle.  Then  $\mathcal F_u$ fits into the exact sequence \eqref{eq:aiuto12feb2044}. 

As in Theorem \ref{CF5.8}, set $l := h^0(K_C-A) = g+1-\nu$, $m = \dim (\ext^1(K_C-A,N))=2g-2\nu+1+j$ and $r:=h^1(N)$. Since $N$ is general of degree $\deg(N) = g-2+\nu-j$, if $\nu-j \leq 0$ one has 
$h^1(N) = - \chi(N) = 1+j-\nu$ whereas, for $\nu - j \geq 1$, $h^1(N) = 0$; therefore $r = {\rm max} \{0,\;1+j-\nu\}$. Since $0 \leq j \leq g-2$, then $l \geq r$  and $m \geq l+1$ hold true. Hence we can apply Theorem \ref{CF5.8} to deduce that, for general $u \in {\rm Ext}^1(K_C-A, N)$, the corresponding coboundary map is surjective, therefore one has $h^1(\mathcal F_u) = 2$.

To prove stability of $\mathcal F_u$, as $\delta = 2g-2-\nu$ and $\nu \leq \frac{g-2}{2}$, one has that 
$2\delta-d \geq 2(2g-2-\nu)-(3g-4)=g -2\nu\geq 2$ so we can apply Theorem \ref{LN}. Thus, set $\Pp := \Pp(\ext^1(K_C-A, N))$ and consider the natural morphism 
$ C \stackrel{\varphi}{\longrightarrow} X\subset\mathbb P$ given by the complete linear system $|2K_C-A-N|$, where $X := \varphi\left(C\right)$; let $\sigma$ be any integer such that  $ \sigma \equiv 2\delta-d\  {\rm{ (mod \;2) \ \ and  }} $ 
$ 2\le \sigma \le \min\{2\delta-d,g\}$ one has 
$$\dim \left(\Sec_{\frac{1}{2}(2\delta-d+\sigma-2)}(X) \right)=2\delta - d + \sigma - 3 < 2\delta - d + g-2 = \dim \left(\mathbb P(\ext^1(L,N))\right),$$as 
$\sigma \leq g$. Hence $\mathcal F_u$ is stable.  
\end{proof}The parametric construction of the locus $ \mathcal B_{\rm sup,2}$, when $N$ varies in a suitable open dense subset \linebreak $S \subseteq \pic^{d-2g+2+\nu}(C)$, is similar to that in the proof of Proposition \ref{prop:caso1sup2}: indeed, as therein one can construct a vector bundle  $\mathcal E_{d,\nu}$ over $S$, of rank $m= \dim(\ext^1(K_C-A,N))= 5g-5- 2 \nu -d $ as in \eqref{eq:m} and consider the associated projective bundle $\mathbb P(\mathcal E_{d,\nu})\to S$, which parametrizes the family of 
$\mathbb P\left(\ext^1(K_C-A,N)\right)$'s as  $N$ varies in $S$. One has$$ \dim (\mathbb P(\mathcal E_{d,\nu})) 
=g+ (5g-5- 2 \nu -d) -1=6g-6-2 \nu  -d. $$ From stability and speciality $2$ as from Lemma \ref{lem:i=2.2}, we have a natural (rational) modular map
 \begin{eqnarray*}
 &\mathbb P(\mathcal E_{d,\nu})\stackrel{\pi_{d,\nu}}{\dashrightarrow} &U_C(d) \\
 &(N, [u])\to & [\mathcal F_u],
 \end{eqnarray*} such that $ \im (\pi_{d,\nu})\subseteq B^{k_2}_d \cap  U^s_C(d)$. Once we show that $\pi_{d,\nu}$ is birational onto its image, 
we will get that the closure of $\im (\pi_{d,\nu})$ in $U_C(d)$, denoted by  $\mathcal B_{\rm sup,2}$, is {\em ruled} (being birational to $ \mathbb P(\mathcal E_{d,\nu})$) of dimension $$\dim (\mathcal B_{\rm sup,2}) = \dim (\mathbb P(\mathcal E_{d,\nu})) = 6g - 6 - 2\nu - d.$$In the same spirit of the proof of Claim \ref{lem:linearly_isolated}, we first prove some results concerning families of sections on associated surface scrolls.

\begin{lemma}\label{lem:linearly_isolated_sup} With assumptions as in Lemma \ref{lem:i=2.1}, let $N \in \pic^{d-2g+2+\nu}(C)$ be general and let $\Ff_u$ correspond to $u \in {\rm Ext}^1(K_C-A, N)$ general. Let $F_u$ denote the surface scroll associated to $\Ff_u$ and let $\Gamma \subset F_u$ be the section of $F_u$ corresponding to the 
quotient line bundle $K_C-A$ of $\Ff_u$. 

\smallskip

\noindent
(i) If  $3g-3-2\nu \leq d\leq 3g-4$, then $\Gamma$ is algebraically isolated (ai). In particular, $\Ff_u$ is rigidly presented (rp) via the  quotient line bundle $K_C-A$.

\smallskip

\noindent
(ii) If $2g-2 \leq d \leq 3g-4-2\nu$, then one has $h^0(\mathcal F_u\otimes N^{\vee})=1$. In particular, the section $\Gamma \subset F_u$ is linearly isolated on the scroll $F_u$. 
\end{lemma} 
\begin{proof} (i) Note that, if $3g-3-2\nu \leq d\leq 3g-4$, from \eqref{eq:Ciro410} and from the generality of $N$ it follows that $h^0(\mathcal N_{\Gamma/F})\simeq h^0(K_C-A-N)=0$, hence it also follows that $\Ff_u$ is rigidly presented via the quotient line bundle $K_C-A$ (recall Definition \ref{def:ass1}). 

\medskip

\noindent
(ii) The second assertion follows directly from the first and from \eqref{eq:isom2}. Therefore, we focus on proving the first assertion. 

Consider the exact sequence obtained from  \eqref{eq:aiuto12feb2044} tensored by $A$, namely:  
\[
\begin{CD}
&& 0\to &N+A&\to &\ \ \mathcal F\otimes A\ \ & \to  \ &\ \ K_C \ \ &\to 0\\
&\deg&       &d+2\nu-2g+2&& d+2\nu&& 2g-2&\\
\end{CD}
\]  Since the bundle $\mathcal F$ corresponds to a general element in $\ext^1(K_C-A, N)$, for $N$ general, and since 
$$\ext^1(K_C, N + A) \simeq \ext^1(K_C-A, N),$$being both isomorphic to $H^1(N+A-K_C)$, then also $\Ff \otimes A$ corresponds to a general element of $\ext^1(K_C, N + A)$, with $N+A$ also general of its degree since $N$ is. As $\deg (N+A) = d+2\nu +2 - 2g$ and from the bounds on $d$, one has that $2 \nu \leq \deg(N+A) \leq g-2$. Therefore, since $N+A$ is general, it is not effective and of speciality $r:= h^1(N+A) = 3g-3-d-2\nu \geq 1$.

We want to apply  Theorem \ref{CF5.8} to show that the coboundary map $\partial: H^0(K_C)\to H^1(N + A)$ associated 
to such a general extension is surjective. To do this, set  $l:=h^0(K_C)=g$, hence we get $l \geq r$ since, from the bounds on $d$, we have $h^1(N+A)\leq g-1-2\nu < g$.

Now set $m:=\dim ({\rm Ext}^1(K_C, N + A))=h^0(2K_C-N-A)$; since 
$\deg(2K_C-A-N)\geq 3g-2 > 2g-2$, then $2K_C-A-N$ is non-special so we have $m=5g-5-d-2\nu$;  thus $m=5g-5-d-2\nu \geq l+1=g+1$ certainly holds from the bounds on $d$. 

From Theorem \ref{CF5.8} it follows that, for a general vector $u \in \ext^1(K_C, N + A)$ the corresponding coboundary map $H^0(K_C)\stackrel{\partial_u}{\longrightarrow} H^1(N + A)$ is surjective, which yields $h^1(\mathcal F_u \otimes A)=h^1(C, K_C)=1$. By Serre duality, this is $h^0(\Ff_u^{\vee} (K_C-A))=1$. Since $\Ff_u$ is of rank two, with determinant  $\det(\Ff_u) =  N-A+K_C$, one has therefore $\Ff_u^{\vee} (K_C-A) \simeq \Ff_u(-N)$ and we are done by \eqref{eq:isom2}.
\end{proof}

\begin{corollary}\label{cor:pidnubir} With assumptions as in Proposition \ref{lem:i=2.1}, the map $\pi_{d,\nu}$ is birational onto its image,  whose closure $\mathcal B_{\rm sup,2}$ in $B_{d}^{k_2} \cap U^s_C(d)$ is {\em ruled}, of dimension  
$$\dim (\mathcal B_{\rm sup,2}) = 6g - 6 - 2\nu - d.$$The general point $[\Ff_u] \in \mathcal B_{\rm reg,2}$ corresponds to a rank-two, degree $d$ stable vector bundle, of speciality $2$ which is {\em rigidly presented} (rp) as an extension \eqref{eq:aiuto12feb2044}.
\end{corollary}

\begin{proof} If one shows that the map $\pi_{d,\nu}$ is birational onto its image, it then follows that $\mathcal  B_{\rm sup,2}$ is ruled and of the stated dimension, since it is birational to $\mathbb P(\mathcal E_{d,\nu})$; moreover one will deduce also that $[\Ff_u] \in \mathcal B_{\rm reg,2}$ general is rigidly presented (rp) via extension \eqref{eq:aiuto12feb2044}, since the general fiber of  $\pi_{d,\nu}$ is a singleton. 

To prove therefore that $\pi_{d,\nu}$ is birational onto its image, one observes that the quotient line bundle $K_C-A$ is fixed in the triples $(N, K_C-A, \Ff_u)$ parametrized by the projective bundle $ \mathbb P(\mathcal E_{d,\nu})$ so, similarly as in the proof of Claim \ref{lem:linearly_isolated}, from Proposition 
\ref{prop:CFliasi}-$(1)$ it is enough to show that for general $(N, K_C-A, \Ff_u) \in \mathbb P(\mathcal E_{d,\nu})$ the section $\Gamma \subset F_u$ corresponding to the 
quotient line bundle $K_C-A$ is linearly isolated (li) on $F_u$. This follows from 
Lemma \ref{lem:linearly_isolated_sup}. 
\end{proof} We finally have all the necessary ingredients to prove Proposition \ref{lem:i=2.1}. 

\begin{proof}[Proof of Proposition \ref{lem:i=2.1}] Point $(iii)$ is a consequence of the following easy observation: from Lemma \ref{lem:i=2.2}, $\im (\pi_{d,\nu}) \subseteq \mathcal B_{\rm sup,2}$ contains stable bundles thus, from Remark \ref{rem:BNloci}, any component of $B^{k_2}_d \cap U^s_C(d)$ containing $\im (\pi_{d,\nu})$ must have at least the expected dimension $\rho_d^{k_2} = 8g-11-2d$. Then one concludes by observing that, for  $2g-2 \leq d \leq 2g-6+2\nu$, one has $\dim(\mathcal B_{\rm sup,2}) = 6g-6- d-2\nu < \rho_d^{k_2}$, thus $\mathcal B_{\rm sup,2}$ cannot fill-up an irreducible component of $B^{k_2}_d \cap U^s_C(d)$. Nonetheless, from the previous section we know the existence of $\mathcal B_{\rm reg, 2}$; furthermore, from the results contained in 
\S\,\ref{ss:nocomp2} below, we will also deduce that no further component can exist. Therefore, for $2g-2 \leq d \leq 2g-6+2\nu$, we have 
$$\im (\pi_{d,\nu}) \subseteq \mathcal B_{\rm sup,2} \subsetneq \mathcal B_{\rm reg, 2}.$$The reader is referred to \S\,\ref{ss:nocomp2} below for the precise conclusion of the proof.

Concerning points $(i)$ and $(ii)$, they will be proved together. The fact that the locus 
$\mathcal B_{\rm sup,2}$ is ruled and of the stated dimension has been showed in Corollary \ref{cor:pidnubir}. Take now 
$[\Ff_u ]\in \mathcal B_{\rm sup,2}$ general, so that $\Ff_u$ arises as a 
general extension in \eqref{eq:aiuto12feb2044} with $N \in \pic^{d-2g+2+\nu}(C)$ general, and, from Lemma \ref{lem:i=2.2}, we have $h^1(\mathcal F_u)=h^1(K_C-A) =2$. 

\medskip

Consider the following exact diagram:

\begin{equation*}\label{eq1b2}
\begin{array}{ccccccccccccccccccccccc}
&&0&&0&&0&&\\[1ex]
&&\downarrow &&\downarrow&&\downarrow&&\\[1ex]
0&\lra& N + A - K_C & \rightarrow & \mathcal F_u\otimes (A-K_C)& \rightarrow & \mathcal O_C &\rightarrow & 0 \\[1ex]
&&\downarrow && \downarrow && \downarrow& \\[1ex]
0&\lra & N \otimes  \mathcal F_u^{\vee}& \rightarrow & \mathcal F_u\otimes \mathcal F_u^{\vee} & \rightarrow &
\mathcal  (K_C-A) \otimes \mathcal F^{\vee} &\rightarrow & 0 \\[1ex]
&&\downarrow &&\downarrow &&\downarrow  &&\\[1ex]
0&\lra & \mathcal O_{C}& \lra & \mathcal F_u\otimes N^{\vee}&\lra& (K_C-A)\otimes N^{\vee}&\rightarrow & 0 \\[1ex]
&&\downarrow &&\downarrow&&\downarrow&&\\[1ex]
&&0&&0&&0&&
\end{array}
\end{equation*} which arises from  \eqref{eq:aiuto12feb2044} and from its dual sequence. If we tensor the column in the middle by  $\omega_C$ and pass to cohomology, we get the injection  $H^0(\mathcal F_u\otimes A)\hookrightarrow  H^0(\omega_C\otimes \mathcal F_u\otimes \mathcal F_u^{\vee})$. 

Since $H^1(\Ff_u) \simeq H^1(K_C-A)$, from \eqref{eq:aiuto12feb2044} and from Lemma \ref{lem:i=2.2}, by Serre duality one has  \linebreak  $H^0(\omega_C \otimes \Ff_u^{\vee}) \simeq H^0(A)$, hence the kernel of the Petri map $\mu_{\Ff_u}$ as in \eqref{eq:petrimap} coincides 
with the kernel of the multiplication map 
$$H^0(\mathcal F_u)\otimes H^0(A)\stackrel{\alpha}{\to}  H^0(\mathcal F_u\otimes A) \subseteq H^0(\omega_C\otimes \mathcal F_u\otimes \mathcal F_u^{\vee}).$$By the 
base-point-free pencil trick, we have $\ker(\alpha) \simeq H^0(\mathcal F_u\otimes A^{\vee})$. Now consider the exact sequence
\[
0\to N - A\to \mathcal F_u\otimes A^{\vee}\to K_C - 2A\to 0, 
\]obtained from \eqref{eq:aiuto12feb2044} tensored by $A^{\vee}$. Since $\mathcal F_u\in \ext^1(K_C-A, N)$ is
general, then also \linebreak $\mathcal F_u\otimes A^{\vee}\in \ext^1(K_C - 2A, N - A)$ is general, as the two extension spaces are isomorphic.

In notation of Theorem \ref{CF5.8}, set  $l:=h^0(K_C - 2A)$. Since $g \geq 2\nu -2$, from Proposition \ref{Ballico}, one has $3 = h^0(2A) = h^1(K_C-2A)$ so, by Riemann-Roch theorem, one has $l=g-2\nu+2$. Now set $r:= h^1(N-A)$; note that $\deg(N-A)=d-2g+2$ which, from the bounds 
$2g-5+2\nu \leq d \leq 3g-4$, gives $2\nu-3 \leq \deg(N-A) \leq g-2$. Thus, by the generality of $N$, 
$N-A$ is non-effective and of speciality $H^1(N-A)=3g-d-3 >0$. 
By the assumption of $d\geq 2g-5+2\nu$, we have in particular $l=g-2\nu+2 \geq r=3g-d-3$. 

Finally, set $m:=\dim ({\rm Ext}^1(K_C - 2A,N - A))=h^0(2K_C-A_N)$; we want to prove that  $m\ge l+1$. 
Since $\deg(2K_C-A_N)> 2g-2$, we have $m=5g-5-d-2\nu$ and it gives 
$m=5g-5-d-2\nu > l+1=g-2\nu+3$ since $d\leq 3g-4$. Thus, from Theorem \ref{CF5.8}, the associated coboundary map $\partial_u: H^0(K_C - 2A)\to H^1(N - A)$ is surjective, so 
$H^1(\Ff_u\otimes A^{\vee})\simeq H^1(K_C\otimes 2A^{\vee})$ which gives $H^0(\Ff_u\otimes A^{\vee})=d-2\nu-2g+5$. Thus 
\begin{eqnarray*}
\dim T_{[\mathcal F_u]}(B^{k_2}_d \cap U_C^s(2,d))&=&4g-3-2(d-2g+4)+(g + 2 -2 \nu) + (d-3g+3)\\
&=&6g-6-2 \nu -d = \dim (\mathcal B_{\rm sup,2}).
\end{eqnarray*} To complete the proof, it suffices to observe that  since $d \geq 2g-5+2\nu$, then $$\rho_d^{k_2} = 8g-11-2d \leq 6g-6-2\nu -d,$$where equality holds if and only if $d=2g-5+2\nu$.
\end{proof}

\begin{remark} \label{rem:corrColl} {\normalfont (a) Notice that the proof of Proposition \ref{lem:i=2.1}-(i) above for $d = 3g-4, \; 3g-5$ fixes \cite[Theorem 3.12-(ii)]{CFK2}, where it was wrongly stated that $\mathcal B_{\rm sup,2}$ was non-reduced. 

\noindent
(b) When $d = 2g-5+2\nu$, the components $\mathcal B_{\rm sup,2}$ and $ \mathcal B_{\rm reg,2}$, even if of the same dimension (because both are {\em regular}), are distinct since their general points have different {\em presentations}, respectively \eqref{degree0} and  \eqref{eq:aiuto12feb2044}. 
}
\end{remark}

\bigskip

\subsection{\bf No other components in speciality 2}\label{ss:nocomp2}  In this section, we will show the following: 

\begin{proposition}\label{thm:noothercompsspec2}  Let $C$ be a general $\nu$-gonal curve of genus $g$ and let $2g-2 \leq d \leq 4g-4$ be an integer. 

Then $B_d^{k_2}\cap U^s_C(d)$ is either empty or it contains no other components than $ \mathcal B_{\rm reg,2}$ and $\mathcal B_{\rm sup,2}$ constructed in Sections \ref{ss:reg2} and \ref{ss:sup2}. 
\end{proposition}

\begin{proof} For $3g-3 \leq d \le 4g-4$, Proposition \ref{thm:noothercompsspec2} follows from what discussed in the previous sections. Namely, for $4g-6 \leq d \leq 4g-4$, we have $B_d^{k_2}\cap U^s_C(d) = \emptyset$, as proved in Proposition \ref{prop:caso1spec2} whereas, for $4g-4-2\nu \leq d \leq 4g-7$, we have $B_d^{k_2}\cap U^s_C(d)= \mathcal B_{\rm reg,2}$, as proved in Proposition \ref{prop:caso2spec2}. For $3g-3 \leq d \leq 4g-5-2\nu$, the statement of Proposition \ref{thm:noothercompsspec2} has been already proved in \cite[Prop.\,3.8,\,Lemma\,3.14]{CFK2} where the interested reader is referred. 

For $2g-2 \leq d \leq 3g-4$, what stated will be a direct consequence of the next Propositions \ref{lem:i=1.1} and \ref{lem:21Feb1632}.
\end{proof}

\smallskip

Let us assume therefore $2g-2 \leq d \leq 3g-4$ and that $\mathcal B \subset B_d^{k_2}\cap U^s_C(d)$ is any irreducible component; from Remark \ref{rem:BNloci}, then $\dim (B)\geq \rho_d^{k_2} = 8g-11-2d$. Moreover, Lemma \ref{specialityhigh} gives that $h^1(\Ff) = 2$ for $[\Ff] \in \mathcal B$ general. Finally, from Lemma \ref{lem:1e2note} and from stability, such a bundle $\mathcal F$ must fit in an exact sequence of the form \eqref{degree}, where $L$ is a line bundle, of a suitable degree $\delta > \frac{d}{2}$, which is effective and special, in particular $1 \leq h^1(L) \leq 2$. We first have the following:

\begin{proposition}\label{lem:i=1.1} Let $C$ be a general $\nu$-gonal curve and let $2g-2 \leq d \leq 3g-4$ be an integer. Assume that $\mathcal B$ is any irreducible component of  $B^{k_2}_d \cap U^s_C(d)$ such that $[\mathcal F] \in \mathcal B$ general fits in an exact sequence \eqref{degree}, with $h^1(\mathcal F)=2$ and $h^1(L)=1$. 

Then, $\mathcal F$ fits in \eqref{degree0}, for $p \in C$ general, and $\mathcal B = \mathcal B_{\rm reg,2}$.  
\end{proposition}
\begin{proof} Let $\mathcal F$  correspond to a general element in $\mathcal B$, fitting in an exact sequence as \eqref{degree}, i.e. $\Ff = \Ff_u$, for some $u \in \ext^1(L,N)$. Set $\delta := \deg(L)$,  $l:=h^0(L)$ and $r:=h^1 (N)$. Then, from the fact that $L$ has to be special and effective, cf. Lemma \ref{lem:1e2note}, we have $g \leq \delta\le 2g-2$. Hence, using that $2g-2 \leq d \leq 3g-4$ and \eqref{eq:neccond}, it follows that 
\begin{equation}\label{degn}
0 \le \deg ( N) =d-\delta \leq \frac{d}{2} \le \frac{3g-4}{2}.
\end{equation} By  \eqref{degree}, together with $h^1(\mathcal F) =2$ and $h^1(L)=1$, $N$ is  special with speciality 
$r \geq 1$, and the coboundary map $\partial_u$ corresponding to $\Ff_u$ has to be of corank one. The conditions  $h^1(L) =1$ 
and ${\rm corank} (\partial_u)=1$ respectively yield  $l=\delta-g+2$ and $l\geq r-1$. 
Now set $m:=\dim ({\rm Ext}^1(L,N))$; from \eqref{eq:m}, one has $m\geq g+2\delta-d -1$, whence $m\ge l+1$ by \eqref{degn}. 

With this set-up, an open dense subset of the component $\mathcal B$ has to be associated to  a suitable open dense subset 
$S \subseteq W^{r-1}_{2g-2 +\delta -d}(C) \times C^{(2g-2-\delta)}$ and, consequently, to an irreducible subvariety \linebreak $\widehat{\mathcal W}^{Tot}_1 \subsetneq \mathbb P(\mathcal E_d)$, endowed with a modular (rational) map:
\[
\widehat{\mathcal W}^{Tot}_1  \stackrel{\pi}{\dasharrow}  \mathcal B \subset B_{d}^{k_2}
\]whose general fiber over $S$ is $\widehat{\mathcal W}_1:= \mathbb P(\mathcal W_1)$, which is the projectivization of the (locally determinant) variety $\mathcal W_1 \subsetneq {\rm Ext}^1(L,N)$.

In other words, the component $\mathcal B$ has to be the image of $ \widehat{\mathcal W}^{Tot}_1$  via a dominant rational modular map $\pi$ as above (cf.\;\cite[\S\;6]{CF} for details). From above, one must have 
$$ \dim(\mathcal B) \leq \dim (W^{r-1}_{2g-2-d+\delta}(C))+\dim (C^{(2g-2-\delta)})+\dim (\widehat{\mathcal W}_1) = 
\dim (W^{r-1}_{2g-2-d+\delta})+2g-2-\delta+\dim (\widehat{\mathcal W}_1).$$

Observe that, from \eqref{degn}, one has $\deg (K_C-N) \le 2g-2$. To conclude the proof for $\deg (K_C-N) \le g-1$ one can refer to \cite[proof of Prop.\;3.13]{CFK}. Assume therefore $\deg (K_C-N)= g+a$ where $0\leq a\leq g-2$; thus $\deg (N) = d-\delta = g-2-a$. Since $\deg(N) = g-2-a$ and $h^1(N) = r \geq 1$, by Riemann-Roch one has $h^0(N) = g-2-a - g + 1 + r= r-a-1$. 

If $r\ge a+2$, then we have $h^0(N) \ge 1$, hence $N\in W^{r-a-2}_{g-2-a} (C)\subsetneq \text{Pic}^{g-2-a}(C)$; otherwise, if 
$r= a+1$, by $\deg (N) = g-2-a$ one deduces that $N\in \text{Pic}^{g-2-a}(C)$ is general.  
Hence we get
\begin{eqnarray*}
\dim(\mathcal B)  \leq  \begin{cases}
\ \dim(\text{Pic}^{g-2-a}(C))+(2g-2-\delta)+ m - \delta +g +r - 4\ &\text{ if } r= a+1 \\
\ \dim(W^{r-a-2}_{g-2-a}(C))+(2g-2-\delta)+ m - \delta +g +r - 4\ &\text{ if } r\ge a+2.
\end{cases}
\end{eqnarray*} Using that $\dim(\mathcal B) \geq \rho^{k_2}_d$ we obtain the conclusion of the proof by using the same arguments as \cite[page 305]{CFK2}, which are independent on the range of $a$. 
\end{proof}

As for the unicity of the component $\mathcal B_{\rm sup,2}$, we will use Lemma \ref{lem:Lar}.

\begin{proposition}\label{lem:21Feb1632} Let $2g-2 \leq d \leq 3g-4$ be an integer and let $C$ be a general $\nu$-gonal curve of genus $g$. Assume that $\mathcal B$ is any irreducible  component of $B^{k_2}_d \cap  U^s_C(d)$ such that $[\mathcal F] \in \mathcal B$ general fits in an exact sequence like \eqref{degree}, with $h^1(\mathcal F)=h^1(L)=2$.  

Then $L = K_C-A$ and $\mathcal B = \mathcal B_{\rm sup,2}$. 
\end{proposition}

\begin{proof}  Let $L \in {\rm Pic}^{\delta}(C)$ be effective, with $h^1(L)=2$. Thus, 
$(K_C - L) \in W^1_t(C)$, where $t = 2g-2-\delta$. 

Since $\delta >\frac{d}{2}$, by Lemma \ref{lem:1e2note}, one has $t < 2g-2-\frac{d}{2} \leq g-1$, as $2g-2 \leq d \leq 3g-4$. Thus, from Lemma \ref{lem:Lar}, $W^1_t(C) \subsetneq {\rm Pic}^t(C)$ and either $|K_C-L| = |A| + D_{2g-2-\delta}$, where $|A| = g^1_{\nu}$ uniquely determined on $C$ and $D_{2g-2-\delta}$ is the base-divisor of $|K_C-L|$ (which is formed by distinct points when $K_C-L$ is general in the component of type $\overline{W^{\vec{w}_{1,1}}}$) or 
$|K_C-L| = g^1_n + D_{b'}$, where $g^1_n$ is base-point-free, for some $\frac{g+2}{2}\leq n \leq 2g-2-\delta$, and $D_{b'}$ is the base-divisor of degree $b' \geq 0$ of 
$|K_C-L|$ (when $K_C-L$ is general in the component $\overline{W^{\vec{w}_{1,0}}}$, then 
$n = 2g-2-\delta$ and $b'=0$). 

Consider first the case $L = K_C -(A- D_{2g-2-\delta-\nu})$. If $2g-2-\delta -\nu = 0$, i.e. when $L = K_C-A$, we get $\mathcal B =  \mathcal B_{\rm sup,2}$; when otherwise $2g-2-\delta - \nu \geq 1$, same arguments as in the proof of \cite[Lemma\,3.14]{CFK2} show that the extensions \eqref{degree} with quotient line bundles like these can fill-up at most a proper, closed subscheme of $\mathcal B_{\rm sup,2}$. 

If otherwise $L = K_C - (g^1_n + D_{b'})$, with $b' \geq 1$, reasoning as in the proof of \cite[Lemma\,3.14]{CFK2} one shows that  extensions \eqref{degree} with quotient line bundles like these are in the boundary of the family of extensions arising with quotient line bundles where $b' =0$, namely of the form $L = K_C - g^1_{2g-2-\delta}$. Take therefore such an extension
$$0 \to N \to \Ff \to  K_C - g^1_{2g-2-\delta} \to 0$$and assume that it gives rise to a
general point $[\Ff]$ of a component $\mathcal B$. By a parameter count, 
$$\dim(\mathcal B) \leq \dim(\{N\}) + \rho(g,1,2g-2-\delta) + \dim(\ext^1(K_C - g^1_{2g-2-\delta}, N)) -1.$$Since $\dim(\{N\}) \leq g$, $\rho(g,1,2g-2-\delta) = 
2(2g-2-\delta) - g - 2 = 3g-6-2\delta$ and $$\dim(\ext^1(K_C - g^1_{2g-2-\delta}, N)) 
= h^0(2 K_C - g^1_{2g-2-\delta} - N) = g-1+2\delta -d$$as it follows from the fact that  
$\deg(2 K_C - g^1_{2g-2-\delta} - N) = 2g-2+2\delta -d > 2g-2$, because $2 \delta - d \geq s(\Ff) \geq 1$ from the stability of $\Ff$, so the line bundle $2 K_C - g^1_{2g-2-\delta} - N$ is non-special. To sum-up, we have therefore 
$$\dim(\mathcal B) \leq g + (3g-6-2\delta) + (g-2+2\delta -d) = 5g- 8-d.$$On the other hand, $\dim(\mathcal B) \geq \rho_d^{k_2} = 8g-11-2d$ and this is a contradiction, as 
$5g- 8-d < 8g-11-2d$ because $2g-2 \leq d \leq 3g-4$. Hence the proof is complete.
\end{proof} 

\begin{corollary}\label{cor:finale2} Let $g \geq 8$, $3 \leq \nu \leq \frac{g-2}{2}$ and 
$2g-2 \leq d \leq 3g-4$ be integers. Let $C$ be a general $\nu$-gonal curve of genus $g$. 
Then one has:

\smallskip

\begin{itemize}
\item[(i)] for $2g-5+2\nu \leq d \leq 3g-4$, $B_{d}^{k_2}\cap U_C^s(d) = \mathcal B_{\rm reg,2} \cup \mathcal B_{\rm sup,2}$; 
\item[(ii)] for $2g-2 \leq d \leq 2g-6 + 2\nu$, $B_{d}^{k_2}\cap U_C^s(d)= \mathcal B_{\rm reg,2}$ and the locus $\mathcal B_{\rm sup,2}$ is contained therein as a proper, closed subscheme.  
\end{itemize}
\end{corollary}

\begin{proof} Part (i) follows from Propositions \ref{thm:i=1.3}, \ref{lem:i=2.1}-(i) and (ii), \ref{lem:i=1.1} and \ref{lem:21Feb1632}. Part (ii) follows from  Propositions \ref{thm:i=1.3}, \ref{lem:i=2.1}-(iii), \ref{lem:i=1.1} and \ref{lem:21Feb1632}. 
\end{proof}

\begin{remark}\label{rem:finale2} {\normalfont Notice that the irreducibility of $B_{d}^{k_2}\cap U_C^s(d)= \mathcal B_{\rm reg,2}$ for $2g-2 \leq d \leq 2g-5 + 2\nu$ is in accordance with Teixidor I Bigas' result, i.e. Theorem \ref{TeixidorRes}-(i) above. 

Indeed, the claim therein is that $B_{d}^{k_2}\cap U_C^s(d)$ can have extra components than $\mathcal B_{\rm reg,2}$ if and only if there exists a positive integer $n$, $2n < 4g-4-d$, for which $W^1_n(C) \neq \emptyset$ and s.t. $\dim(W^1_n(C)) \geq 2g+2n-d-5$. The bounds on $d$ in particular give $2n < 2g-2$, i.e. $n < g-1$; moreover, since we are on a general $\nu$-gonal curve, one must have also $n \geq \nu$. Thus $n$ must be such that $$\nu \leq n < g-1.$$From Lemma \ref{lem:Lar}, we know that in such cases one has either 
\[
\dim (W^1_n(C)) = \left\{ 
\begin{array}{l}
n-\nu \\
 \\
2n-g-2
\end{array}
\right.
\]Using Teixidor I Bigas' reducibility condition as in Theorem \ref{TeixidorRes}, one should have therefore either $$2n-g-2 = \dim(W^1_n(C)) \geq 2g+2n-d-5,$$implying $d \geq 3g-3$,  or $$n-\nu = \dim(W^1_n(C)) \geq 2g+2n-d-5,$$giving $ d \geq 2g-5 + n + \nu \geq 2g-5 + 2 \nu$. Both of them contradict that $2g-2 \leq d \leq 2g-6 + 2\nu$. 
}
\end{remark}


\end{document}